\documentclass[a4paper,10pt]{article}

    \usepackage[utf8]{inputenc}
    \usepackage[T1]{fontenc}
    \usepackage{siunitx}
    \usepackage{colonequals}
    \usepackage[super]{nth}
    \usepackage{algorithm}
    \usepackage{amsmath}
    \usepackage{amsthm}
    \usepackage{amssymb}
    \usepackage{tikz}
    \usepackage{url}
    \usepackage{float}
    \usepackage{tabularx}
    \usepackage{pbox}
    \usepackage{tabulary}
    \usepackage{breqn}
    \usepackage{longtable}
    \usepackage{mdframed}
    \usepackage{booktabs}

    \usepackage{graphicx}
    \graphicspath{{./images/}{./}}

    \usepackage{wrapfig}
    \usepackage[font={small}]{caption}
    \usepackage{enumitem}
    \usepackage{physics}
    \usepackage{lipsum}
    \usepackage{mwe}
    \usepackage{breqn}
    \usepackage[noend]{algpseudocode}
    \usepackage{subcaption}
    \usepackage{graphicx}
    \usepackage{enumitem}
    \usepackage{bm}
    \usepackage{textcmds}

    \usepackage{apptools}
    \AtAppendix{\counterwithin{lemma}{section}}
    \AtAppendix{\counterwithin{remark}{section}}
    
    \usepackage{hyperref}
    \hypersetup{
        colorlinks,
            citecolor=black,
        filecolor=black,
        linkcolor=black,
        urlcolor=black
    } 

\newtheorem{theorem}{Theorem}
\newtheorem{lemma}{Lemma}
\newtheorem{remark}{Remark}

\newtheorem{definition}{Definition}

\newtheorem{example}{Example}
\newtheorem{alg}{Algorithm}

\DeclareMathOperator*{\argmin}{arg\,min}

\newcommand{\ra}{\rangle}
\newcommand{\la}{\langle}
\newcommand{\norms}[1]{\|#1\|_2^2}
 
\newcommand{\N}{\mathbb{N}}
\newcommand{\R}{\mathbb{R}}

\renewcommand{\P}{\mathbb{P}}

\renewcommand{\complement}[1]{\left( #1\right)^c}

\newcommand\muAOneHashing{C_1 \sqrt{C_2} \sqrt{\epsilon}  \min \left\{\frac{\log(E/ (C_2 \epsilon))}{4r + \log(1/\delta)}, \sqrt{\frac{\log(E)}{4r + \log(1/\delta)}} \right\} }
\newcommand\logDeltaI{\log(1/\delta)}
\newcommand\fourRplusLog{\left[ 4r + \logDeltaI\right]}
\newcommand{\nuOneHashing}{C_1 \sqrt{\epsilon} \min \left\{ \frac{\log(E/\epsilon)}{\log(1/\delta)}, \sqrt{\frac{\log(E)}{\log(1/\delta)}} \right\}}
 
\newcommand{\muASHashing}{\sqrt{s} C_{\nu} C_1^{-1} \muBarEpsDelta}
\newcommand{\muASHashingVariant}{\sqrt{s}\muBarEpsDelta}
\newcommand{\EUpperSHashing}{C_2^2 \epsilon^2 s[4r + \logDeltaI]^{-1} e^{C_s (C_2 \epsilon s)^{-1} \left[ 4r+\logDeltaI \right]}}
\newcommand{\muLower}{ \sqrt{r/n} + \sqrt{8 \log(n/\delta_1)/n } }
\newcommand{\muLowerNoN}{ \sqrt{r} + \sqrt{8 \log(n/\delta_1) } }
\newcommand{\nLower}{ \frac{ \left( \muLowerNoN  \right)^2}{ 
sC_{\nu}^2C_1^{-2}\bar{\mu}(\epsilon,\delta)^2 }}
\newcommand{\nLowerSVariant}{ \frac{ \left( \muLowerNoN  \right)^2}{ s\bar{\mu}(\epsilon,\delta)^2 } }
\newcommand{\mLower}{E C_2^{-2} \epsilon^{-2} \fourRplusLog}

\newcommand{\epsilonPrimeFactor}{\frac{(1-\gamma)(1-\gamma^2)}{1+2\gamma-\gamma^2}}
\renewcommand{\k}{p}
\newcommand{\set}[1]{\left\{ #1 \right\}}
\newcommand{\yPlus}{Y_+}
\newcommand{\yMinus}{Y_-}
\newcommand{\nPlus}{N_+}
\newcommand{\nMinus}{N_-}
\newcommand{\yPlusExpression}{\set{y_1 + y_2: y_1,y_2 \in Y}}
\newcommand{\yMinusExpression}{\set{y_1 - y_2: y_1,y_2 \in Y}}
\newcommand{\nPlusExpression}{\set{y_i + y_j: i,j \in \left[1,|N|\right]}}
\newcommand{\nMinusExpression}{\set{y_i - y_j: i,j  \in \left[1,|N|\right]}}
\newcommand{\union}[2]{#1\cup#2}
\newcommand{\yExpression}{\set{y_1, y_2, \dots y_{|Y|}}}
\newcommand{\probability}[1]{\P \left( #1 \right)}
\renewcommand{\complement}[1]{\left( #1 \right)^c}
\newcommand{\squareBracket}[1]{\left[ #1 \right]}

\newcommand{\nMinusOneExpression}{\set{y_i-y_j: 1\leq i <j \leq |N|}}
\newcommand{\nMinusTwoExpression}{\set{y_i - y_j: 1\leq j <i \leq |N|}}
\newcommand{\nMinusOne}{N_-^{(1)}}
\newcommand{\nMinusTwo}{N_-^{(2)}}
\newcommand{\muBarEpsDelta}{\bar{\mu}(\epsilon,\delta)}
\newcommand{\EUpperHashing}{\frac{2 e^{4r}}{\left[ 4r + \logDeltaI \right]\delta }}
\newcommand{\muHatSEpsDelta}{\hat{\mu}(s,\epsilon,\delta)}
\newcommand{\HRHTProb}{1-\delta-\delta_1 + \delta\delta_1}
\newcommand{\bracket}[1]{\left( #1\right)}

\newcommand{\constantsDescription}{problem-independent constants}
\newcommand{\whereMuBarIsDefined}{where $\muBarEpsDelta$ is defined in \eqref{A:mu-1}}
\newcommand{\theoremFiveFirstSentence}{Suppose that $\epsilon,\delta \in (0,1)$, $r \leq d \leq n, m\leq n \in \N$, $E >0$ satisfy }
\newcommand{\theoremThreeFirstSentence}{Let  $C_1, C_2, C_3, C_M, C_{\nu}, C_s >0$ be \constantsDescription. Suppose that $\epsilon,\delta \in (0, C_3)$, 
$m,\,s\in \N^+$ and $E>0$ satisfy }

\newcommand{\anSHashingMat}{an $s$-hashing matrix}
\newcommand{\anSHashingVariantMat}{an $s$-hashing variant matrix}
\newcommand{\refAlgOne}{Algorithm \ref{alg1}}
\newcommand{\solverName}{Ski-LLS}
\newcommand{\solverNameDense}{Ski-LLS-dense}
\newcommand{\solverNameSparse}{Ski-LLS-sparse}
\newcommand{\calibrationDenseCaptionSentenceOne}[1]{Runtime of #1 on dense matrices $A \in \R^{n \times d}$ from Test Set 1
with $n=50000, d=4000$ and $n=50000, d=7000$ and different values of $\gamma=m/d$} 
\newcommand{\calibrationDenseCaptionSentenceTwo}[1]{For each plot, #1 is run three times on (the same) randomly generated $A$. We see that the runtime has low variance despite the randomness in the solver} 
\newcommand{\calibrationDenseCaptionSentenceThree}[1]{We choose #1 to approximately minimize the runtime across the above plots}
\newcommand{\performanceProfileCaption}[1]{Performance profile comparison of \solverName{} with LSRN, LS\_HSL and LS\_SPQR for all matrices $A\in\R^{n\times d}$ in the Florida matrix collection with #1}

\newcommand{\calibrationSixFigures}[8]{
\begin{figure}[H]
    \centering
    \begin{minipage}{\mysize\textwidth}
        \centering
        \includegraphics[width=\textwidth]{#1} 
    \end{minipage}
    \begin{minipage}{\mysize\textwidth}
        \centering
        \includegraphics[width=\textwidth]{#2} 
    \end{minipage}
    \begin{minipage}{\mysize\textwidth}
        \centering
        \includegraphics[width=\textwidth]{#3} 
    \end{minipage}    
    \centering
    \begin{minipage}{\mysize\textwidth}
        \centering
        \includegraphics[width=\textwidth]{#4} 
    \end{minipage}
    \begin{minipage}{\mysize\textwidth}
        \centering
        \includegraphics[width=\textwidth]{#5} 
    \end{minipage}
    \begin{minipage}{\mysize\textwidth}
        \centering
        \includegraphics[width=\textwidth]{#6} 
    \end{minipage}        
    \caption{#7} \label{#8}
\end{figure}
}
\newcommand{\threeFigures}[9]{
\begin{figure}
    \begin{minipage}{\mysize\textwidth}
    \centering
        \includegraphics[width=\textwidth]{#1} 
        \caption{#2}
        \label{#3}
    \end{minipage}\hfill
    \begin{minipage}{\mysize\textwidth}
    \centering
        \includegraphics[width=\textwidth]{#4} 
        \caption{#5}
        \label{#6}
    \end{minipage}
    \begin{minipage}{\mysize\textwidth}
    \centering
        \includegraphics[width=\textwidth]{#7} 
        \caption{#8}
        \label{#9}
    \end{minipage}    
\end{figure}
}

\newcommand{\twoFigures}[6]{
    \begin{figure}
    \centering
    \begin{minipage}{\mysize\textwidth}
        \centering
        \includegraphics[width=\textwidth]{#1} 
        \caption{#2}
        \label{#3}
    \end{minipage}\hfill
    \begin{minipage}{\mysize\textwidth}
    \centering
    \includegraphics[width=\textwidth]{#4}
    \caption{#5}
    \label{#6}
    \end{minipage}
    \end{figure}
}

\newcommand{\mathInTitle}[1]{\texorpdfstring{#1}{TEXT}}

\newcommand{\singleQuote}[1]{\lq #1\rq}

\usepackage[margin=1.25in]{geometry} 

\begin{document}

\title{Hashing embeddings of optimal dimension, with applications to linear least squares}

\author{Coralia Cartis\thanks{The  order  of  the  authors  is  alphabetical;  the  third  author  is  the  primary  contributor.}  \textsuperscript{\normalfont,}\thanks{Mathematical Institute, University of Oxford, Radcliffe Observatory Quarter, Woodstock Road, Oxford, OX2 6GG, United Kingdom (\texttt{cartis@maths.ox.ac.uk}). This author's work was supported by the Alan Turing Institute for Data Science, London, UK.}
	\and
	Jan Fiala\footnotemark[1] \textsuperscript{\normalfont,}\thanks{Numerical Algorithms Group Ltd, 	30 St Giles', Oxford, OX1 3LE, United Kingdom
	(\texttt{jan.fiala@nag.co.uk}).} 
	\and
	Zhen Shao\footnotemark[1]  \textsuperscript{\normalfont,}\thanks{Mathematical Institute, University of Oxford, Radcliffe Observatory Quarter, Woodstock Road, Oxford, OX2 6GG, United Kingdom (\texttt{shaoz@maths.ox.ac.uk}). This author's work was supported by the EPSRC Centre For Doctoral Training in Industrially Focused Mathematical Modelling (EP/L015803/1) in collaboration with the Numerical Algorithms Group Ltd.} 
	}
\date{\today}

\maketitle

\begin{abstract}
The aim of this paper is two-fold: firstly, 
to present subspace embedding properties for $s$-hashing sketching matrices, with $s\geq 1$, that are optimal in the projection dimension $m$ of  the sketch, namely, $m=\mathcal{O}(d)$, where $d$ is the dimension of the subspace. A diverse set of results are presented that 
address the case when the input matrix has sufficiently low coherence
(thus removing the $\log^2 d$ factor dependence in $m$, in the low-coherence result of Bourgain et al (2015) at the expense of a smaller coherence requirement);
how this coherence changes with the number $s$ of column nonzeros   (allowing a scaling of $\sqrt{s}$ of the coherence bound), or is reduced through suitable transformations (when considering hashed- instead of subsampled- coherence reducing transformations such as randomised Hadamard). Secondly, we apply these  general hashing sketching results to the special case of Linear Least Squares (LLS), and develop \solverName{}, a generic software package for these problems, 
that builds upon and improves the Blendenpik solver on dense input and the (sequential) LSRN performance on sparse problems. In addition to the hashing sketching improvements, we add suitable linear algebra tools for rank-deficient and for sparse problems that lead Ski-LLS to outperform 
not only sketching-based routines on randomly generated input, but also  state of the art direct solver SPQR and iterative code HSL  on certain subsets of the sparse Florida matrix collection; namely, on least squares problems that are 
 significantly overdetermined, or  moderately sparse, or
 difficult.

\end{abstract}

\\


\section{Introduction}

	Over the past fifteen years, sketching techniques have proved to be useful tools for improving the computational efficiency and scalability of linear algebra and optimization techniques, such as of methods for solving least-squares, sums of functions and low-rank matrix problems  \cite{10.1561/2200000035, 10.1561/0400000060}. 
The celebrated Johnson-Lindenstrauss Lemma \cite{Johnson:1984aa} and subsequent results  
 use a carefully-chosen random  matrix $S\in\R^{m\times n}$, $m\ll n$, to sample/project the column vectors of a matrix $A\in \R^{n\times d}$,
 with $n\geq d$ to lower dimensions, while approximately preserving the length of vectors in this column space; this quality of $S$ (and of its associated distribution) is captured by the (oblivious) subspace embedding property \cite{10.1561/0400000060} (see Definition \ref{subspace_embedding_def1_statement}). Sketching has found a variety of uses, such as in the solution of linear least squares problems and low-rank matrix approximations, in subsampling data points and reducing the dimension of the parameter space in training tasks for machine learning systems and  imaging, and more. 
 
 Gaussian matrices have been shown to have optimal subspace embedding properties in terms of the size $m$ of the sketch, allowing $m=\mathcal{O}(d)$ for any given matrix $A\in \R^{n\times d}$ (with high probability).
 Due to their density, however, the computational cost of calculating the sketch $SA$ is often prohibitively expensive, potentially limiting its use in large-scale contexts. To alleviate these deficiencies, as well as an alternative to subsampling techniques, 
 sparse random matrix ensembles have been proposed in the randomized linear algebra literature by
 Clarkson and Woodruff \cite{10.1145/3019134}; namely, $s$-hashing matrices with $s$ (fixed) non-zeros per column, $s\geq 1$. Hashing matrices have been shown empirically to be almost as good as Gaussian matrices for sketching purposes \cite{10.1145/3219819.3220098}, but their theoretical embedding guarantees require at least $m = O(d^2)$ rows for the sketch $S$ \cite{Nelson:2014uu}. Furthermore, numerical evidence also shows that using $s\geq 2$ instead of $s=1$ in $S$ leads to a more accurately sketched input, especially when this input is sparse. Here, in our theoretical developments, we aim to make precise and quantify these numerical observations; namely, we show that subspace embedding properties can be shown for 
  $s$-hashing matrices with $s\geq 1$   that have the optimal dimension dependence $m=\mathcal{O}(d)$  on some classes of inputs.
 
 Sketching results have direct implications on the efficiency of  solving Linear Least Squares (LLS) problems, which can be written as the optimization problem,
 \begin{align} 
\min_{x \in \R^d}\|Ax-b\|_2^2, \label{LLS-statement}
\end{align}
where  $A \in \R^{n\times d}$ is a given data matrix that has (unknown) rank $r$,  $b\in \R^n$ is the vector of observations, and
$n \geq d \geq r$. The sketched  matrix $SA$ is used to either directly compute an approximate solution to  (\ref{LLS-statement}) or to generate a high-quality preconditioner for the iterative solution of   (\ref{LLS-statement}); the latter has been the basis of state-of-the-art randomized linear algebra codes such as Blendenpik \cite{doi:10.1137/090767911} and LSRN \cite{Meng:2014ib}, where the latter improves on the former by exploiting input sparsity, parallelization  and allowing rank-deficiency of the input matrix. Here, we propose a  generic solver called \solverName{} (SKetchIng for Linear Least Squares) that builds upon and extends the Blendenpik/(serial) LSRN frameworks and judiciously uses hashing sketching with one or more nonzeros and efficiently addresses both dense and sparse inputs; extensive numerical results are presented.

\paragraph{Existing literature: sparse random ensembles and subspace embedding properties} 
Various sparse and structured random matrices have been proposed, in an attempt to reduce the computational cost of using dense Gaussian sketching. As an obvious choice,
subsampling matrices that have one nonzero per row have been used in \cite{10.5555/1109557.1109682} but do not have good subspace embedding properties unless sampling is done in a non-uniform way using computationally expensive probabilities. 

The fast Johnson-Lindenstrauss transform (FJLT) proposed by Ailon and Chazelle \cite{10.1145/1132516.1132597} is a structured matrix, needing  $\mathcal{O} \left( nd \bracket{\log(d)+ \log \log n} \right) $ operations to apply to a given $A\in \R^{n\times d}$. 
To ensure subspace embedding properties of the sketched input,
 the sketching matrix $S$ is required to have  $m = \mathcal{O} \left( (\sqrt{d} + \sqrt{\log n})^2 \log(d) \right) $  \cite{Tropp:wr, 
 10.1145/2483699.2483701}.
 Clarkson and Woodruff \cite{10.1145/3019134} proposed the use of $1$-hashing matrices, with  one non-zero ($\pm1$)  per column in random rows. Such a random matrix takes $\mathcal{O} \left( nnz(A) \right) $ operations to be applied to $A$, where $nnz(A)$ denotes the number of nonzeros in $A$. It 
 needs $m=\Theta \left( d^2 \right) $ rows to be a subspace embedding  \cite{10.1145/2488608.2488621, Nelson:2014uu,10.1145/2488608.2488622}. Cohen \cite{10.5555/2884435.2884456} and Nelson and Nguyen  \cite{Nelson:te}  have shown that increasing number of non-zeros per column in  hashing matrices, namely, using $s$-hashing matrices with $s> 1$, leads to a reduced requirement in the number of rows $m$ of the sketching matrix.
 Bourgain et al \cite{Bourgain:2015tc} further showed that if the coherence of the input matrix -- a measure of the non-uniformity of its rows -- is sufficiently low, $1$-hashing sketching  requires fewer rows, namely $m=\mathcal{O}(d\log^2 d)$; see Table \ref{tab::m_and_mu_1_hashing}. When $s$-hashing matrices are employed, the coherence requirement can be relaxed by a factor of $\sqrt{s}$ for similar $m$ requirements
\cite{Bourgain:2015tc}.
A recent paper on tensor subspace embeddings \cite{iwen2020lower} can be particularised to   vector subspace embeddings, leading to a matrix distribution of matrices $S$ with $m = \mathcal{O} (d \log^4 d)$  that requires $\mathcal{O}(n\log n)$ operations to apply to any vector. 

In \cite{CHEN2020105639},  a `stable' $1$-hashing matrix is proposed, for which each row  has approximately the same number of non-zeros, and that has good JL embedding properties with the optimal $m = \mathcal{O}(\epsilon^{-2} \log(1/\delta)$ for such projections.
Furthermore, \cite{liu2021extending} proposed learning the positions and values of non-zero entries in $1$-hashing matrices by assuming the data comes from a fixed distribution.  Our contributions to the embedding properties of $s$-hashing matrices with $s\geq 1$ are summarized below.

\paragraph{Existing literature: efficient algorithms for LLS problems employing sketching}
Theoretical understanding of sketching properties led to the creation of novel algorithms for diverse problem classes, including LLS in \eqref{LLS-statement}. In particular, Sarlos \cite{10.1109/FOCS.2006.37} first proposed directly sketching problem 
\eqref{LLS-statement} and solving the reduced problem, while 
Rokhlin \cite{Rokhlin:2008wb} introduced the idea of using  sketching for the construction of a suitable preconditioner, to help with solving \eqref{LLS-statement} via iterative means; this has been further successfully implemented in the state of the art solvers that use sketching, Blendenpik \cite{doi:10.1137/090767911} and LSRN \cite{Meng:2014ib}. Various sketching matrices have been used in algorithms for \eqref{LLS-statement}: Blendenpik  uses a variant of FJLT; LSRN \cite{Meng:2014ib}, Gaussian sketching; Iyer \cite{10.5555/3019094.3019103} and Dahiya\cite{10.1145/3219819.3220098} experimented with $1$-hashing. Recently, \cite{Iyer:2016aa} implemented Blendenpik in a distributed computing environment,  showing the advantages of using sketching over LAPACK routines on huge-size LLS problems with dense matrices.


Large scale 
LLS problems may alternatively be solved by first-order methods, such as stochastic gradient or block coordinates, that use sketching directly in subselecting the gradient terms or components that form the search direction \cite{lacotte2020optimal,lacotte2019faster, MR4187148, MR3432148, lopes2018error}.
Sketching curvature information has also been investigated for problem 
 \eqref{LLS-statement} such as in \cite{kahale2020leastsquares}, while 
\cite{zhu2018gradientbased} proposed a gradient-based sampling method.
A nice overview of these methods can be found in  \cite{2019arXiv190912176L}. However, such first order approaches are beyond our scope here, as we are searching for high-accuracy solutions to problem \eqref{LLS-statement}, extending the Blendenpik and LSRN frameworks. 

LLS problems have been the focus of the numerical linear algebra community for several decades now; techniques abound, and we refer the interested reader to  Chapter 9 in \cite{Nocedal:2006uv} for a brief introduction, and to classical textbooks \cite{Bjorck:1996uz} for an extensive coverage. For a more recent benchmarking paper that addresses recent developments in preconditioned iterative methods, see \cite{10.1145/3014057}; we use one of the most competitive codes in \cite{10.1145/3014057} for the  benchmarking of our solver. 


\paragraph{Summary of our contributions}
In terms of theoretical contributions to {\bf subspace embedding properties} of $s$-hashing matrices, we have the following main results.
\begin{itemize}
    \item We provide a subspace embedding result for $1$-hashing matrices $S\in \R^{m\times n}$   that has the optimal choice of embedding dimension, $m=\mathcal{O}(d)$,
    when applied to a(ny) given $A\in \R^{n\times d}$ with sufficiently small coherence (see Theorem \ref{thm1}). This result improves upon 
    Bourgain et al \cite{Bourgain:2015tc}, where  $m=\mathcal{O}(d\log^2 d)$, but where a more relaxed coherence requirement is imposed; see Table \ref{tab::m_and_mu_1_hashing}.
    Our optimal dimension dependence comes at the expense of a stricter coherence requirement, reducing the problem class we can address. 
    \item We cascade our  result for $1$-hashing to $s$-hashing matrices with $s>1$, achieving an embedding result with a sketching matrix of size  $m=\mathcal{O}(d)$ when applied to a larger input matrix class, whose coherence is allowed to increase by at most a factor of $\sqrt{s}$; see Theorem \ref{thm::s-hashing}.
    \item Instead of using subsampled coherence-reduction transformations, we propose to replace the subsampled aspect with hashing, leading to novel hashed- variants of any such transformations; Figure \ref{fig::blen_motivation} gives a simple illustration of the benefits of our proposal. Then, in the case of Randomised Hadamard Transform (RHT), we show that for its $1$-hashed variant, H-RHT, a subspace embedding result holds that has optimal dimension dependence, namely, $m=\mathcal{O}(d)$, when applied to any $n\times d$ input matrix with $n>d^3$ (that is sufficiently overdetermined); see Theorem \ref{thm::HRHT} which improves corresponding results for subsampled RHT \cite{Tropp:wr}.
    \item Finally, we propose a new $s$-hashing variant that coincides with (the usual) $1$-hashing when $s=1$ but may differ when $s>1$; in particular, this variant has  at most $s$ nonzeros per column, as opposed to $s$-hashing that has precisely $s$ nonzero entries.   We show a general subspace embedding result that allows translating any subspace embedding property of $1$-hashing sketching and input matrix with coherence at most (some/any) $\mu>0$ to a similar guarantee for the $s$-hashing variant and input of coherence at most $\mu\sqrt{s}$; see Theorem \ref{1-hashing-and-s-hashing}. Our result relies on a single, intuitive proof, and is not tied to any particular proof of embedding properties in the case of $1$-hashing; thus if the latter improves, so will the guarantees for the $s$-hashing variant, immediately.
\end{itemize}
We note that the $d$ dependence in our embedding results is replaced by $r$ dependence, where $r$ is the rank of the embedded matrix $A$, as this accurately captures the true data dimension.

Regarding our {\bf algorithmic contributions}, we introduce and analyse a general framework (Algorithm \ref{alg1}) for solving \eqref{LLS-statement} that includes Blendenpik and LSRN, while allowing arbitrary choice of sketching and a wide range of rank-revealing factorizations in order to build a preconditioner using the sketched matrix $SA$. We show that a minimal residual solution of \eqref{LLS-statement} is generated by Algorithm \ref{alg1} when a rank-revealing factorization is calculated for $SA$. If the latter is replaced by a 
complete orthogonal factorization of $SA$, then the minimal norm solution of \eqref{LLS-statement} is obtained. 
Our {\bf software contribution} is \solverName{}, a C++ 
implementation of Algorithm \ref{alg1} that allows and distinguishes between dense and sparse input and makes extensive use of $s$-hashing matrices for the sketching step. 
When $A$ in \eqref{LLS-statement} is dense, \solverName{} employs the following changes/improvements to Blendenpik:  choice of sketching is the {\it $1$-hashed} 
DHT transform instead of subsampled; a randomised column pivoted QR factorization \cite{Martinsson:2017eh} of $SA$ instead of just QR, allowing rank-deficient input. The latter factorization is less computationally expensive than the SVD choice in LSRN. Similarly, in the case of sparse input $A$,
we let $S$ be an $s$-hashing matrix (with $s=2$) and use a  rank-revealing sparse QR factorization for $SA$ (SPQR \cite{10.1145/2049662.2049670}), which leads to improved performance over (serial) LSRN and other state of the art LLS solvers. In particular,
\solverName{} is more than 10 times faster than LSRN for sparse Gaussian inputs. We extensively compare \solverName{}  with LSRN, SPQR and the iterative approach in HSL, on the Florida Matrix Collection \cite{10.1145/2049662.2049663}, and find that \solverName{} is competitive on significantly-overdetermined or ill-conditioned inputs.

\paragraph{Structure of the paper} 
The necessary technical background is given in Section 2.
In Section 3, we state and prove our theorem on the coherence requirement needed to use $1$-hashing with $m=O(d)$ as a subspace embedding. In Section 4, we show how increasing the number of non-zeros per column from $1$ to $s$ relaxes the coherence requirement for $s$-hashing matrices by a factor of $\sqrt{s}$. In the same section, we also consider hashed coherence reduction transformations and show their embeddings properties. The algorithmic framework that uses sketching for rank-deficient linear least squares is introduced and analysed in Section 5. In Section 6 we introduce our solver \solverName{} and discuss its key features and implementation details. We show extensive numerical experiments for both dense and sparse matrices $A$ and demonstrate the competitiveness of our solver in Section 7.

Throughout the paper, we let $\langle\cdot,\cdot\rangle$ and $\|\cdot\|_{2}$ denote the usual Euclidean inner product and norm, 
respectively, and  $\|\cdot\|_{\infty}$, the $l_{\infty}$ norm. Also, for some $n\in \N$, $[n]=\{1,2,\ldots,n\}$. For a(ny) symmetric positive definite matrix $\overline{W}$,
we define the norm $\|x\|_{\overline{W}}:=x^T\overline{W}x$, for all $x$,  as the norm induced  by $\overline{W}$.
The notation $\Theta \left( \cdot\right)$ denotes both lower and upper bounds of the respective order. $\Omega \left( \cdot \right)$ denotes a lower bound of the respective order.

\section{Technical Background}

In this section, we review some important concepts and their properties that we then use throughout the paper. 
We employ several variants of the notion of random embeddings for finite or infinite sets, as we define next.

\subsection{Random embeddings}

We start with a very general concept of embedding a (finite or infinite) number of points; throughout, we let $\epsilon \in (0,1)$ be the user-chosen/arbitrary error
tolerance in the embeddings 
\enlargethispage{4ex}
and $n, k\in \N$. 

\begin{definition}[Generalised JL\footnote{Note that `JL' stands for Johnson-Lindenstrauss, recalling their pioneering lemma \cite{Johnson:1984aa}.} embedding \cite{10.1561/0400000060}]
\label{genJL}
A generalised $\epsilon$-JL embedding for a set $Y\subseteq \R^n$
is a matrix $S \in\R^{m\times n}$ such that
\begin{equation}\label{JL-plus}
-\epsilon \|y_i\|_2\cdot \|y_j\|_2 \leq \langle Sy_i, Sy_j \rangle - \langle y_i, y_j \rangle \leq \epsilon \|y_i\|_2 \cdot \|y_j\|_2, \quad \text{for all }\,\, y_i, y_j \in Y.
\end{equation}
\end{definition}

If we let $y_i=y_j$ in (\ref{JL-plus}),  we recover the common notion of an $\epsilon$-JL embedding, that approximately preserves the length of vectors in a given set. 

\begin{definition}[JL embedding \cite{10.1561/0400000060}]
\label{def::JL_embedding}
An $\epsilon$-JL embedding for a set $Y\subseteq \R^{n}$ 
is a matrix $S \in\R^{m\times n}$ such that
\begin{equation}\label{JL}
(1-\epsilon)\|y\|_2^2 \leq \|Sy\|_2^2 \leq (1+\epsilon)\|y\|_2^2 \quad \text{for all}\,\, y \in Y.
\end{equation}
\end{definition}
Often, in the above definitions, the set $Y=\{y_1,\ldots,y_k\}$ is a finite collection of vectors in $\R^n$. But an infinite number of points may also be embedded, such as in the
case when $Y$ is an entire subspace. Then, an embedding approximately preserves pairwise distances between any points in the column 
space of a  matrix $B\in \R^{n\times k}$.
\begin{definition}[$\epsilon$-subspace embedding \cite{10.1561/0400000060}]\label{subspace_embedding_def1_statement}
An $\epsilon$-subspace embedding for a matrix $B \in \R^{n\times k}$ is a matrix $S\in\R^{m\times n}$ such that
\begin{equation}\label{subspace_embedding_def1}
(1-\epsilon)\|y\|_2^2 \leq \|Sy\|_2^2 \leq (1+\epsilon)\|y\|_2^2 \quad \text{ for all $y\in Y=\{y: y=Bz, z\in \R^k\}$}. 
\end{equation}
\end{definition}
In other words, $S$ is an $\epsilon$-subspace embedding for $B$ if and only if $S$ is 
an $\epsilon$-JL embedding for  the column subspace $Y$ of $B$.

 Oblivious embeddings are matrix distributions such that given a(ny)  subset/column subspace of vectors in $\R^n$, a random matrix drawn 
 from such a distribution is an embedding for these vectors with high probability. We let $1-\delta\in [0,1]$ denote a(ny) success probability of an embedding.
 
\begin{definition}[Oblivious embedding \cite{10.1561/0400000060,10.1109/FOCS.2006.37}] \label{Oblivious_embedding}
A distribution $\cal{S}$ on $S \in \R^{m \times n}$  is an $(\epsilon,\delta)$-oblivious embedding if given a fixed/arbitrary set of vectors, we have that, with probability at least $1-\delta$, a matrix $S$ from the distribution is an $\epsilon$-embedding for these vectors.
\end{definition}
Using the above definitions of embeddings, we have distributions that are {\it oblivious JL-embeddings} for a(ny) given/fixed set $Y$ of some vectors $y\in \R^n$, and distributions that are   {\it oblivious subspace embeddings} 
 for a(ny) given/fixed matrix $B\in \R^{n\times k}$ (and for the corresponding subspace $Y$ of its columns).  We note that depending on the quantities being embedded, in addition to $\epsilon$ and $\delta$ dependencies, the size $m$
 of $S$ may depend on $n$ and  the `dimension' of the embedded sets; for example, in the case of a finite set $Y$ of $k$ vectors in $\R^n$, $m$ additionally may depend on $k$ while in the subspace embedding case, $m$ may depend on  the rank $r$ of $B$.

\subsection{Generic properties of subspace embeddings}


A necessary condition for a matrix $S$ to be an $\epsilon$-subspace embedding for a given matrix is that the sketched matrix  has the same rank. The proofs of the next two lemmas are provided in the appendix. 

\begin{lemma}\label{rank-of-sketched-equal-to-unsketched}
If the matrix $S$ is an $\epsilon$-subspace embedding for a given matrix $B$ for some $\epsilon \in (0,1)$, then  $rank(SB) = rank(B)$, where $rank(\cdot)$ denotes the rank of the argument matrix.
\end{lemma}

Given any matrix $A \in \R^{n\times d}$ of rank $r$, the compact singular value decomposition (SVD) of $A$ provides a perfect subspace embedding.
In particular, let
\begin{align}
A = U \Sigma V^T,\label{thin-SVD}
\end{align}
where $U \in \R^{n \times r}$ with orthonormal columns, $\Sigma \in \R^{r \times r}$ is diagonal matrix with strictly positive diagonal entries, and $V \in \R^{d \times r}$ with orthonormal columns  \cite{10.5555/248979}.  Then the matrix $U^T$ is a $\epsilon$-subspace embedding for $A$ for any $\epsilon\in (0,1)$.

Next we connect the embedding properties of $S$ for $A$ with those for $U$ in \eqref{thin-SVD}, using a proof technique in Woodruff \cite{10.1561/0400000060}.

\begin{lemma} \label{subspace-embedding-def-2}
Let $A \in \R^{n \times d}$ with rank $r$ and SVD-decomposition factor 
$U \in \R^{n \times r}$ defined in \eqref{thin-SVD}, and let
$\epsilon \in (0,1)$. Then the following equivalences hold:
\begin{itemize}
\item[(i)] 
a matrix $S$ is an $\epsilon$-subspace embedding for  $A$ if and only if $S$ is an $\epsilon$-subspace embedding for $U$, namely,
\begin{align}
(1-\epsilon)\|Uz\|_2^2 \leq \|SUz\|_2^2 \leq (1+\epsilon)\|Uz\|_2^2, \quad \text{for all $z \in \R^{r}$}.\label{U-condition}
\end{align}
\item[(ii)] A matrix $S$ is an $\epsilon$-subspace embedding for  $A$ if and only if for all $z\in \R^{r}$ with $\|z\|_2=1$, we have\footnote{We note that since
$\|z\|_2=1$ and $U$ has orthonormal columns, $\|Uz\|_2=\|z\|_2=1$ in \eqref{U-condition-2}.}
\begin{align}
(1-\epsilon)\|Uz\|_2^2 \leq \|SUz\|_2^2 \leq (1+\epsilon)\|Uz\|_2^2. \label{U-condition-2}
\end{align}
\end{itemize}
\end{lemma}

\begin{remark}\label{rem_rd}
Lemma \ref{subspace-embedding-def-2}  shows that to obtain a subspace embedding for an $n\times d$ matrix $A$ it is sufficient (and necessary) to embed correctly its left-singular matrix that has rank $r$. Thus, the dependence on $d$ in subspace embedding results can be replaced by dependence on $r$, the rank of the input matrix $A$. As rank deficient matrices $A$ are important in this paper, we opt to state our results in terms of their $r$ dependency (instead of $d$).
\end{remark}

The matrix $U$ in \eqref{thin-SVD} can be seen as the ideal `sketching' matrix for $A$; however,  there is not much computational gain in doing this as computing the compact SVD has similar complexity as computing a minimal residual solution to (\ref{LLS-statement}) directly.

\subsection{Sparse matrix distributions and their embeddings properties}

In terms of optimal embedding properties, it is well known that (dense)
scaled Gaussian matrices  $S$ with $m = \mathcal{O} \left( \epsilon^{-2} (r + \log(1/\delta) )\right)$ provide an $(\epsilon, \delta)$-oblivious subspace embedding for $n\times d$ matrices $A$ of rank $r$ \cite{10.1561/0400000060}. However, the computational cost of the matrix-matrix product $SA$  is  $\mathcal{O}(nd^2)$, which is similar 
to the complexity of solving the original LLS problem \eqref{LLS-statement}; thus it seems difficult to achieve computational gains by calculating a sketched solution of 
\eqref{LLS-statement} in this case.
In order to improve the computational cost of using sketching for solving LLS problems, and to help preserve input sparsity (when $A$ is sparse), sparse random matrices have been proposed, 
namely, such as random matrices with one non-zero per row.  However, uniformly sampling rows of $A$ (and entries of $b$ in  \eqref{LLS-statement}) may miss choosing some (possibly important) row/entry.  A more robust proposal, both theoretically and numerically, is to use hashing matrices, with one (or more) nonzero entries per column, which when applied to $A$ (and $b$), captures all rows of $A$  (and entries of $b$) by adding two (or more) rows/entries with randomised signs.


 \begin{definition}\cite{10.1145/3019134} \label{def::sampling_and_hashing}
$S \in \R^{m \times n}$ is a $s$-hashing matrix if independently for each $j \in [n]$, we sample without replacement $i_1, i_2, \dots, i_s \in [m]$ uniformly at random and let $S_{i_k j} = \pm 1/\sqrt{s}$ for $k = 1, 2, \dots, s$.
\end{definition}
It follows that when $s=1$, 	$S \in \R^{m \times n}$ is a {\it $1$-hashing matrix} if independently for each $j \in [n]$, we sample $i \in [m]$ uniformly at random and let $S_{ij} = \pm1$ \footnote{The random signs are so as  to ensure that in expectation, the matrix $S$ preserves the norm of a vector $x \in \R^n$.}.

Still, in general, the optimal dimension dependence present in Gaussian subspace embeddings cannot be replicated even for hashing distributions, as our next example illustrates.

\begin{example}[The $1$-hashing matrix distributions fails to yield an oblivious subspace embedding with $m = \mathcal{O}(r)$]
	Consider the matrix
	\begin{align}
	A = \begin{pmatrix}
		I_{r\times r} & 0 \\ 0 & 0
	\end{pmatrix} \in \R^{n \times d}. \label{eqn::problematic_matrix}
	\end{align}
	If $S$ is a 1-hashing matrix with $m=\mathcal{O}(r)$, then 
	$\displaystyle SA = \bigl( S_1 \quad 0 \bigr)$, where the $S_1$ block contains the first $r$ columns of  $S$ 
	To ensure that the rank of $A$ is preserved (cf. Lemma \ref{rank-of-sketched-equal-to-unsketched}), a necessary condition for $S_1$ to have rank $r$ is that the $r$ non-zeros of $S_1$ are in different rows. Since, by definition, the respective  row is chosen independently and uniformly at random for each column of $S$, the probability of $S_1$ having rank $r$ is no greater than

	\begin{align}\label{hashing-fails}
	\left(1 - \frac{1}{m}\right) \cdot \left(1 - \frac{2}{m}\right)\cdot \ldots \cdot\left(1-\frac{r-1}{m}\right) 
	\leq e^{-\frac{1}{m}-\frac{2}{m}-\ldots -\frac{r-1}{m}}  
	= e^{-\frac{r(r-1)}{2m}},
	\end{align}
For the probability\footnote{The argument in the example relating  to 1-hashing sketching is related to the birthday paradox, as mentioned (but not proved) in Nelson and Nguyen \cite{10.1145/2488608.2488622}.}  \eqref{hashing-fails}  to be at least $1/2$, we must have
 $m \geq \frac{r(r-1)}{2 \log (2)}$.
\end{example}

The above example  improves upon the lower bound in Nelson et al.  \cite{10.1145/2488608.2488622} by slightly relaxing the requirements on $m$ and $n$\footnote{We note that
in fact,  \cite{Nelson:te} considers a more general set up, namely, any matrix distribution with column sparsity one.}.
We note that in the order of $r$ (or equivalently\footnote{See Remark \ref{rem_rd}.}, $d$), the lower bound $m=\mathcal{O}(r^2)$ for $1$-hashing matches the upper bound given in Nelson and Nguyen \cite{Nelson:te}, Meng and Mahoney \cite{10.1145/2488608.2488621}. 


When $S$ is an $s$-hashing matrix, with $s>1$, the tight bound $m=\Theta(r^2)$ can be improved to 
$m=\Theta (r\log r)$ for $s$ sufficiently large. In particular, Cohen \cite{10.5555/2884435.2884456} derived a general upper bound  that implies, for example, subspace embedding properties of 
$s$-hashing matrices provided $m=\mathcal{O}(d\log d)$ and $s=\mathcal{O}(\log d)$; the value of $s$ may be further reduced to a constant (that is not equal to $1$) at the expense of increasing $m$ and worsening its dependence of $d$. 
A lower bound for guaranteeing oblivious embedding properties of $s$-hashing matrices is given in \cite{Nelson:2014uu}.
Thus we can see that for $s$-hashing (and especially for $1$-hashing) matrices, their general subspace embedding properties are suboptimal in terms of the dependence of $m$ on $d$ when compared to the Gaussian sketching results. To improve the embedding properties of hashing matrices, we must focus on special structure input matrices.

\subsubsection{Coherence-dependent embedding properties of sparse random matrices}

A  feature of the problematic matrix (\ref{eqn::problematic_matrix}) is that its rows are separated into two groups, with the first $r$ rows containing all the information. If the rows  of $A$ were more `uniform' in the sense of equally important in terms of relevant information content, hashing may perform better as a sketching matrix. Interestingly, it is not the uniformity of the rows of $A$ but the uniformity of the rows of $U$, the left singular matrix from the compact SVD of $A$, that plays an important role. The concept of coherence is a useful proxy for the uniformity of the rows of $U$ and $A$\footnote{We note that sampling matrices were shown to have good subspace embedding properties for input matrices with low coherence \cite{10.1145/1132516.1132597, Tropp:wr}.
Even if the coherence is minimal, the size of the sampling matrix has a $d\log d$ dependence where the $\log d $ term cannot be removed due to the coupon collector problem \cite{Tropp:wr}.}.

\begin{definition} (Matrix coherence \cite{10.1561/2200000035})
\label{def::coherence}
The coherence of a matrix $A \in \R^{n\times d}$, denoted $\mu(A)$, is the largest Euclidean norm of the rows of $U$ defined in (\ref{thin-SVD}). Namely,
\begin{align}
\mu(A) = \max_{i\in [n]} \|U_i\|_2,
\end{align}
where  $U_i$ denotes the $i$th row of $U$.
\end{definition}

Some useful properties follow.

\begin{lemma}
Let $A \in \R^{n\times d}$ have rank $r\leq d\leq n$. Then
\begin{equation} \label{mu_bound}
\sqrt{\frac{r}{n}} \leq \mu(A) \leq 1.
\end{equation}
Furthermore, if $\mu(A) = \sqrt{\frac{r}{n}}$, then $\|U_i\|_2= \sqrt{\frac{r}{n}}$ for all $i\in [n]$ where $U$ is defined in \eqref{thin-SVD}.
\end{lemma}

\begin{proof}
Since the matrix $U \in \R^{n\times r}$ has orthonormal columns, we have that
\begin{equation}
    \sum_{i=1}^{n} \|U_i\|_2^2 = r. \label{eqn:sum_of_U_i}
\end{equation} Therefore the maximum 2-norm of $U$ must not be less than $\sqrt{\frac{r}{n}}$, and thus $\mu(A) \geq \sqrt{\frac{r}{n}}$. Furthermore, if $\mu(A) = \sqrt{\frac{r}{n}}$, then \eqref{eqn:sum_of_U_i} implies $\|U_i\|_2 = \sqrt{\frac{r}{n}}$ for all $i \in [n]$.

Next, by expanding the set of columns of $U$ to a basis of $R^n$, there exists $U_f \in \R^{n \times n}$ such that $U_f = \bigl(  U \; \hat{U}  \bigr)$ orthogonal where $\hat{U} \in \R^{n \times (n-d)}$ has orthonormal columns. The 2-norm of $i$th row of $U$ is bounded above by the 2-norm of $i$th row of $U_f$, which is one. Hence $\mu(A) \leq 1$. 
\end{proof}

We note that for $A$ in (\ref{eqn::problematic_matrix}), we have $\mu(A) = 1$. The  maximal coherence of this matrix sheds some light on the ensuing poor embedding properties we noticed in Example 1.

Bourgain et al \cite{Bourgain:2015tc}  gives a general result that captures the coherence-restricted subspace embedding properties 
of $s$-hashing matrices.

\begin{theorem}[Bourgain et al \cite{Bourgain:2015tc}] \label{Bourgain}
Let $A \in \R^{n \times d}$ with coherence $\mu(A)$  and rank $r$; and  let $0 < \epsilon, \delta <1$.
Assume also that
\begin{align}
m \geq c_1\max\left\{\delta^{-1},(r + \log m ) \left[ \min \left\{ \log^2(r/\epsilon), \log^2(m)
\right\} + r\log(1/\delta) \right]\epsilon^{-2}\right\} \label{Bourgain-m}\\
{\rm and}\quad s \geq c_2 \left[ \log(m) \log(1/\delta) \min \left\{ \log^2(r/\epsilon), \log^2(m)
\right\} + \log^2(1/\delta) \right] \mu(A)^2 \epsilon^{-2},\label{Bourgain_s_eqn} 
\end{align}
where $c_1$ and $c_2$ are  positive constants.
Then  a(ny) s-hashing matrix $S \in \R^{m\times n}$ is an $\epsilon$-subspace embedding for  $A$  with probability at least $1-\delta$.

\end{theorem}
Substituting $s=1$ in (\ref{Bourgain_s_eqn}), we can use the above Theorem to deduce an upper bound $\mu$ of acceptable  coherence values of the input matrix $A$, namely,
\begin{align}
\mu(A)\leq c_2^{-1/2} \epsilon \left[ \log(m) \log(1/\delta) \min \left\{ \log^2(r/\epsilon), \log^2(m)
\right\} + \log^2(1/\delta) \right]^{-1/2} :=\mu.\label{Bourgain-1-hashing-mu}
\end{align}
Thus Theorem \ref{Bourgain} implies that the distribution of $1$-hashing matrices with
$m$ satisfying (\ref{Bourgain-m})  is 
 an oblivious subspace embedding for any input matrix $A$ with 
 $\mu(A)\leq \mu$, where  $\mu$ is defined in \eqref{Bourgain-1-hashing-mu}.

\subsubsection{Non-uniformity of vectors and their relation to embedding properties of sparse random matrices}

In order to prove some of our main results, we need a corresponding notion of coherence of vectors,  in order to be able to measure the `importance' of their respective entries; this is captured by the so-called non-uniformity of a vector.

\begin{definition}[Non-uniformity of a vector]
Given $x \in \R^{n}$, the non-uniformity of $x$, $\nu(x)$, is defined as
\begin{align}
\nu(x) = \frac{\|x\|_{\infty}}{\|x\|_2}. 
\end{align}
\end{definition}

We note that for any vector $x \in \R^{n}$, we have $\frac{1}{\sqrt{n}} \leq \nu(x) \leq 1$.

\begin{lemma}\label{non_uniformity_col_subspace_coherence}
Given $A \in \R^{n\times d}$, let $y=Ax$ for some $x \in \R^{d}$. Then
\begin{align}
\nu(y) \leq \mu(A).
\end{align}
\end{lemma}

The next lemmas are crucial to our results in the next section; the proof of the first lemma can be found in the respective paper \cite{10.5555/3327345.3327444},
while the proofs of Lemmas \ref{point to JL plus} and Lemma \ref{lem::lemma7} are given in the appendix.

We also note, in subsequent results, the presence of {\it problem-independent constants}, also called absolute constants that will be implicitly or explicitly defined, depending on the context. Our convention here is as expected, that the same notation denotes the same constant across all results in the paper. 

The following expression will be needed in our results,
\begin{equation}
\bar{\nu}(\epsilon,\delta):= \nuOneHashing,
      \label{L5:nu}
\end{equation}
where $\epsilon, \delta \in (0,1)$ and $E, C_1>0$.

\begin{lemma}[\cite{10.5555/3327345.3327444}, Theorem 2] \label{Freksen}
Suppose that $\epsilon, \delta \in (0,1)$, and $E$ satisfies $C \leq E < \frac{2}{\delta \log(1/\delta)}$, where $C>0$ and $C_1$ are  problem-independent constants. Let 
$m \leq n \in \N$ with
$m\geq E \epsilon^{-2} \log(1/\delta)$.

Then,  for any $x\in \R^n$ with 
\begin{equation}
    \nu(x) \leq\bar{\nu}(\epsilon,\delta),
    \label{L5:nu-1}
\end{equation}
where $\bar{\nu}(\epsilon,\delta)$ is defined in \eqref{L5:nu},
a randomly generated 1-hashing matrix $S\in \R^{m\times n}$ is an $\epsilon$-JL embedding for $\{x\}$ with probability at least $1-\delta$.

\end{lemma}

\begin{lemma} \label{point to JL plus}
Let $\epsilon, \delta \in (0,1)$, $\nu \in (0,1]$ and $m,n \in \N$. Let $\cal{S}$ be a distribution of $m\times n$ random matrices. Suppose that for any given 
y with $\nu(y) \leq \nu$, a matrix $S \in \R^{m\times n}$ randomly drawn from $\cal{S}$ is an $\epsilon$-JL embedding for $\{y\}$ with probability at least $1-\delta$.
Then for any given set $Y\subseteq \R^n$ with $\max_{y \in Y} \nu(y) \leq \nu$ and cardinality $|Y|\leq 1/\sqrt{\delta}$, a matrix $S$ randomly drawn from $\cal{S}$ is an $\epsilon$-JL embedding for Y with probability at least $1-|Y|\delta$.
\end{lemma}

\begin{lemma}\label{lem::lemma7}
Let $\epsilon\in (0,1)$, and $Y \subseteq \R^n$ be a finite set such that
$\|y\|_2=1$ for each $y\in Y$. Define 
\begin{align}
    & \yPlus= \yPlusExpression \\
    & \yMinus = \yMinusExpression.
\end{align} 
If $S \in \R^{m\times n}$ is an $\epsilon$-JL embedding for $\set{ \union{\yPlus}{\yMinus} }$, then $S$ is a generalised $\epsilon$-JL embedding for $Y$.
\end{lemma}

\section{Hashing sketching with \mathInTitle{$m=\mathcal{O}(r)$}}

Our first result shows that if the coherence of the input matrix is sufficiently low, the distribution of $1$-hashing matrices with $m = \mathcal{O}(r)$ is an $(\epsilon,\delta)$-oblivious subspace embedding.

The following expression will be useful later,
\begin{equation}\label{A:mu-1}
\bar{\mu}(\epsilon,\delta):=  \muAOneHashing,
\end{equation}
where $\epsilon, \delta \in (0,1)$,  $r, E>0$ are to be chosen/defined depending on the context, and $C_1, C_2>0$ are problem-independent constants.

\begin{theorem} \label{thm1}
Suppose that $\epsilon,\delta \in (0,1)$, $r \leq d \leq n, m\leq n \in \N^+$, $E >0$ satisfy 
\begin{align}
&C \leq E \leq \EUpperHashing,\label{eqn::theoremTwoOne}\\[0.5ex]
&m \geq \mLower,\label{eqn::theoremTwoTwo}
\end{align}
where $C>0$ and $C_1, C_2>0$  are problem-independent constants. Then for any matrix $A\in\R^{n\times d}$ with rank $r$ and 
\begin{equation}\label{A:mu}
    \mu(A) \leq  \bar{\mu}(\epsilon,\delta),
\end{equation}
where $\bar{\mu}(\epsilon,\delta)$ is defined in \eqref{A:mu-1}, a randomly generated 1-hashing matrix $S\in\R^{m\times n}$ is an $\epsilon$-subspace embedding for $A$ with probability at least $1-\delta$.
\end{theorem}

The proof of Theorem \ref{thm1} 
relies on the fact that the coherence of the input matrix gives a bound on the non-uniformity of the entries for all vectors in its column space (Lemma \ref{non_uniformity_col_subspace_coherence}),
adapting standard arguments in \cite{10.1561/0400000060} involving set covers, which are defined next and proved in the appendix.

\begin{definition}
A $\gamma$-cover of a set $M$ is a subset $N\subseteq M$ with the property that given any point $y \in M$, there exists $w \in N$ such that
 $\|y-w\|_2 \leq \gamma$. 
\end{definition}

Consider a given real  matrix $U\in\R^{n\times r}$ with orthonormal columns, and let
\begin{equation}\label{MU}
M:= \{Uz \in \R^n: z \in \R^r, \|z\|_2=1\}.
\end{equation} 

The next two lemmas show the existence of a $\gamma$-net $N$ for $M$, and connect generalised JL embeddings for $N$ with JL embeddings for $M$; their proofs can be found in the appendix.

\begin{lemma} \label{Gamma-cover-existance}

Let $0< \gamma <1$, $U\in\R^{n\times r}$ have orthonormal columns and $M$ be defined in \eqref{MU}. Then there exists a $\gamma$-cover $N$ of $M$ such that $|N| \leq (1 + \frac{2}{\gamma})^r$.
\end{lemma}

\begin{lemma} \label{approximation-of-net}
Let $\epsilon, \gamma \in (0,1), U\in \R^{n\times d}, M \subseteq \R^n$ associated with $U$ be defined in \eqref{MU}. Suppose $N \subseteq M$ is a $\gamma$-cover of $M$ and $S \in \R^{m\times n}$ is a generalised $\epsilon_1$-JL embedding for $N$, where $\epsilon_1 = \epsilonPrimeFactor \epsilon$. Then $S$ is an $\epsilon$-JL embedding for $M$. 
\end{lemma}

We are ready to prove Theorem \ref{thm1}.
\begin{proof}[Proof of Theorem 2]
Let $A\in \R^{n\times d}$ with rank $r$ and satisfying \eqref{A:mu}.
Let  $U\in \R^{n\times r}$ be an SVD factor of $A$ as defined in \eqref{thin-SVD}, which by definition of coherence, implies
\begin{equation}
    \mu(U)=\mu (A)\leq \bar{\mu}(\epsilon,\delta),
    \label{th2:barmu}
\end{equation}
where $\bar{\mu}(\epsilon,\delta)$ is defined in \eqref{A:mu-1}. 
We let $\gamma, \epsilon_1, \delta_1 \in (0,1)$ be defined as
  \begin{equation}
        \gamma=\frac{2}{e^2-1}, \quad C_2 =  \epsilonPrimeFactor,\quad 
        \epsilon_1 = C_2 \epsilon \quad {\text{and}}\quad \delta_1 = e^{-4r}\delta,
    \label{th2:eps}
  \end{equation}
and note that $C_2\in (0,1)$ and 
\begin{equation}
\bar{\nu}(\epsilon_1,\delta_1) = \bar{\mu} (\epsilon,\delta),
\label{th2:numu}
\end{equation}
where $\bar{\nu}(\cdot,\cdot)$ is defined in \eqref{L5:nu}.
Let $M \in \R^n$ be associated to $U$ as in \eqref{MU}
and let $N\subseteq M$ be the $\gamma$-cover of $M$  as guaranteed by Lemma \ref{Gamma-cover-existance}, with $\gamma$ defined in \eqref{th2:eps} which implies that $|N| \leq e^{2r}$.

Let $S \in \R^{m\times n}$ be a randomly generated 1-hashing matrix with $m\geq E \epsilon_1^{-2}\log (1/\delta_1)=E C_2^{-2}\epsilon^{-2}[4r+\log(1/\delta)]$, where to obtain the last equality, we used \eqref{th2:eps}.

To show that the sketching matrix $S$ is an $\epsilon$-subspace embedding for $A$ (with probability at least $1-\delta$),
it is sufficient to show that $S$ is an $\epsilon_1$-generalised JL embedding for $N\subseteq M$ (with probability at least $1-\delta$). To see this, recall \eqref{MU} and Lemma \ref{subspace-embedding-def-2}(ii) which show that $S$ is an $\epsilon$-subspace embedding for $A$ if and only if 
$S$ is an $\epsilon-$JL embedding for $M$. Our sufficiency claim now follows by invoking Lemma \ref{approximation-of-net} for $S$, $N$ and $M$.

We are left with considering in detail the cover set $N = \set{y_1,y_2, \dots, y_{|N|}}$ and the following useful ensuing  sets
\begin{align*}
    &\nPlus = \nPlusExpression \quad
    \text{and}\quad \nMinus = \nMinusExpression, \\\nonumber
    &\nMinusOne = \nMinusOneExpression \quad \text{and}\quad
    \nMinusTwo = \nMinusTwoExpression.
\end{align*}
Now let  $Y := \nPlus \cup \nMinusOne$  and show that 
\begin{equation}
    \nu(y)\leq \bar{\nu}(\epsilon_1,\delta_1) \quad\text{for all}\quad y\in Y.
  \label{th2:nu-barnu}
\end{equation}
 To see this, assume first that $y=y_i+y_j\in \nPlus$, with 
$y_i, y_j \in N\subseteq M$. Thus there exist $z_i, z_j\in R^r$ such that 
$y_i=Uz_i$ and $y_j=Uz_j$, and so $y=U(z_i+z_j)$. Using Lemma \ref{non_uniformity_col_subspace_coherence}, $\nu(y)\leq \mu(U)=\mu(A)$, which together with \eqref{th2:barmu}
and \eqref{th2:numu}, gives \eqref{th2:nu-barnu} for points $y\in\nPlus$;
the proof for $y\in \nMinusOne$ follows similarly. 

Lemma \ref{Freksen} with $(\epsilon,\delta):= (\epsilon_1,\delta_1)$ provides that for any $x\in\R^n$ with $\nu(x) \leq \bar{\nu}(\epsilon_1,\delta_1)$, $S$ is an $\epsilon_1$-JL embedding for $\{x\}$ with probability at least $1-\delta_1$.  This and  \eqref{th2:nu-barnu} imply that the conditions of Lemma \ref{point to JL plus} are satisfied for $Y = \nPlus \cup \nMinusOne$, from which we conclude that 
$S$ is an $\epsilon_1$-JL embedding for $Y$ with probability at least $1-|Y|\delta_1$. Note that 
\[
|Y|\leq |\nPlus|+|\nMinusOne|\leq \frac{1}{2} |N|(|N|+1) +  \frac{1}{2}
 |N|(|N|-1)=|N|^2.
 \]
 This, the definition of $\delta_1$ in \eqref{th2:eps} and $|N|\leq e^{2r}$ imply that $1-|Y|\delta_1\geq 1-\delta$. 
Therefore $S$ is an $\epsilon_1$-JL embedding for $\nPlus \cup \nMinusOne$ with probability at least $1-\delta$.

Finally, Definition \ref{def::JL_embedding} of JL-embeddings implies that the sign of the embedded vector is irrelevant and that $\{0\}$ is always embedded, and so if
$S$ is an $\epsilon_1$-JL embedding for $\nPlus \cup \nMinusOne$, it is also an $\epsilon_1$-JL embedding for $\union{\nPlus}{\nMinus}$.
Lemma \ref{lem::lemma7} now provides us with the desired result that then, 
$S$ is a generalised $\epsilon_1$-JL embedding for $N$. 
\end{proof}

Next we discuss the results in Theorem \ref{thm1}.

\paragraph{Conditions for a well-defined coherence requirement}

While our result guarantees optimal dimensionality reduction for the sketched matrix, using a very sparse 1-hashing matrix for the sketch, it imposes  implicit restrictions on the number $n$ of rows of $A$. Recalling 
\eqref{mu_bound}, we note that condition \eqref{A:mu} is well-defined
when 
\begin{equation}\label{th2:nr-mu}
\sqrt{\frac{r}{n}}\leq \bar{\mu}(\epsilon,\delta).
\end{equation}
Using the definition of $\bar{\mu}(\epsilon,\delta)$ in \eqref{A:mu-1} and
assuming reasonably that $\logDeltaI=\mathcal{O}(r)$, 
we have the lower bound

\begin{equation*}
\bar{\mu}(\epsilon,\delta) \geq C_1\sqrt{C_2}\sqrt{\epsilon} \frac{\min\left\{ \log(E/(C_2\epsilon)), \sqrt{\log E}\right\}}{4r+\log(1/\delta)},
\end{equation*}
and so
\eqref{th2:nr-mu} is satisfied if
\begin{align}
n \geq \frac{r (4r+\log(1/\delta)^2}{C_1^2C_2\epsilon\min\left\{ \log^2(E/(C_2\epsilon)), \log E\right\}} = \mathcal{O} \left(  \frac{r^3}{\epsilon \log^2(\epsilon)} \right).
\end{align}

\paragraph{Comparison with data-independent bounds}
    Existing results show that $m = \Theta \left( r^2\right)$ is both necessary and sufficient in order to secure an oblivious subspace embedding property for 1-hashing matrices with no restriction on the coherence of the input matrix \cite{10.1145/2488608.2488622,Nelson:te, 10.1145/2488608.2488621}. Aside from requiring more projected rows than in Theorem \ref{thm1}, these results implicitly impose $n \geq \mathcal{O}(r^2)$ for the size/rank of data matrix in order to secure  meaningful dimensionality reduction.

    \paragraph{Comparison with data-dependent bounds}
    To the best of our knowledge, the only data-dependent result subspace embedding result for hashing matrices is   \cite{Bourgain:2015tc} (see Theorem \ref{Bourgain}). From \eqref{Bourgain-m}, we have that $m\geq c_1 r \min\left\{ \log^2(r/\epsilon), \log^2(m) \right\}\epsilon^{-2}$ and hence $m = \Omega(r \log^2 r)$;  while Theorem \ref{thm1} only needs $m = \mathcal{O}(r)$. However, the coherence requirement on $A$ in Theorem \ref{Bourgain} is weaker than \eqref{A:mu} and so \cite{Bourgain:2015tc}  applies to a wider range of inputs at the expense of a larger value of $m$ required for the sketching matrix.
    
    \paragraph{Summary and look ahead}
Table \ref{tab::m_and_mu_1_hashing} summarises existing results and we see stricter coherence assumptions lead to improved dimensionality reduction properties. In the next section, we investigate relaxing coherence requirements by using hashing matrices with increased column sparsity ($s$-hashing) and coherence reduction transformations. 

    \begin{table}
        \caption{Summary of results for $1$-hashing}
        \label{tab::m_and_mu_1_hashing}
        \centering
\begin{tabular}{|c|c|c|}
\hline
Result                     & \mbox{$\mu$ (coherence of $A$)}                         & \mbox{$m$ (size of sketching $S$)}                        \\ \hline
\cite{10.1145/2488608.2488621} & --                              & $\Theta(r^2)$                \\ \hline
\cite{Bourgain:2015tc}         & $\mathcal{O} \left( \log^{-3/2} (r)\right)$ & $\mathcal{O}\left(r \log^2(r) \right)$ \\ \hline
Theorem \ref{thm1}             & $\mathcal{O} \left( r^{-1}\right)$        & $\mathcal{O}(r)$                       \\ \hline
\end{tabular}
    \end{table}

\section{Relaxing the coherence requirement using \mathInTitle{$s$}-hashing matrices}

This section investigates the embedding properties of  $s$-hashing matrices when $s\geq 1$. 
Indeed, \cite{Bourgain:2015tc} shows that $s$-hashing relaxes their particular coherence requirement by $\sqrt{s}$. 
Theorem \ref{thm::s-hashing} presents a similar result for our particular coherence requirement  \eqref{A:mu} that again guarantees embedding properties for $m=\mathcal{O}(r)$. Then we present a new $s$-hashing variant that allows us to give a general result showing that (any) subspace embedding properties of $1$-hashing matrices immediately translate into similar properties for these $s$-hashing matrices when applied to a larger class of data matrices, with larger coherence. 
A simplified embedding result with $m=\mathcal{O}(r)$ is then deduced for this $s$-hashing variant.

Numerical benefits of $s$-hashing (for improved preconditioning) are investigated in later sections; see Figures \ref{fig::1-2-3-inco} and \ref{fig::1-2-3-semi-co} for example.

\subsection{The embedding properties of \mathInTitle{$s$}-hashing matrices}
Our next result shows that using $s$-hashing (Definition \ref{def::sampling_and_hashing}) relaxes the particular coherence requirement in Theorem \ref{thm1} by $\sqrt{s}$. 

\begin{theorem} \label{thm::s-hashing} 
Let $ r \leq d\leq n \in \N^+$. Let  $C_1, C_2, C_3, C_M, C_{\nu}, C_s >0$ be \constantsDescription. Suppose that $\epsilon,\delta \in (0, C_3)$, 
$m,\,s\in \N^+$ and $E>0$ satisfy\footnote{Note that the expressions
of the lower bounds in \eqref{eqn::theoremTwoTwo} and \eqref{s:m_lower}
are identical apart from  the choice of $E$ and the condition $m \geq se$.}
\begin{align}
&1\leq s \leq C_s C_2^{-1} \epsilon^{-1} \fourRplusLog,\label{eqn::theoremThreeOne}\\[0.5ex]
& C_M \leq E \leq \EUpperSHashing,\label{eqn::theoremThreeTwo}\\[0.5ex]
&m \geq \set{\mLower, se}.\label{s:m_lower}
\end{align}
Then for any matrix $A\in \R^{n\times d}$ with rank $r$ and $\mu(A) \leq \sqrt{s}C_{\nu} C_1^{-1} \bar{\mu}(\epsilon,\delta)$, where $\bar{\mu}(\epsilon,\delta)$ is defined in \eqref{A:mu-1}, a randomly generated $s$-hashing matrix $S \in \R^{m\times n}$ is an $\epsilon$-subspace embedding for $A$ with probability at least $1-\delta$.
\end{theorem}

Theorem \ref{thm::s-hashing} parallels Theorem \ref{thm1}; and its proof relies on the following lemma which parallels Lemma \ref{Freksen}.

\begin{lemma}[\cite{NIPS2019_9656}, Theorem 1.5] \label{Jaga-lemma}
Let $C_1, C_3, C_M, C_{\nu}, C_s>0$ be \constantsDescription. Suppose that $\epsilon,\delta \in (0, C_3), m,s\in \N^+, E\in \R$ satisfy 
\begin{align}\nonumber
&1\leq s \leq C_s \epsilon^{-1} \logDeltaI,\\\nonumber
& C_M \leq E < \epsilon^2 s\log^{-1}(1/\delta) e^{C_s (\epsilon s)^{-1}\logDeltaI },\\\nonumber
& m\geq \max \set{ E \epsilon^{-2} \log(1/\delta), se}.
\end{align}
Then for any $x\in \R^n$ with $\nu(x) \leq \sqrt{s}C_{\nu}C_1^{-1}\bar{\nu}(\epsilon,\delta)$, where $\bar{\nu}(\epsilon,\delta)$ is defined in \eqref{L5:nu},
a randomly generated $s$-hashing matrix $S\in \R^{m\times n}$ is an $\epsilon$-JL embedding for $\{x\}$ with probability at least $1-\delta$.
\end{lemma}

The proof of Theorem \ref{thm::s-hashing} follows the same argument as Theorem \ref{thm1}, replacing $1$-hashing with $s$-hashing and using Lemma \ref{Jaga-lemma} instead of Lemma \ref{Freksen}. We omit the details.

\subsection{A general embedding property for an \mathInTitle{$s$}-hashing variant}

Note that in both Theorem \ref{Bourgain} and Theorem \ref{thm::s-hashing}, allowing column sparsity of hashing matrices to increase from $1$ to $s$ results in coherence requirements being relaxed by $\sqrt{s}$. We introduce an
$s$-hashing variant that allows us to generalise this result.

\begin{definition} 
\label{def::s-hashing-variant}
 We say $T \in \R^{m \times n}$ is an s-hashing variant matrix if independently for each $j \in [n]$, we sample with replacement $i_1, i_2, \dots, i_s \in [m]$ uniformly at random and add $\pm 1/\sqrt{s}$ to $T_{i_k j}$, where $k = 1, 2, \dots, s$. \footnote{We add $\pm 1/\sqrt{s}$ to $T_{i_k j}$ because we may have $i_k = i_l$ for some $l <k$, as we have sampled with replacement.  }
\end{definition}

Both $s$-hashing and $s$-hashing variant matrices reduce to $1$-hashing matrices when $s=1$. For $s\geq 1$, the $s$-hashing variant has at most $s$ non-zeros per column, while the usual $s$-hashing matrix has precisely $s$ nonzero entries per same column.

The next lemma connects $s$-hashing variant matrices to $1$-hashing matrices. 

\begin{lemma}
\label{decompose_s_hashing_varaint}
An $s$-hashing variant matrix $T\in \R^{m\times n}$ (as  in Definition \ref{def::s-hashing-variant}) could alternatively be generated by calculating $T = \frac{1}{\sqrt{s}} \left[ S^{(1)} + S^{(2)} + \dots + S^{(s)} \right]$, where $S^{(k)} \in \R^{m\times n}$ are independent $1$-hashing matrices for $k = 1, 2, \dots, s$.
\end{lemma}

\begin{proof}
In Definition \ref{def::s-hashing-variant}, an s-hashing variant matrix $T$ is generated by the following procedure:
\begin{algorithmic}
      \For{\texttt{$j = 1, 2, \dots n$}}
        \For{ \texttt{$k = 1, 2, \dots, s$}}
            \State \texttt{Sample $i_k \in [m]$ uniformly at random and add $\pm 1/\sqrt{s}$ to $T_{i_k,j}$. }
        \EndFor
      \EndFor
\end{algorithmic}
Due to the independence of the entries, the 'for' loops in the above routine can be swapped, leading to the equivalent formulation,
\begin{algorithmic}
      \For{\texttt{$k = 1, 2, \dots s$}}
        \For{ \texttt{$j = 1, 2, \dots, n$}}
            \State \texttt{Sample $i_k \in [m]$ uniformly at random and add $\pm 1/\sqrt{s}$ to $T_{i_k,j}$}.
        \EndFor
      \EndFor
\end{algorithmic}
For each $k\leq s$, the 'for' loop over $j$ in the above routine generates an independent random $1$-hashing matrix $S^{(k)}$ and adds $\left( 1/\sqrt{s} \right)S^{(k)}$ to $T$.

\end{proof}

We are ready to state and prove the main result in this section. 
\begin{theorem} \label{1-hashing-and-s-hashing}
Let $s, r \leq d\leq n \in \N^+$, $\epsilon, \delta\in (0, 1)$. Suppose that $m \in \N^+$ is chosen such that the distribution of $1$-hashing matrices $S \in \R^{m\times ns}$ is an $(\epsilon,\delta)$-oblivious subspace embedding for any matrix $B\in \R^{ns\times d}$ with rank $r$ and $\mu(B)\leq \mu$ for some $\mu>0$. Then the distribution of $s$-hashing variant matrices $T\in\R^{m\times n}$ is an $(\epsilon, \delta)$-oblivious subspace embedding for any matrix $A\in \R^{n\times d}$ with rank $r$ and  $\mu(A) \leq \mu\sqrt{s}$.

\end{theorem}

\begin{proof}
Applying Lemma \ref{decompose_s_hashing_varaint}, we let 
\begin{equation}
T = \frac{1}{\sqrt{s}} \left[ S^{(1)} + S^{(2)} + \ldots + S^{(s)} \right]
\end{equation}
be a randomly generated $s$-hashing variant matrix where $S^{(k)} \in \R^{m\times n}$ are independent $1$-hashing matrices, $k\in \{1,\ldots, s\}$. Let $A\in \R^{n\times d}$ with rank $r$ and  with $\mu(A)\leq \mu\sqrt{s}$; let
$U \in \R^{n \times r}$ be an SVD-factor of $A$ as defined in \eqref{thin-SVD}. Let 
\begin{equation}
W = \frac{1}{\sqrt{s}} \begin{pmatrix} U\\ \vdots \\ U \end{pmatrix} \in \R^{ns \times r}.
\end{equation}
As $U$ has orthonormal columns, the matrix $W$ also has orthonormal columns and hence the coherence of $W$ coincides with the largest Euclidean norm of its rows
	\begin{equation}
	\mu(W) = \frac{1}{\sqrt{s}} \mu(U)=\frac{1}{\sqrt{s}} \mu(A) \leq \mu.
	\end{equation}
Let $S = \begin{pmatrix} S^{(1)} \hdots  S^{(s)} \end{pmatrix} \in \R^{m \times ns}$.
We note that the $j$-th column of $S$ is generated by sampling $i \in [m]$ and setting $S_{ij} = \pm 1$. Moreover, as $S^{(k)}$,   $k\in \{1,\ldots, s\}$, are independent, the sampled entries are independent. Therefore, $S$ is distributed as a $1$-hashing matrix. Furthermore, due to our assumption on the distribution of $1$-hashing matrices, 
$m$ is chosen such that  $S \in \R^{m \times ns}$ is an $(\epsilon, \delta)$-oblivious subspace embedding for $(ns)\times r$ matrices of coherence at most $\mu$. Applying this to input matrix $W$, we have that 
 with probability at least $1-\delta$, 
	\begin{equation}\label{s-hash:SW}
	(1-\epsilon)\|z\|^2_2=	(1-\epsilon)\|Wz\|^2_2 \leq \|SWz\|^2_2 \leq (1+\epsilon)\|Wz\|^2_2=	(1+\epsilon)\|z\|^2_2,
	\end{equation}
	for all $z\in \R^r$, where in the equality signs, we used that  $W$ has orthonormal columns.
On the other hand, we have that 
\begin{equation*}
SW = \frac{1}{\sqrt{s}} \begin{pmatrix} S^{(1)} \hdots  S^{(s)} \end{pmatrix}  \begin{pmatrix} U\\ \vdots \\ U \end{pmatrix}  
   = \frac{1}{\sqrt{s}} \left[ S^{(1)} U + S^{(2)} U + \dots + S^{(s)} U \right]
   = TU. 
\end{equation*}
This and \eqref{s-hash:SW} provide that, with probability 
at least $1-\delta$, 
	\begin{equation}
	(1-\epsilon)\|z\|^2_2 \leq \|TUz\|^2_2 \leq (1+\epsilon)\|z\|^2_2,
	\end{equation}
	which implies that $T$ is an $\epsilon$-subspace embedding for $A$ by Lemma \ref{subspace-embedding-def-2}.
\end{proof}

 Theorem \ref{thm1} and Theorem \ref{1-hashing-and-s-hashing} imply an $s$-hashing variant version of Theorem \ref{thm::s-hashing}.

\begin{theorem}\label{thm::s-hashing-variant}
Suppose that $\epsilon,\delta \in (0,1)$, $s, r \leq d \leq n, m\leq n \in \N^+$, $E >0$ satisfy \eqref{eqn::theoremTwoOne} and \eqref{eqn::theoremTwoTwo}. Then for any matrix $A\in\R^{n\times d}$ with rank $r$ and 
$\mu(A) \leq \bar{\mu}(\epsilon,\delta)\sqrt{s}$, \whereMuBarIsDefined,
a randomly generated $s$-hashing variant matrix $S\in\R^{m\times n}$ is an $\epsilon$-subspace embedding for $A$ with probability at least $1-\delta$.
\end{theorem}

\begin{proof}
Theorem \ref{thm1} implies that the distribution of $1$-hashing matrices $S\in \R^{m\times ns}$ is an $(\epsilon,\delta)$-oblivious subspace embedding
for any matrix $B \in \R^{ns \times d}$ with rank $r$ and $\mu(B) \leq \bar{\mu}(\epsilon,\delta)$. We also note that this result is invariant to the number of rows in $B$ (as long as the column size of $S$ matches the row count of $B$), and so the expressions for $m$, $\bar{\mu}(\epsilon,\delta)$ and the constants therein remain unchanged. 
 
 Theorem \ref{1-hashing-and-s-hashing} then provides that the distribution of $s$-hashing variant matrices $S \in \R^{m\times n}$ is an $(\epsilon,\delta)$-oblivious subspace embedding for any matrix $A \in \R^{n\times d}$ with rank $r$ and $\mu(A) \leq \bar{\mu}(\epsilon,\delta) \sqrt{s}$; the desired result follows. 
\end{proof}

Theorem \ref{thm::s-hashing} and Theorem \ref{thm::s-hashing-variant} provide similar results, and we find that the latter provides simpler constant expressions (such as for $E$).

\subsection{The Hashed-Randomised-Hadamard-Transform sketching}

Here we consider the Randomised-Hadamard-Transform \cite{10.1145/1132516.1132597}, to be applied to the input matrix $A$ before sketching, as
another approach that allows reducing the coherence requirements under which  good subspace embedding properties can be guaranteed. It is common to use the Subsampled-RHT (SHRT) \cite{10.1145/1132516.1132597}, but the size of the sketch needs to be at least 
$\mathcal{O}(r\log r)$; this prompts us to consider using hashing instead of subsampling in this context (as well), and obtain an optimal order sketching bound. Figure \ref{fig::1-2-3-inco} illustrates numerically the benefit of HRHT sketching for preconditioning compared to SRHT.

\begin{definition}\label{def::HRHT}
A Hashed-Randomised-Hadamard-Transform (HRHT) is an $m\times n$ matrix of the form $S =  S_h HD$ with $m\leq n$, where
\begin{itemize}[topsep=0pt,itemsep=-1ex,partopsep=1ex,parsep=1ex]
    \item $D$ is a random $n \times n$ diagonal matrix with $\pm 1$ independent entries.
    \item $H$ is an $n\times n$ Walsh-Hadamard matrix defined by
        \begin{equation}
            H_{ij} = n^{-1/2}(-1)^{\la (i-1)_2, (j-1)_2\ra},
        \end{equation}
        where $(i-1)_2$, $(j-1)_2$ are binary representation vectors of the numbers $(i-1), (j-1)$ respectively\footnote{For example, $(3)_2 = (1,1)$.}.
    \item $S_h$ is a random $m\times n$ $s$-hashing or $s$-hashing variant matrix, independent of $D$.
\end{itemize}
\end{definition}

Our next results show 
that if the input matrix is sufficiently over-determined, the distribution of $HRHT$ matrices with optimal sketching size and either choice of $S_h$, is an $(\epsilon, \delta)$-oblivious subspace embedding.

\begin{theorem}[$s$-hashing version] \label{thm::HRHT}
$ r \leq d\leq n \in \N^+$. \theoremThreeFirstSentence \eqref{eqn::theoremThreeOne}, \eqref{eqn::theoremThreeTwo} and \eqref{s:m_lower}. Let $\delta_1 \in (0,1)$ and suppose further that
\begin{equation}
    n \geq \nLower \label{eqn::theoremSixN},
\end{equation}
\whereMuBarIsDefined. Then for any matrix $A \in \R^{n \times d}$ with rank $r$, an $HRHT$ matrix $S \in \R^{m\times n}$ with an $s$-hashing matrix $S_h$, is an $\epsilon$-subspace embedding for $A$ with probability at least $(1-\delta)(1-\delta_1)$. 

\end{theorem}

\begin{theorem}[$s$-hashing variant distribution] \label{thm::HRHT_variant}
\theoremFiveFirstSentence \eqref{eqn::theoremTwoOne} and \eqref{eqn::theoremTwoTwo}. Let $\delta_1 \in (0,1)$ and suppose further that 
\begin{equation}
    n \geq \nLowerSVariant \label{eqn::theoremSevenN},
\end{equation}
\whereMuBarIsDefined.
Then for any matrix $A \in \R^{n \times d}$ with rank $r$, an $HRHT$ matrix $S \in \R^{m\times n}$ with an $s$-hashing variant matrix  $S_h$,  is an $\epsilon$-subspace embedding for $A$ with probability at least $\HRHTProb$. 
\end{theorem}

The proof of Theorem \ref{thm::HRHT} and Theorem \ref{thm::HRHT_variant} relies on the analysis in \cite{Tropp:wr} of Randomised-Hadamard-Transforms, which are shown to reduce the coherence of any given matrix with high probability.

\begin{lemma} \cite{Tropp:wr} \label{thm::Tropp_HD}
Let $r\leq n \in \N^+$ and $U \in \R^{n \times r}$ have orthonormal columns. Suppose that $H, D$ are defined in Definition \ref{def::HRHT} and $\delta_1 \in (0,1)$. Then $$\mu(HDU) \leq \sqrt{\frac{r}{n}} + \sqrt{\frac{8\log(n/\delta_1)}{n}} $$ with probability at least $1-\delta_1$.
\end{lemma}

We are ready to prove Theorem \ref{thm::HRHT}.

\begin{proof}[Proof of Theorem \ref{thm::HRHT} and Theorem \ref{thm::HRHT_variant}]
Let $A = U\Sigma V$ be defined in \eqref{thin-SVD}, $S = S_h HD$ be an HRHT matrix. Define the following events:
\begin{itemize}[topsep=0pt,itemsep=-1ex,partopsep=1ex,parsep=1ex]
    \item $B_1 = \left\{ \mu(HDU) \leq \muLower \right\}$,
    \item $B_2 = \left\{ \mu(HDA) \leq \muLower \right\}$, 
    \item $B_3 = \left\{ \mu(HDA) \leq \muHatSEpsDelta \right\}$,
    \item $B_4 = \left\{ \text{$S_h$ is an $\epsilon$-subspace embedding for $HDA$ } \right\}$,
    \item $B_5 = \left\{ \text{$S_h HD$ is an $\epsilon$-subspace embedding for $A$} \right\}$,
\end{itemize}
where $\muHatSEpsDelta = \muASHashing$ if $S_h$ is an $s$-hashing matrix and $\muHatSEpsDelta = \muASHashingVariant$ if $S_h$ is an $s$-hashing variant matrix, and \whereMuBarIsDefined.

Observe that $B_4$ implies $B_5$ because $B_4$ gives 
\begin{equation}
    (1-\epsilon) \|Ax\|^2 \leq (1-\epsilon)\|HDA x\|^2 \leq \| S_h HDA x\|^2 \leq (1+\epsilon) \|HDA x\|^2 \leq (1+\epsilon) \|Ax\|^2,
\end{equation}
where the first and the last equality follows from $HD$ being orthogonal. Moreover, observe that $B_1 = B_2$ because $\mu(HDA) = \max_i \| (HDU)_i \|_2 = \mu(HDU)$, where the first equality follows from $HDA = (HDU) \Sigma V^T$ being an $SVD$ of $HDA$. Furthermore, $B_2$ implies $B_3$ due to \eqref{eqn::theoremSixN} in the $s$-hashing case; and \eqref{eqn::theoremSevenN} in the $s$-hashing variant case.

Thus $\P(B_5) \geq \P(B_4) = \P(B_4 | B_3) \P(B_3) \geq P(B_4|B_3) \P(B_2) = \P(B_4 |B_3) \P(B_1)$. If $S_h$ is \anSHashingMat, Theorem \ref{thm::s-hashing} gives $\P(B_4 | B_3) \geq 1-\delta$. If $S_h$ is \anSHashingVariantMat, Theorem \ref{thm::s-hashing-variant} gives $\P(B_4 | B_3) \geq 1-\delta$. Therefore in both cases, we have
\begin{equation}
     P(B_5) \geq \P(B_4 |B_3) \P(B_1) \geq (1-\delta) \probability{B_1} \geq (1-\delta)(1-\delta_1),
\end{equation}
where the third inequality uses Lemma \ref{thm::Tropp_HD}.
\end{proof}

\section{Algorithmic framework and analysis for linear least squares}
\label{sec:algo_analysis}
We now turn our attention to the LLS problem \eqref{LLS-statement} we are interested in solving. 
Building on the Blendenpik \cite{doi:10.1137/090767911} and LSRN \cite{Meng:2014ib} techniques, we introduce a generic algorithmic framework 
for \eqref{LLS-statement} that can employ any rank-revealing factorization of $SA$, where $S$ is a(ny) sketching matrix; we then analyse its convergence.

\begin{alg}[Generic Sketching Algorithm]\label{alg1}
Given $A \in \R^{n\times d}$ and $b \in \R^n$, set positive integers $m$ and $it_{max}$, and accuracy tolerances $\tau_a$ and $ \tau_r $, and an $m\times n$ random matrix distribution $\cal{S}$.
\begin{enumerate}[topsep=0pt,itemsep=-1ex,partopsep=1ex,parsep=1ex]
	\item Randomly draw a sketching matrix $S \in \R^{m\times n}$ from $\cal{S}$, compute the matrix-matrix product $SA \in \R^{m\times d}$ and the matrix-vector product $Sb \in \R^{m}$.

	\item Compute a factorization of $SA$ of the form, 
		\begin{align}
		SA= QR\hat{V}^T, \label{SA-QRV-fac}
		\end{align}
		where 
			\begin{itemize}
				\item $R = 
					\left( \begin{matrix} R_{11} & R_{12} \\ 0 & 0 	\end{matrix} 
					 \right) \in \R^{d \times d}$, where $R_{11} \in \R^{\k \times \k}$ is nonsingular.

				\item $Q =\left( \begin{matrix} Q_{1} & Q_{2} \end{matrix} \right) \in\R^{m \times d} $, where $Q_1 \in \R^{m \times \k}$ and $Q_2 \in \R^{m \times (d-\k)}$ have orthonormal columns.

				\item $\hat{V}= \left(  \begin{matrix} V_{1} \quad  V_{2}\end{matrix} \right)\in \R^{d \times d}$ is an orthogonal matrix with $V_1\in \R^{d \times \k}$.
			\end{itemize}

	\item Compute $x_{s} = V_1 R_{11}^{-1} Q_1^TSb$. If $\|Ax_{s}-b\|_2 \leq \tau_a$, terminate with solution $x_s$.

	\item 
	Else, iteratively, compute 
\begin{equation}\label{ytau}	 
	 y_{\tau} \approx \argmin_{ y \in \R^{\k}} \|Wy - b\|_2,
	 \end{equation}
	 where 
	  \begin{equation}
	      W = A V_1 R_{11}^{-1},\label{def::W}
	  \end{equation}
	  using LSQR \cite{10.1145/355984.355989} with (relative) 
	 tolerance $\tau_r$ and maximum iteration count $it_{max}$. 	  

	\item Return  $x_{\tau} = V_1 R_{11}^{-1} y_{\tau}$. 
	
\end{enumerate}
\end{alg}

\begin{remark}
\begin{itemize}
\item[(i)]
The factorization $SA = QR\hat{V}^T$ allows column-pivoted QR, or other rank-revealing factorization, complete orthogonal decomposition ($R_{12}=0$) and the SVD ($R_{12} = 0$, $R_{11}$ diagonal). It also includes the usual QR factorisation if $SA$ is full rank; then the $R_{12}$ block is absent. 

\item[(ii)]						Often in  implementations, the factorization \eqref{SA-QRV-fac} has $R=\left( \begin{matrix} R_{11} & R_{12} \\ 0 & R_{22}  \end{matrix} \right)$, where $R_{22} \approx 0$ and is treated as the zero matrix.

\item[(iii)] For computing $x_s$ in Step 3, we note that in practical implementations, $R_{11}$ in \eqref{SA-QRV-fac}  is upper triangular, enabling efficient calculation
	 of matrix-vector products involving $R_{11}^{-1}$; then, there is no need to form/calculate $R_{11}^{-1}$ explicitly.
	 \item[(iv)] For the solution of \eqref{ytau},
	 we use the termination criterion $\|y-y_*\|_{W^TW} \leq \tau_r \|y-y_* \|_{W^T W}$ in the theoretical analysis, 
	 where $y_*$ is defined in \eqref{def-ystar}. In practical implementations different termination criteria need to be employed (see 
	 Section \ref{subsec::implemntation_of_alg1}).
\end{itemize} 
\end{remark}

	\subsection{Analysis of Algorithm \ref{alg1}}
	Given problem (\ref{LLS-statement}), we denote its minimal Euclidean norm solution as follows
	\begin{equation}\label{def-xstar}
	x_{*,2} = \argmin_{x^*\in \R^d} \|x_*\|_2\quad{\rm subject\,\, to}\quad \|Ax_*-b\|_2 =\min_x \|Ax-b\|_2.
	\end{equation}
	and let
	\begin{equation}\label{def-ystar}
	 y_* = \argmin_{y\in \R^\k} \| Wy-b \|_2, \quad \text{where $W$ is defined in \eqref{def::W}}.
	\end{equation}

The following two lemmas provide basic properties of Algorithm \ref{alg1}. Their proofs can be found in the Appendix. 
\begin{lemma}
\label{Lemma::w}
$W\in\R^{n\times \k}$ defined in \eqref{def::W} has full rank $\k$.
\end{lemma}

\begin{lemma}
\label{Lemma::p_equals_r}
In Algorithm \ref{alg1}, if $S$ is an $\epsilon$-subspace embedding for $A$ for some $\epsilon\in (0,1)$, then $\k = r$ where $r$ is the rank of $A$. 
\end{lemma}

	  If the LLS problem (\ref{LLS-statement}) has a sufficiently small optimal residual, then
	Algorithm \ref{alg1} terminates early in Step 3 with the solution $x_s$ of the sketched problem $\min \|SAx-Sb\|_2$; then,
	 no LSQR iterations are required.
	

\begin{lemma}[Explicit Sketching Guarantee] \label{explicit-sketching-guarantee}
Given problem \eqref{LLS-statement},
suppose that the matrix $S \in \R^{m\times n}$ in Algorithm \ref{alg1} is an $\epsilon$-subspace embedding for the augmented matrix $\left(A\; \;b \right)$ for some $0<\epsilon<1$. Then 
\begin{align}\label{explicit-sketch}
\|Ax_s - b\|_2 \leq \frac{1+\epsilon}{1-\epsilon} \|Ax_*-b\|_2,
\end{align}
where $x_s$ is defined in Step 3 of Algorithm \ref{alg1} and $x_*$ is a(ny) solution of \eqref{LLS-statement}.
\end{lemma}

The proof is similar to the result in \cite{10.1561/0400000060} that shows that any solution of the sketched problem $\min_x \|SAx-Sb\|_2$ satisfies \eqref{explicit-sketch}.  For completeness, the proof is included in the appendix. 
		
The following technical lemma is needed in the proof of our next theorem.
		\begin{lemma}\label{VA-cap}
Let $A \in \R^{n\times d}$ and $V_1 \in \R^{d \times \k}$ be defined in Algorithm \ref{alg1}. Then
$\ker(V_1^T) \cap {\rm range}(A^T) = \{0\}$, where $\ker(V_1^T)$ and ${\rm{range}}(A^T)$ denote the null space of $V_1^T$ and range subspace 
generated by the rows of $A$, respectively.
\end{lemma}

Theorem \ref{Implicit-sketching-guarantee} shows that when the LSQR algorithm in Step 4 converges, Algorithm \ref{alg1} returns a minimal residual solution of
\eqref{LLS-statement}.

\begin{theorem}[Implicit Sketching Guarantee]\label{Implicit-sketching-guarantee}
Given problem \eqref{LLS-statement},
suppose that the matrix $S \in \R^{m\times n}$ in Algorithm \ref{alg1} is an $\epsilon$-subspace embedding  for the augmented matrix $\left(A\; \;b \right)$ for some $0<\epsilon<1$. 
If  $y_{\tau} = y_*$ in Step 4 of Algorithm \ref{alg1} (by setting $\tau_r := 0$), where $y_*$ is defined in \eqref{def-ystar},  then $x_{\tau}$ in Step 5 satisfies
$x_{\tau}=x_*$, where $x_*$ is a solution of \eqref{LLS-statement}.
\end{theorem}

\begin{proof}
Using the optimality conditions (normal equations) for the LLS in \eqref{def-ystar}, and $y_{\tau} = y_*$, we deduce 
$W^T Wy_{\tau} = W^T b$, 
where $W$ is defined in \eqref{def-ystar}. 
Substituting the definition of $x_{\tau}$ from Step 5 of Algorithm \ref{alg1}, 
we deduce
$$
(R_{11}^{-1})^TV_1^TA^TAx_{\tau}=(R_{11}^{-1})^TV_1^TA^Tb.
$$
Multiplying the last displayed equation by $R_{11}^T$, we obtain
\begin{align}
V_1^T \left( A^TA x_{\tau} - A^T b \right)  = 0. \label{tmp8}
\end{align}
It follows from
(\ref{tmp8}) that $A^T Ax - A^T b \in \ker(V_1^T) \cap {\rm range}(A^T)$. But Lemma \ref{VA-cap} implies that the latter set intersection only contains the origin, and so
 $A^T Ax_{\tau} - A^T b =0$; this and the normal equations for \eqref{LLS-statement} imply that $x_{\tau}$ is an optimal solution of \eqref{LLS-statement}.
\end{proof}

The following technical lemma is needed for our next result; it re-states Theorem 3.2 from  \cite{Meng:2014ib} in the context of Algorithm \ref{alg1}.

		\begin{lemma} \cite{Meng:2014ib} \label{Meng-min-norm}
		Given problem \eqref{LLS-statement}, let $x_{*,2}$ be its minimal Euclidean norm solution defined in \eqref{def-xstar} and $P\in\R^{d\times \k}$, a nonsingular matrix. 	Let $x_{\tau}:=Py_{\tau}$, where $y_{\tau}$ is assumed to be the minimal Euclidean norm solution  of $\min_{y\in \R^\k}\|APy-b\|_2$.
		Then $x_{\tau}=x_{*,2}$ if ${\rm range}(P)={\rm range}(A^T)$.	
\end{lemma}

Theorem \ref{tmp9} further guarantees that if $R_{12}=0$ in \eqref{SA-QRV-fac} such as when a complete orthogonal factorization is used, then the minimal Euclidean norm solution of \eqref{LLS-statement} is obtained. 

\begin{theorem}[Minimal-Euclidean Norm Solution Guarantee] \label{tmp9}
Given problem \eqref{LLS-statement},
suppose that the matrix $S \in \R^{m\times n}$ in Algorithm \ref{alg1} is an $\epsilon$-subspace embedding  for the augmented matrix $\left(A\; \;b \right)$ for some $0<\epsilon<1$. 
If  $R_{12}=0$ in \eqref{SA-QRV-fac} and $y_{\tau} = y_*$ in Step 4 of Algorithm \ref{alg1} (by setting $\tau_r := 0$), where $y_*$ is defined in \eqref{def-ystar}, 
then $x_{\tau}$ in Step 5 satisfies
$x_{\tau}=x_{*,2}$, where $x_{*,2}$ is the minimal Euclidean norm solution \eqref{def-xstar} of \eqref{LLS-statement}.

\end{theorem}

\begin{proof}
The result follows from Lemma \ref{Meng-min-norm} with $P:=V_1R_{11}^{-1}$, provided  ${\rm range}(V_1 R_{11}^{-1}) = {\rm range}(A^T)$. To see this, note that 
\[
{\rm range}(V_1 R_{11}^{-1}) = {\rm range}(V_1) = {\rm range}( (SA)^T),
\]
where the last equality follows from $(SA)^T=V_1R_{11}^TQ_1^T+V_2R_{12}Q_1^T$ and $R_{12}=0$. Using the SVD decomposition \eqref{thin-SVD} of $A$, we further have
\[
{\rm range}(V_1 R_{11}^{-1})  = {\rm range}( A^T S^T) ={\rm range}( V\Sigma U^T S^T) ={\rm  range}(V\Sigma (SU)^T).
\]
Since $S$ is an $\epsilon-$subspace embedding for $A$, it is also an $\epsilon$-subspace embedding for $U$ by Lemma \ref{subspace-embedding-def-2} and therefore by Lemma \ref{rank-of-sketched-equal-to-unsketched}, $\rank(SU) = \rank(U) = r$. Since $S U \in \R^{m\times r}$ has full column rank, we have that ${\rm range}(V\Sigma (S U)^T) = {\rm range}(V) = {\rm range}(A^T)$. 
\end{proof}

Theorem \ref{Speed-of-convergence} gives an iteration complexity bound for the inner solver in Step 4  of Algorithm \ref{alg1}, as well as particularising this result for a special starting point for which  an optimality guarantee can be given. It relies crucially on the quality of the preconditioner provided by the sketched factorization in \eqref{SA-QRV-fac}, and its proof, that uses standard LSQR results, is given in the appendix.

\begin{theorem}[Rate of convergence] \label{Speed-of-convergence}
Given problem \eqref{LLS-statement},
suppose that the matrix $S \in \R^{m\times n}$ in Algorithm \ref{alg1} is an $\epsilon$-subspace embedding  for the augmented matrix $\left(A\; \;b \right)$ for some $0<\epsilon<1$. 
Then: 
\begin{itemize}
\item[(i)]
Step 4 of Algorithm \ref{alg1} takes at most
\begin{align}\label{LSQR_iterations}
\tau \leq O\left( \frac{|\log\tau_r|}{|\log\epsilon|}\right)
\end{align}
LSQR iterations to return a solution $y_{\tau}$ such that 
\begin{equation}\label{LSQR-tc}
\|y_{\tau}- y_*\|_{ W^T W} \leq \tau_r \|y_0 - y_*\|_{ W^T W}, 
\end{equation}
where $y_*$ and $W$ are defined in \eqref{def-ystar}.
\item[(ii)] 
If we initialize $y_0 : = Q^TSb$ for the LSQR method in Step 4, then at termination of Algorithm \ref{alg1}, we can further guarantee that
\begin{align}
		\|Ax_{\tau} - b\|_2 \leq \left( 1 + \frac{2\epsilon \tau_r}{1-\epsilon} \right) \|Ax_*-b\|_2,
		\end{align}
		where $x_{\tau}$ is computed in Step 5 of Algorithm \ref{alg1} and $x_*$ is a solution of \eqref{LLS-statement}.

\end{itemize}
\end{theorem}

\section{The Ski-LLS solver for linear least squares}
\label{subsec::implemntation_of_alg1}
\textbf{SK}etch\textbf{i}ng-for-\textbf{L}inear-\textbf{L}east-\textbf{S}quares (\solverName{})\footnote{Available at: 
https://github.com/numericalalgorithmsgroup/Ski-LLS}
is a C++ implementation of \autoref{alg1} for solving problem \eqref{LLS-statement}. We distinguish two cases based on whether the data matrix $A$ in \eqref{LLS-statement} is stored as a dense matrix or a sparse matrix. 

\paragraph{Dense matrix input}
When $A$ in \eqref{LLS-statement} is stored as a dense matrix \footnote{Namely, we assume that sufficiently many entries in $A$ are nonzero that  specialised, sparse numerical linear algebra techniques are ineffective.}, we employ the following implementation of \refAlgOne. We refer here to the resulting Ski-LLS variant as {\bf \solverNameDense{}}.
\begin{enumerate}
    \item In Step 1 of \refAlgOne, we let 
    \begin{equation}
        S = S_h F D, \label{eqn::HR-DHT}
    \end{equation}
    where
        \begin{enumerate}
            \item $D$ is a random $n \times n$ diagonal matrix with $\pm1$ independent entries, as in \autoref{def::HRHT}.
            \item F is a matrix representing the normalized Discrete Hartley Transform (DHT), defined as $F_{ij} = \sqrt{1/n} \squareBracket{ \cos{ \bracket{ 2\pi (i-1)(j-1)/n}} + \sin{ \bracket{2 \pi (i-1)(j-1)/n}}}$ \footnote{Here we use the same transform (DHT) as that in Blendenpik for comparison of other components of \solverNameDense{}, instead of the Walsh-Hadamard transform defined in \autoref{def::HRHT}.}. We use the (DHT) implementation in FFTW 3.3.8 \footnote{Available at http://www.fftw.org.}.  
            \item $S_h$ is an $s$-hashing matrix, defined in \autoref{def::sampling_and_hashing}. We use the sparse matrix-matrix multiplication routine in SuiteSparse 5.3.0 \footnote{Available at https://people.engr.tamu.edu/davis/suitesparse.html.} to compute $S_h  \times (FDA)$. 
        \end{enumerate}
    
    \item In Step 2 of \refAlgOne, the default factorization for the sketched matrix $SA$ is the Randomised Column Pivoted QR (R-CPQR)\footnote{The implementation can be found at https://github.com/flame/hqrrp/. The original code only has a 32-bit integer interface. We wrote a 64-bit integer wrapper as our code has 64-bit integers.} proposed in \cite{Martinsson:2017eh,martinsson2015blocked}. When $A$ is (assumed to be) full rank, we allow a usual QR factorization as in Blendenpik, using DGEQRF from LAPACK for its computation, with the same default values, except the parameter ${rcond}$ (defined below) is absent/not needed.

    \item In Step 3 of \refAlgOne, since $R_{11}$ from R-CPQR is upper triangular, we do not explicitly compute its inverse, but instead, use back-solve from the LAPACK provided by Intel MKL 2019 \footnote{See https://software.intel.com/content/www/us/en/develop/tools/oneapi/components/onemkl.html.}. 
    
    \item In Step 4 of \refAlgOne, we use the LSQR routine\footnote{Available at https://web.stanford.edu/group/SOL/software/lsrn/. We fixed some very minor bugs in the code. } to solve \eqref{ytau}, as implemented in LSRN \cite{Meng:2014ib}.
   The solution of  \eqref{ytau} terminates
   when
   \begin{equation}\label{TC-lsqr}
   \frac{\|W^T (Wy_k - b)\|}
			{\|W\|\cdot\|Wy_k-b\|} \leq \tau_r
			\end{equation}
			where $W$ is defined in \eqref{def::W}; see Section 6 in \cite{10.1145/355984.355989} for a justification of this termination condition.
\end{enumerate}

The choice of sketching matrix \eqref{eqn::HR-DHT}  in Step 1 of 
\solverNameDense{} is novel. It is a hashing variant of the subsampled DHT (SR-DHT, \eqref{eq::SR-DHT}) choice in 
Blendenpik. In Theorem \ref{thm::HRHT}, we made the case that using hashing -- instead of sampling -- with the randomised Walsh-Hadamard transform (HRHT) has optimal embedding properties  in terms of the sketching dimension $m$. Here, we aim to show the numerical gains in running time of using hashing (instead of sampling) with 
coherence-reduction transformations. We focus on an DHT variant
(rather than RHT) as the former is more flexible (in terms of allowing any values of $n$), stable and faster according to experience in
\cite{doi:10.1137/090767911}.

The solver employs the following user-chosen parameters, already introduced in \autoref{alg1}: $m$, number of rows of sketching matrix $S$ (default is $1.7d$); $s$, number of nonzero entries per column in $S$ (default is $1$); accuracy tolerances ${\tau_a}$ (default is $10^{-8}$)
and $\tau_r$ (default is $10^{-6}$); maximum iteration count ${it_{max}}$ (default value is $10^4$). 


Two more parameters are needed:
\begin{itemize}
    \item 
${rcond}$ (default value is $10^{-12})$, which is a parameter used in Step 2 of Algorithm 1 as follows. R-CPQR computes \eqref{SA-QRV-fac} as
            $SA = Q \tilde{R} \hat{V}^T$, which is then used to compute $\tilde{R}_{11}$,   the upper left $p \times p$ block of $\tilde{{R}}$ by
        letting $p = \max \set{q: \tilde{r}_{q,q}\geq rcond}$, where 
        $\tilde{r}_{q,q}$ is the $q$th diagonal entry of $\tilde{R}$.

\item $wisdom$ (default value is $1$). The DHT routine we use is faster with pre-tuning, see Blendenpik \cite{Avron:2009aa} for a detailed discussion. If the DHT has been pre-tuned, the user needs to set $wisdom=1$; otherwise  $wisdom=0$.  In all our experiments, the default is to tune the DHT using the simplest tuning mechanism offered by FFTW (which in our experiments has been very efficient). The cost of this tuning is not included in the reported runs here. 
\end{itemize}


\paragraph{Sparse matrix input}
When $A$ is stored as a sparse matrix \footnote{Namely,  the user stores the data in a sparse matrix format, which is maintained through the algorithm's run for efficient linear algebra procedures.}, we employ the following implementation of \refAlgOne. We refer here to the resulting Ski-LLS variant as {\bf \solverNameSparse{}}.
\begin{enumerate}
    \item In Step 1 of \refAlgOne, we let $S$ be an $s$-hashing matrix, as  in \autoref{def::sampling_and_hashing}.
    \item In Step 2 of \refAlgOne, we use the sparse QR factorization (SPQR) proposed in \cite{10.1145/2049662.2049670} and implemented in SuiteSparse.
    \item In Step 3 of \refAlgOne, since $R_{11}$ from SPQR is upper triangular, we do not explicitly compute its inverse, but instead, use the sparse back-substitution routine from SuiteSparse.
    \item In Step 4 of \refAlgOne, we use the LSQR routine as implemented in LSRN, extended to include the employment of a sparse preconditioner and sparse numerical linear algebra tools from SuiteSparse.  Again, we terminate the solution of  \eqref{ytau} 
   when \eqref{TC-lsqr} holds.
\end{enumerate}

The solver employs the following user-chosen parameters, already introduced in \autoref{alg1}:
 $m$ (default value is $1.4d$), $s$ (default value is $2$), $\tau_a$ (default value is $10^{-8}$), $\tau_r$ (default value is $10^{-6}$), $it_{max}$ (default value is $10^4$). We also need 
 \begin{itemize}
     \item 
$ordering$ (default value $2$) which is a parameter of the SPQR routine that influences the permutation matrix $\hat{V}$ and the sparsity of $R$. \footnote{Note that this is slightly different from the SPQR default, which  uses COLAMD if m2<=2*n2; otherwise switches to AMD. Let f be the flops for chol((S*P)’*(S*P)) with the ordering P found by AMD. Then if f/nnz(R) $\geq$ 500 and nnz(R)/nnz(S) $\geq$ 5 then use METIS, and take the best ordering found (AMD or METIS), where typically $m_2 = m$, $n_2 = n$ for $SA \in \R^{m \times n}$. By contrast, \solverName{} by default always uses the AMD ordering.}.
\item $rcond_{thres}$ (default value $10^{-10}$), which checks the conditioning of $R_{11}$ in \eqref{SA-QRV-fac} computed by SPQR. If the condition number  $\kappa(R_{11}) \geq 1/rcond_{thres}$, we use the perturbed back-solve for upper triangular linear systems involving $R_{11}$  with $perturb$ (default value $10^{-10}$).  In particular, any back-solve involving $R_{11}$ or its transpose is modified as follows: when divisions by a diagonal entry $r_{ii}$ of $R_{11}$ is required ($1\leq i \leq p$), we divide by $r_{ii} + perturb$ instead\footnote{This is a safeguard when SPQR fails to detect the rank of $A$. This happens infrequently \cite{10.1145/2049662.2049670}.}.
 \end{itemize}

 A few comments regarding our Ski-LLS implementation 
 are in order, that connect earlier theoretical developments to our implementation choices. 

\paragraph{Subspace embedding properties}
Our analysis of \refAlgOne\, in the previous section relies crucially on the sketching matrix $S$ being a subspace embedding. Theorem \ref{thm::HRHT} does not apply specifically to the coherence-reduction DHT we use in \solverNameDense{}, but it inspired us to explore the practical efficiency of using hashing instead of subsampling in the random transform used in
Blendenpik. 
For sparse matrices, Theorem \ref{thm::s-hashing} guarantees the desired embedding properties of $s$-hashing matrices when $A$ has low coherence. However as Figures \ref{fig::Ls_qr_engineering_time} and \ref{fig::Ls_qr_engineering_residual} indicate, these results may not be tight in that $s$-hashing matrices with $s>1$ and $m=\mathcal{O} \bracket{d}$ also seem to embed correctly matrices $A$ with high(er) coherence.
In both the dense and sparse Ski-LLS, numerical calibration is used to determine the default value of $m$ -- and in the sparse case also $s$ --  such that $\epsilon$-subspace embedding of $A$ is achieved with sufficiently high  probability (namely, for all the matrices in the calibration set); see \ref{Implementation-subsection} for details.

Even if the distribution of hashing matrices provides an oblivious subspace embedding, there is still a (small) positive probability that for a given $A$, a randomly drawn $S$ fails to sketch $A$ accurately. 
Then, if  $A$ has full rank, \solverName{} can still compute  an accurate solution of $\eqref{LLS-statement}$ as   the preconditioner $V_1 R_{11}^{-1}$ is a nonsingular  matrix.
When $A$ is rank deficient, and $m$ has been calibrated as mentioned above, we have not found detrimental evidence of solver failure in that case, just possibly some loss of accuracy; see for example, the Florida Collection test problem 'landmark' in Table 2 where the sketched input matrix has lower rank,
indicating failure of subspace embedding property, and resulting in slightly lower accuracy in the final residual value.

\paragraph{Sketching size tuning to ensure subspace embedding properties}
\label{Implementation-subsection}

We use numerical calibration  to determine the default value of $m$ for \solverNameDense{} (with and without R-CPQR), as well as the values of both $m$ and $s$ for \solverNameSparse{}. These choices are made 
such that $\epsilon$-subspace embedding properties are achieved with sufficiently high  probability  for all the random matrices in the respective calibration set (see Test Set 1 and 2 next).  Note for example,  the typical U-shaped curve in Figure \ref{fig::new_blen_engineering}:  as $\gamma:= m/d$ grows, we have better subspace embedding properties with smaller $\epsilon$ and hence, fewer LSQR iterations (due to \eqref{LSQR_iterations}). However the cost of sketching in Ski-LLS' Step 1 and of the factorization  in Step 2 will grow as $m$ grows. Thus a trade-off is achieved for an approximately median value of $m$; see Appendices B (dense input) and C (sparse input) for the plots of the calibration results for $m$, and $s$ when needed.

\paragraph{Approximate factorization in Step 2}
In both  dense and sparse Ski-LLS, the factorization in Step 2  is not guaranteed to be accurate when $A$ is rank-deficient. This is because R-CPQR, like column-pivoted QR \cite{10.5555/248979}, does not guarantee detection of rank. However,    the numerical rank is typically correctly determined, in the sense that using the procedure described in the definition of the parameter $rcond$, the factorization of  $SA$ will be accurate approximately up to  $rcond$ error. Similarly, SPQR only performs heuristical rank detection, without guarantees. Our encouraging numerical results testing for accuracy, alleviate these shortcomings;
see Section \ref{Acc-subsection}.



\section{Numerical experiments}

\subsection{Test sets}
The following linear least squares problems of the form  \eqref{LLS-statement} are used to test and benchmark Ski-LLS.

\paragraph{The vector $b$} In each test problem, the vector $b \in \R^n$ in \eqref{LLS-statement} is chosen to be a vector of all ones.

\paragraph{The input matrix $A$}
\begin{enumerate} \label{List_of_test_set}
	\item (Test Set 1) Three different types of {\bf random dense matrices} as used by Avron et al \cite{doi:10.1137/090767911} to compare Blendenpik with the LAPACK least square solver. They have different \singleQuote{non-uniformity} of the rows.
		\begin{enumerate}
			\item Incoherent dense type, defined by \eqref{thin-SVD} with $r=d$, and where 
			  $U \in \R^{n \times d}$, $V \in \R^{d \times d}$ are  generated by orthogonalising columns of two independent matrices with i.i.d. N(0,1) entries, and $\Sigma \in \R^{d \times d}$ is a diagonal matrix with diagonal entries equally spaced from $1$ to $10^6$.

			\item Semi-coherent dense type, defined by
			    \begin{equation}
			      A = \left( \begin{matrix} B & 0 \\ 0 & I_{\frac{d}{2}} \end{matrix} \right) + 
			10^{-8}  J \in \R^{n \times d}, \label{eq::A_semi_dense}  
			    \end{equation}
			 where $B$ is an incoherent dense matrix as in 1(a) and $J \in \R^{n \times d}$ is a matrix of all ones.

			\item Coherent dense type, defined by
			    \begin{equation}
			        A = \left(\begin{matrix} I_{d\times d} \\ 0  \end{matrix}\right) + 10^{-8}J \in \R^{n\times d} , \label{eq::A_co_dense}
			    \end{equation}
			where $J$ is again a matrix of all ones. 
		\end{enumerate}
	
	\item (Test Set 2) The following are three different types of {\bf random sparse matrices} with different \singleQuote{non-uniformity} of rows.
	    \begin{enumerate}
	        \item Incoherent sparse type, defined by 
                    \begin{equation}
                        A = \text{sprandn}(n, d, 0.01, 1e-6), \label{eq:A_inco_sparse}
                    \end{equation}
                    where \singleQuote{sprandn} is a command in MATLAB that generates a matrix with approximately $0.01nd$ normally distributed non-zero entries and a condition number approximately equal to $10^6$. 
            \item Semi-coherent sparse type, defined by
                \newcommand{\dTemp}{\hat{D}}
                \begin{equation}
                    A = \dTemp^5 B, \label{eq:A_semi-co_sparse} 
                \end{equation}
                where $B \in \R^{n\times d}$ is an incoherent sparse matrix defined in \eqref{eq:A_inco_sparse} and $\dTemp$ is a diagonal matrix with independent $N(0,1)$ entries on the diagonal. 
            \item Coherent sparse type, defined by
                \begin{equation}
                    A = \dTemp^{20} B, \label{eq:A_co_sparse} 
                \end{equation}
                where $B \in \R^{n\times d}, \dTemp$ are the same as in \eqref{eq:A_semi-co_sparse}.
	    \end{enumerate}
	
	\item (Test Set 3) A total of 181 matrices from  the {\bf Florida (SuiteSparse) matrix collection} \cite{10.1145/2049662.2049663} such that 
		if the matrix is under-determined, we transpose it to make it over-determined; and 
			once transposed,  $A \in \R^{n \times d}$ must have $n \geq 30000$ and $n \geq 2 d$.
\end{enumerate}

\begin{remark}
Note that the  \singleQuote{coherence} in Test Sets 1 and 2 above is slightly different than the earlier notion of coherence $\mu$ in \eqref{def::coherence}. In particular, the former is a more general concept, indicating that the leverage scores of $A$ are somewhat non-uniform, while the latter only predicates the value of the maximum leverage score. Although \singleQuote{coherent dense/sparse} $A$ tends to have higher values of $\mu(A)$ than that of \singleQuote{incoherent dense/sparse} $A$, the value of $\mu(A)$ may be similar for semi-coherent and coherent test matrices (namely, it is $1$ in Test Sets 1 and 2),
while their differences lie in that the 
 row norms of $A$ (and hence of $U$ in \eqref{thin-SVD}) tend to be more non-uniform than in the semi-coherent case.
The coherence terminology for the Sets 1 and 2 is derived from Blendenpik and we have maintained it for ease of comparison with earlier results. 
\end{remark}

\subsection{Solvers and their parameters}

 For {\bf dense input}, we compare  \solverNameDense{}  to the state-of-the-art sketching solver Blendenpik
\cite{doi:10.1137/090767911} \footnote{Available at https://github.com/haimav/Blendenpik. For the sake of a fair comparison, we wrote a C interface and use the same LSQR routine as in \solverName{}.}. The row dimension of the sketching matrix in Blendenpik is set to $m = 2.2d$, which we obtained  by calibrating Blendenpik as shown in  Appendix B. Note that we do {\it not} use the default value of $m$ in Blendenpik which is $4$; this default choice would yield significantly longer runtimes.  We also set $\tau_r = 10^{-6}$, $it_{max} = 10^4$ and $wisdom=1$. The same wisdom data file as in \solverName{} is used. 

For {\bf sparse input}, we compare  \solverNameSparse{}  to the following solvers:
\begin{enumerate}
	\item HSL\_MI35 (LS\_HSL), that uses an incomplete Cholesky factorization of $A$ to compute a preconditioner for problem \eqref{LLS-statement}, before using LSQR. \footnote{See http://www.hsl.rl.ac.uk/specs/hsl_mi35.pdf for a full specification. For a fair comparison, we wrote a C interface and used the same LSQR routine as in \solverName{}. We also disabled the pre-processing of the data as it was not done for the other solvers. We found that using no scaling and no ordering was more efficient than the default scaling and ordering, and so we chose the former. It may thus be possible that the performance of HSL may improve, however \cite{10.1145/3014057} experimented with the use of different scalings and orderings, providing some evidence that the improvement will not be significant.} This is a state of the art preconditioned, iterative solver  for sparse LLS problems  \cite{10.1145/3014057,Gould:2016vg}. We use $\tau_r = 10^{-6}$ and $it_{max} = 10^4$. 

	\item SPQR\_SOLVE (LS\_SPQR), that uses SPQR from Suitesparse to compute a sparse QR factorization of $A$, which is exploited to solve \eqref{LLS-statement} directly. \footnote{Available at https://people.engr.tamu.edu/davis/suitesparse.html.} This  is a state of the art direct solver for sparse LLS problems \cite{10.1145/2049662.2049670}. 

	\item LSRN, that uses the framework of \refAlgOne, with $S$ having i.i.d. $N\bracket{0, 1/\sqrt{m}}$ entries in Step 1; SVD factorization from Intel LAPACK of the matrix $SA$ in Step 2; the same LSQR routine as \solverName{} in Step 4. \footnote{Note that LSRN does not contain the Step 3.} LSRN has been shown to be  efficient  for possibly rank-deficient dense and sparse LLS problems in a parallel computing environment \cite{Meng:2014ib}. However, parallel techniques are outside the scope of this paper  and we therefore run LSRN  serially. The parameters are chosen to be $m = 1.1d$ (calibrated, random test sets) while $m$ is set to LSRN default otherwise; $\tau_r = 10^{-6}$, $it_{max}=10^4$. 
\end{enumerate}
For each solver, the algorithm parameters are set to their default values unless otherwise specified above.

\paragraph{Compilation and environment for timed experiments}
\label{subsec::compilation_and_running_env}
For our numerical studies,  unless otherwise mentioned, we use Intel C compiler icc with optimisation flag -O3 to compile all the C code, and Intel Fortran compiler ifort with -O3 to compile Fortran-based code. All code has been compiled in sequential mode and linked with sequential dense/sparse linear algebra libraries provided by Intel MKL, 2019 and Suitesparse 5.3.0. 
The machine used has Intel(R) Xeon(R) CPU E5-2667 v2 @ 3.30GHz with 8GB RAM.  
All reported times are wall clock times in seconds.

\paragraph{Performance profiles}

Performance profiles \cite{dolan2002benchmarking} are a popular tool when benchmarking software. In the performance profile to be encountered here, for some of the results, we plot the runtime ratio against the fastest solver on the horizontal axis,  in $\log 2$ scale. For each runtime ratio $a$, we have the ratio of problems in the test set $b$ on the vertical axis such that for a particular solver, the runtime ratio against the best solver is within $a$ for $b$ percent of the problems in the test set. For example, the intersect between the performance curve and the y-axis gives the ratio of the problems in the test set such that a particular solver is the fastest.

Given a problem $A$ from the test set, let $(r_1, r_2, r_3, r_4)$ be the residual values at the solutions computed by the four solvers, that we compare in the sparse case. And let $r = \min_{i\in\{1,\ldots,4\}} r_i$. A solver is declared as having failed on this particular problem if one of the following two conditions holds
\begin{enumerate}
    \item $r_i > (1+\tau_r)r$ and $r_i > r + \tau_a$, implying that the residual at the computed solution is neither relatively or absolutely sufficiently close to the residual at the best solution found by (one of the) remaining solvers.
    
    \item The solver takes more than 800 wall clock seconds to compute a solution. 
\end{enumerate}
When a solver fails, we set the runtime of the solver to be 9999 seconds on the corresponding problem. 
Thus a large runtime (ratio) could be due to either an inaccurate or an inefficient solution obtained by the respective solver.  We note that for all successful solvers, the runtime is bounded above by 800 seconds so that there can be no confusion  whether a solver is successful or not. 


\subsection{Numerical results}
We now present our numerical findings. 

\subsubsection{Numerical illustrations}

The advantages of two key choices of our implementation are illustrated next. We use Matlab for these illustrations as they do not involve runtime comparisons, nor the running of Ski-LLS. 

\paragraph{Using hashing instead of sampling with coherence-reduction transformations}
In \autoref{fig::blen_motivation}, we generate a random coherent dense matrix $A \in \R^{4000 \times 400}$ as  in \eqref{eq::A_co_dense}.  For each $m/d$ (where $d=400$), we sketch $A$ using an HR-DHT $S \in \R^{m \times n}$ defined in \eqref{eqn::HR-DHT} and using SR-DHT as in Blendenpik, defined by 
    \begin{equation}
        S = S_s FD, \label{eq::SR-DHT}
    \end{equation}
    where $S_s \in \R^{m\times n} $ is a sampling matrix, whose individual rows contain a single non-zero entry at a random column with value $1$; $F, D$ are defined the same as in \eqref{eqn::HR-DHT}.  We then compute a QR factorization without pivoting, of each sketch $SA=QR$, and the condition number of $AR^{-1}$. 

    We see that using hashing instead of sampling in randomised DHT allows the use of a smaller $m$ to reach a given preconditioning quality. 

\begin{figure}
\centering
\includegraphics[width=0.5\textwidth]{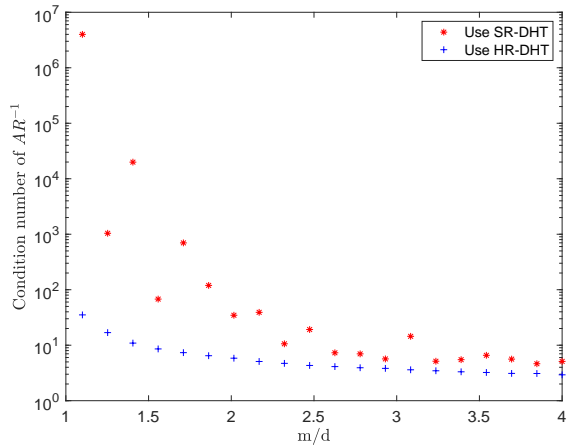}
\caption{Hashing combined with a randomised Discrete Hartley Transform (DHT) produces more accurate sketched matrix $SA$ for a given $m/d$ ratio comparing to sampling combined with a randomised DHT; the accuracy of the sketch is reflected in the quality of the preconditioner $R$ constructed from the matrix $SA$, see \eqref{Low condition number}.}
\label{fig::blen_motivation}
\end{figure}

\paragraph{Using $s$-hashing with $s>1$ to sketch sparse input}
In Figure \ref{fig::1-2-3-inco} , we let $A \in \R^{4000\times 400}$ be a random incoherent sparse matrix as in \eqref{eq:A_inco_sparse}. In Figure
\ref{fig::1-2-3-semi-co}, $A$ is defined as in \eqref{eq::A_semi_dense} 
but with $ B\in \R^{n\times d}$ being a random incoherent sparse matrix
 \footnote{We use this type of random sparse matrix instead of one of the types in Test Set 2 as this matrix better showcases the failure of $1$-hashing.}. Comparing \autoref{fig::1-2-3-inco} with \autoref{fig::1-2-3-semi-co}, we see that using $s$-hashing matrices with $s>1$ is essential in order to obtain a good preconditioner.

		\begin{figure}[H]
		    \centering
		    \begin{minipage}{0.48\textwidth}
		        \centering
		        \includegraphics[width=\textwidth]{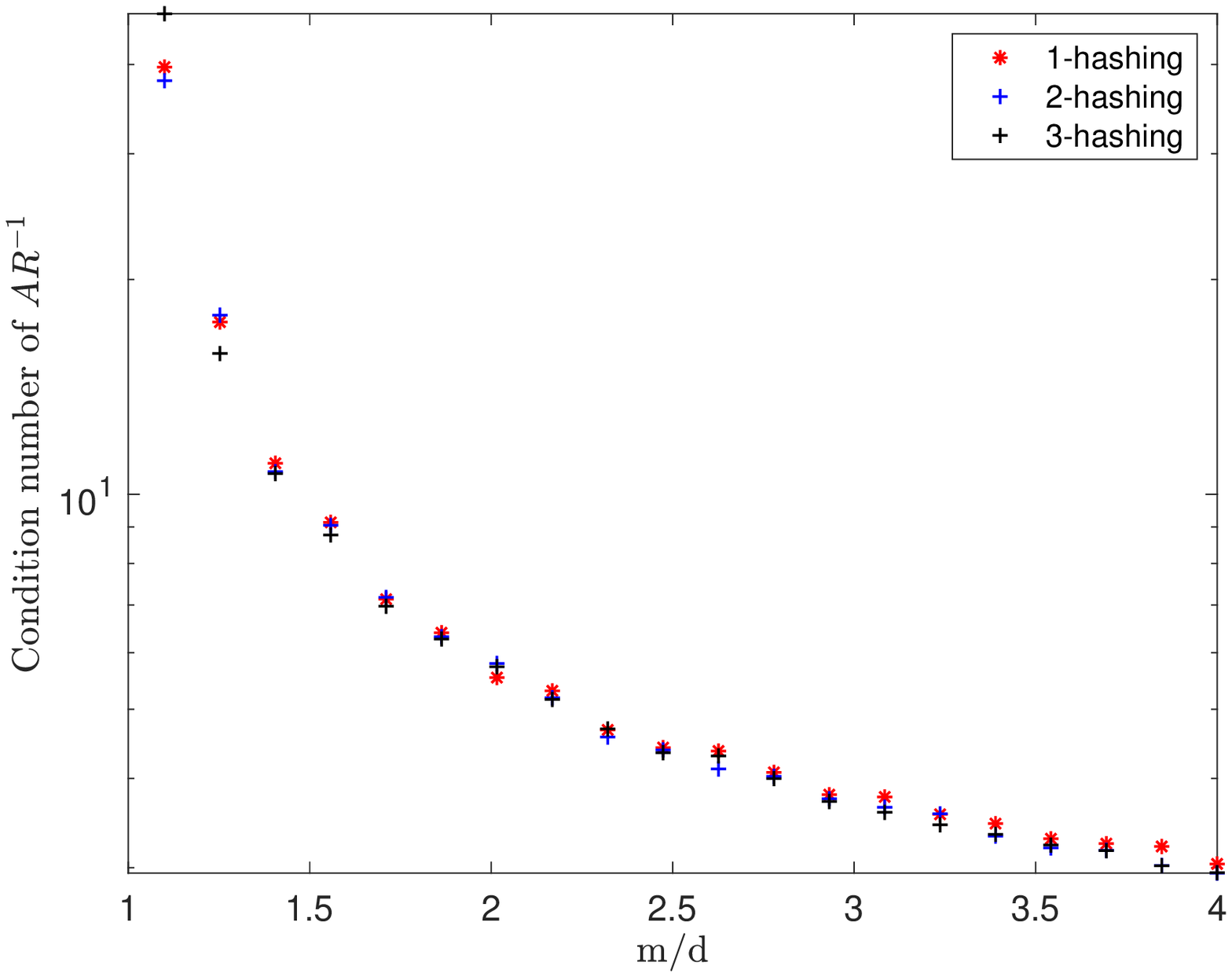} 
		        \caption{When the data matrix $A$ is an ill-conditioned sparse Gaussian, using $1,2,3$-hashing produces similarly good preconditioners.}
		        \label{fig::1-2-3-inco}
		    \end{minipage}\hfill
		    \begin{minipage}{0.48\textwidth}
		        \centering
		        \includegraphics[width=\textwidth]{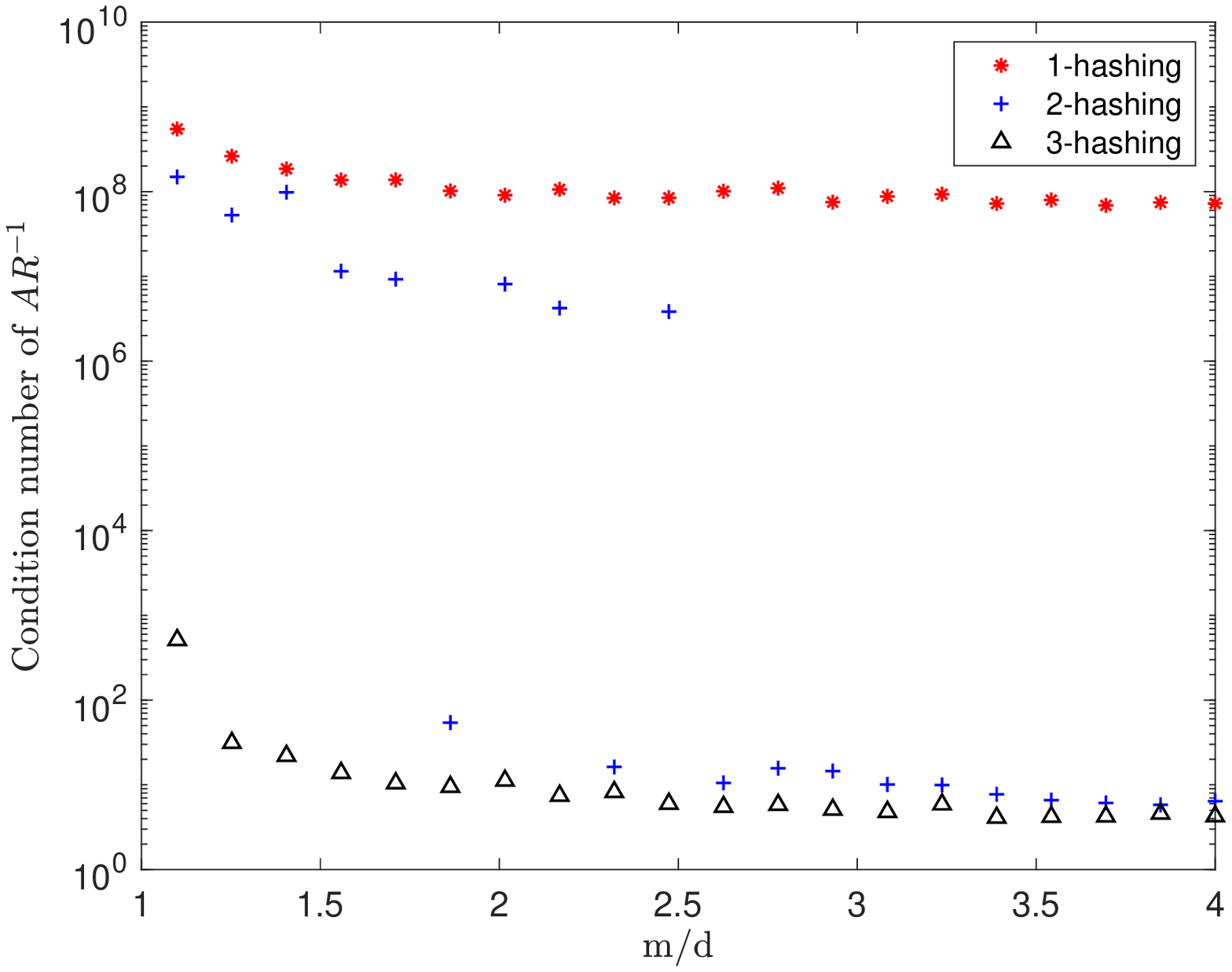} 
		        \caption{When the data matrix $A$ has higher coherence, using $s$-hashing with $s> 1$ is crucial to obtaining an acceptable preconditioner.}
		        \label{fig::1-2-3-semi-co}
		    \end{minipage}   
		\end{figure}


\subsubsection{Residual accuracy obtained by \solverName{} for some rank-deficient problems}
\label{Acc-subsection}
Here we check the final residual accuracy (objective decrease in \eqref{LLS-statement}) obtained by \solverName{}, as an illustration/check that the implementation does not depart much from the theoretical guarantees in Section \ref{sec:algo_analysis}. We choose 14 matrices  in the Florida matrix collection that are rank-deficient (Table \ref{tab::rank_def_accuracy_dim}). We use LAPACK's SVD-based LLS solver (SVD), LSRN, Blendenpik, \solverNameDense{} and \solverNameSparse{} on these problems with the residual shown in Table \ref{tab::rank_def_accuracy}.
Both  \solverNameDense{} and \solverNameSparse{} have very good residual accuracy compared to the SVD-based LAPACK solver; also, as expected, Blendenpik is unable to accurately solve rank-deficient problems.  

Furthermore, in our large scale numerical study with the Florida matrix collection in Section \ref{section:Florida_section}, we have also compared the residual values and found that the solution of \solverNameSparse{} is no-less accurate than the state-of-the-art sparse solvers LS\_SPQR and LS\_HSL.

\begin{table}[]
\center
{\scriptsize
\begin{tabular}{|l|l|l|l|l|}
\hline
                & Ski-LLS-Sparse & Ski-LLS-Dense & Blendenpik           & LSRN         \\ \hline
lp_ship12l  & 1.28E-09    & 9.84E-10  & NaN              & 1.01E-08   \\ \hline
Franz1      & 3.99E-09     & 1.55E-09   & 9405.535 & 5.18E-08  \\ \hline
GL7d26      & 3.11E-09    & 2.33E-09   & NaN             & 1.08E-08  \\ \hline
cis-n4c6-b2 & 1.24E-14   & 3.82E-14  & 174.4694  & 5.54E-15 \\ \hline
lp_modszk1  & 4.93E-09    & 4.83E-09    & NaN            & 5.48E-08  \\ \hline
rel5       & 1.29E-10     & 1.92E-10    & NaN             & 1.13E-09   \\ \hline
ch5-5-b1    & 2.20E-10   & 1.07E-10  & 11.49551   & 7.85E-10  \\ \hline
n3c5-b2     & 5.56E-15   & 1.36E-16  & 86.50651   & 1.41E-15 \\ \hline
ch4-4-b1    & 1.48E-12   & 4.30E-12  & 283.7672  & 9.77E-15 \\ \hline
n3c5-b1    & 0             & 0            & 9.185009   & 3.99E-14   \\ \hline
n3c4-b1     & 0             & 0            & 43.42909   & 0            \\ \hline
connectus   & 5.11E-03  & NaN     & NaN              & 5.11E-03  \\ \hline
landmark    & 2.83E-10   & 6.81E-16  & NaN              & 6.55E-11  \\ \hline
cis-n4c6-b3 & 2.83E-09    & 1.34E-09   & 2400.632 & 8.81E-09   \\ \hline
\end{tabular}%
}

\caption{Difference of residual values generated by the solvers  compared to the residual of the truncated-SVD solution, for a range of rank-deficient problems  from the Florida matrix collection \cite{10.1145/2049662.2049663}. The matrices, each given in sparse format, are converted into dense format before applying a dense solver such as Blendenpik. We find that \solverNameDense{} and \solverNameSparse{} achieve good residual accuracy on rank-deficient problems, as well as LSRN. Blendenpik is not designed for rank-deficient problems and either returns a large residual or encounters numerical issues/out of memory. }
\label{tab::rank_def_accuracy}
\end{table}


\subsubsection{Runtime performance of Ski-LLS on random, full-rank and dense input (Test Set 1)}
We compare \solverNameDense{} (with and without R-CPQR and with default settings) 
with (calibrated) Blendenpik (as above) on dense matrix input in \eqref{LLS-statement}. For each of the  sizes shown (on the horizontal axis) in Figures \ref{fig::compare_blen_coherent}, \ref{fig::compare_blen_semi-coherent} and \ref{fig::compare_blen_incoherent},  we generate three types of
random matrices as in Test Set 1. 
The residual values obtained by the different solvers coincide to six significant digits, indicating all three solvers give an accurate solution of \eqref{LLS-statement}.

Our results in Figures \ref{fig::compare_blen_coherent}, \ref{fig::compare_blen_semi-coherent} and \ref{fig::compare_blen_incoherent} showcase the improvements obtained  when using hashing instead of subsampling in the DHT sketching, yielding improved runtimes for \solverNameDense{} without R-CPQR compared to Blendenpik, especially when the input matrix $A$ is of coherent type, as in
 \eqref{eq::A_co_dense}. We also see that \solverNameDense{} with R-CPQR is as fast as Blendenpik on these full-rank and dense problems, while also being able to solve rank-deficient problems (see Table \ref{tab::rank_def_accuracy}). 

\newcommand{\mysize}{0.45}
\threeFigures
{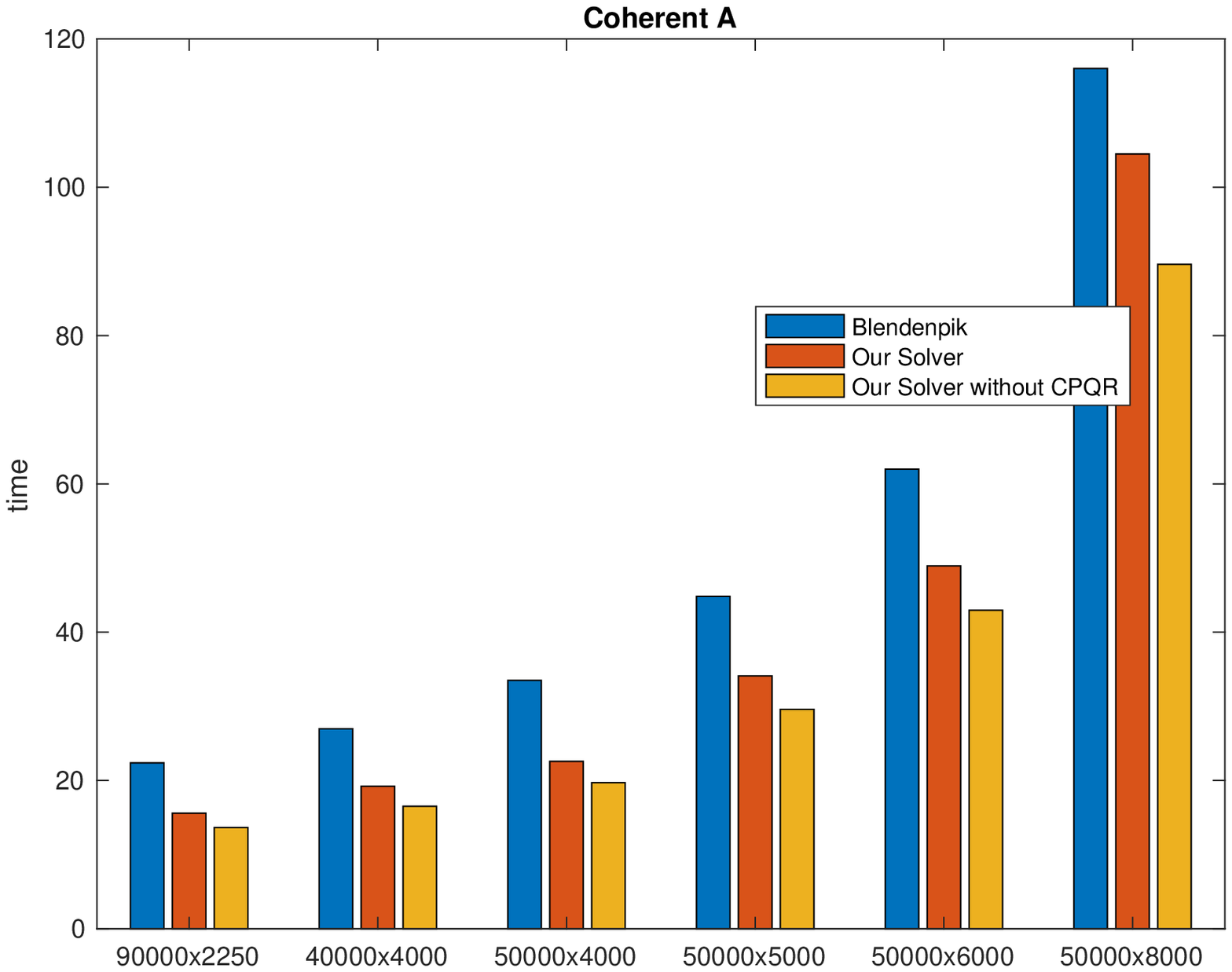}
{Time (in seconds) taken by solvers applied to problem (\ref{LLS-statement}) with coherent, dense and random  matrices  of various sizes (x-axis).}
{fig::compare_blen_coherent}
{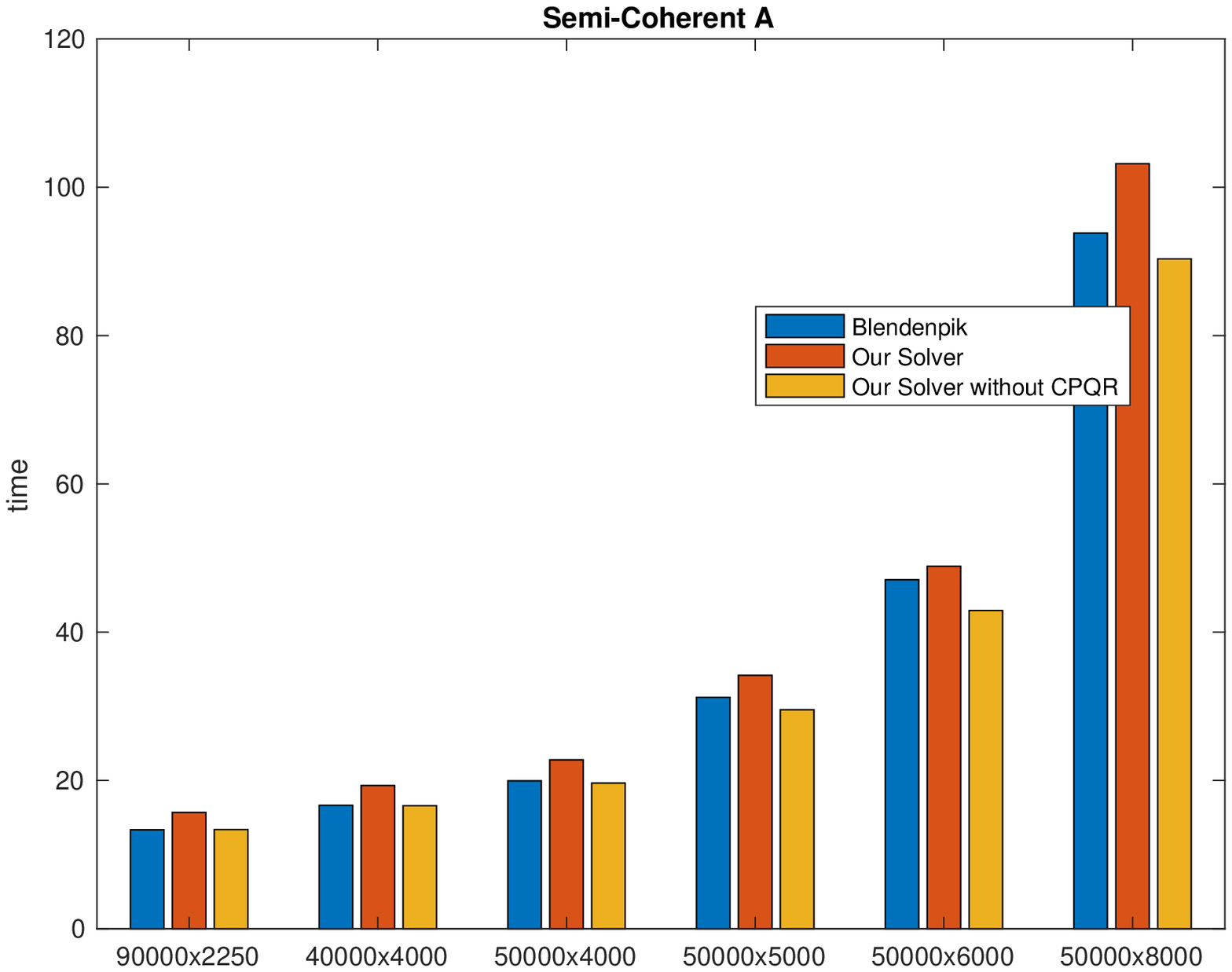}
{Time (in seconds) taken by solvers applied to problem (\ref{LLS-statement}) with semi-coherent, dense and random matrices  of various sizes (x-axis).}
{fig::compare_blen_semi-coherent}
{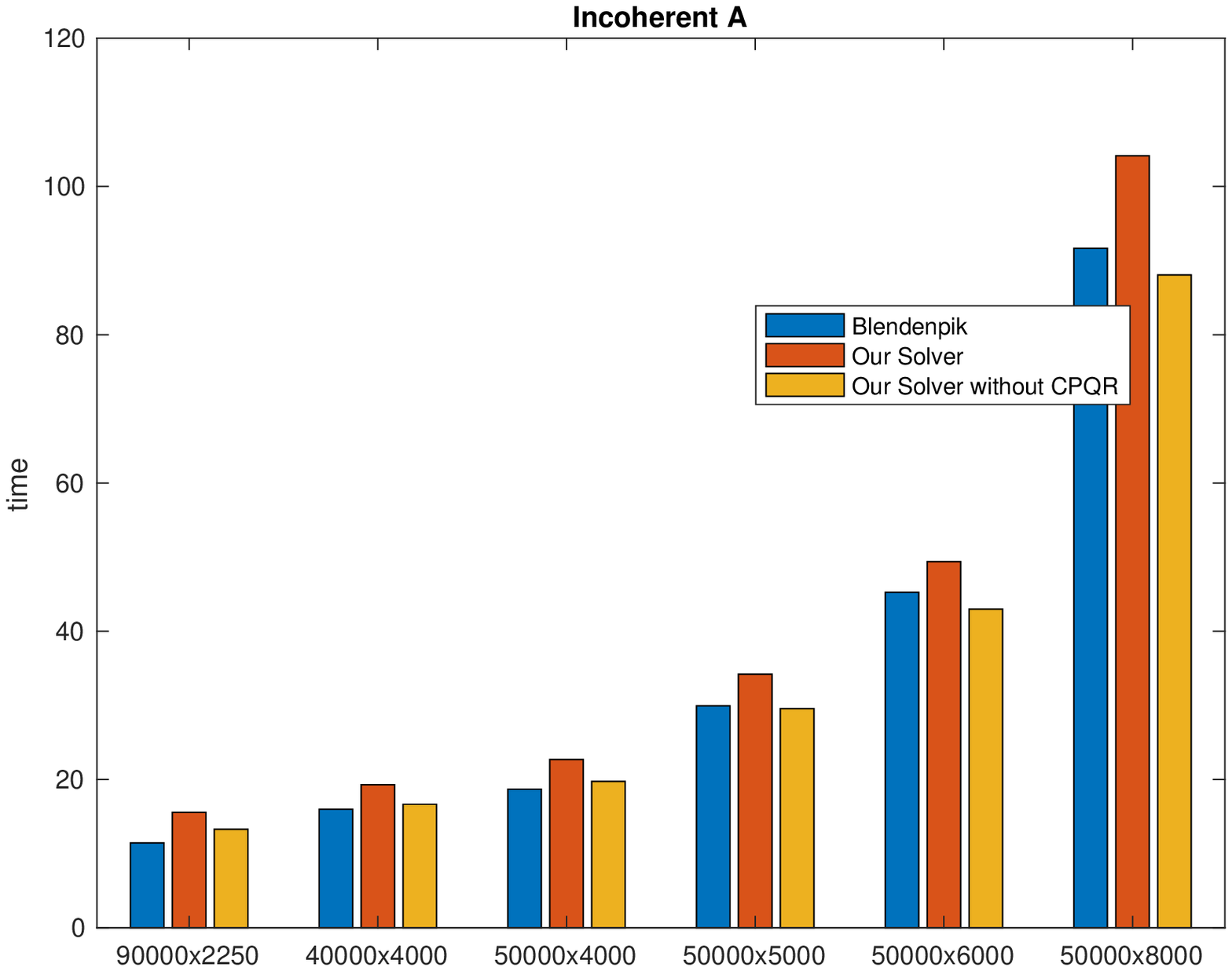}
{Time (in seconds) taken by solvers applied to problem (\ref{LLS-statement}) with incoherent, dense and random matrices of various sizes (x-axis).}
{fig::compare_blen_incoherent}

\subsubsection{Runtime performance of Ski-LLS on random, full-rank and sparse input (Test Set 2)}
Figures \ref{fig::sparse_rand_inco}, \ref{fig::sparse_rand_semi_co} and \ref{fig::sparse_rand_co} show the performance of \solverNameSparse{} compared to LS\_HSL, LS\_SPQR and LSRN on sparse random matrices of different types and sizes. We see \solverName{} can be up to $10$ times faster on this class of data matrix. Tables \ref{tab::res_sparse_rand_inco}, \ref{tab::res_sparse_rand_semi_co} and \ref{tab::res_sparse_rand_co} in the appendix record the residual values at the solution. We see that our solver can be more accurate than the state-of-the-art LS\_HSL and LS\_SPQR, although LSRN is the most accurate.

\newcommand{\runningTimeSparseRandomCaption}[1]{
Runtime comparison of \solverName{} with LS\_HSL, LS\_SPQR and LSRN on randomly generated #1 sparse matrices of different sizes.}

\threeFigures{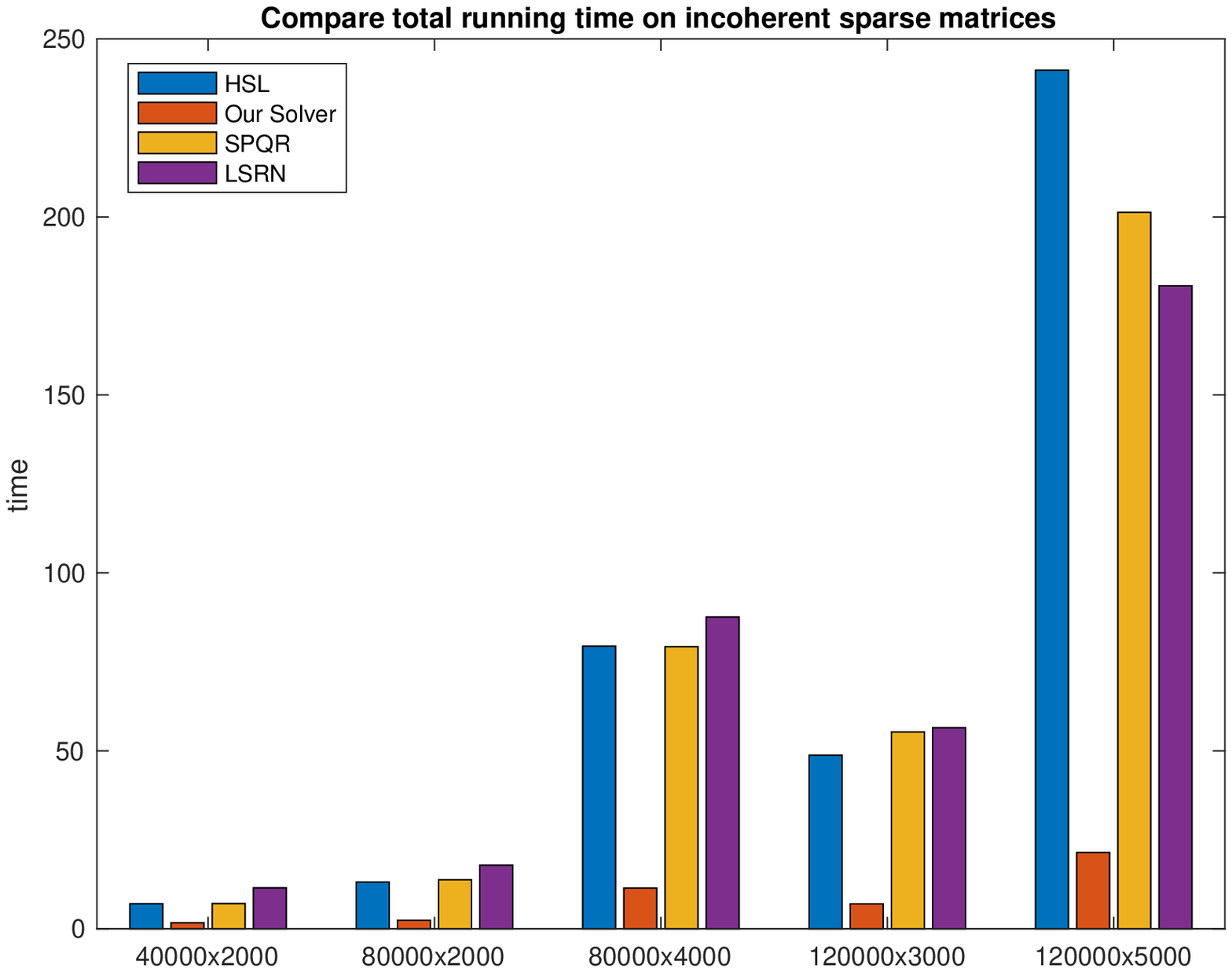}{\runningTimeSparseRandomCaption{incoherent}}{fig::sparse_rand_inco}
{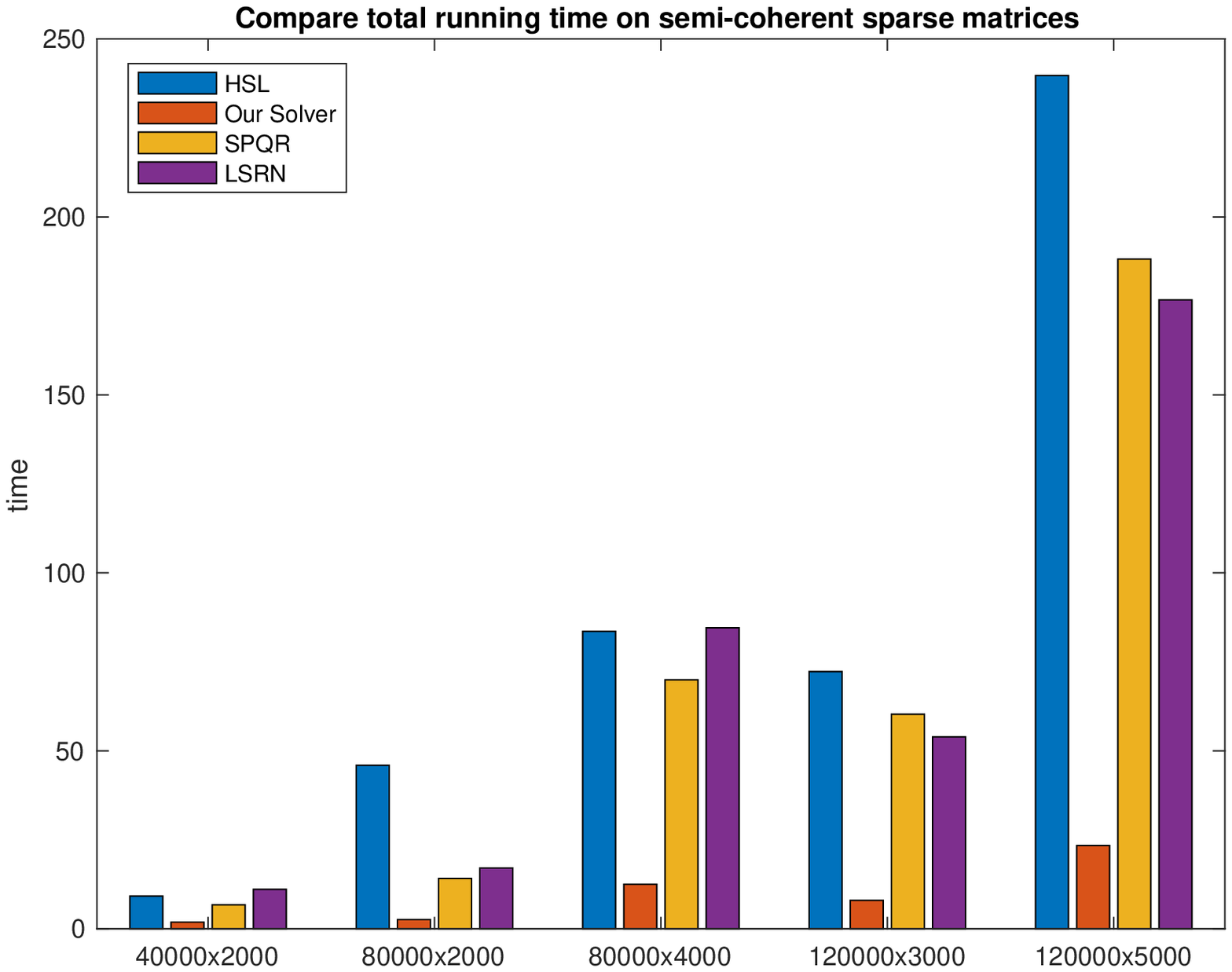}{\runningTimeSparseRandomCaption{semi-coherent}}{fig::sparse_rand_semi_co}
{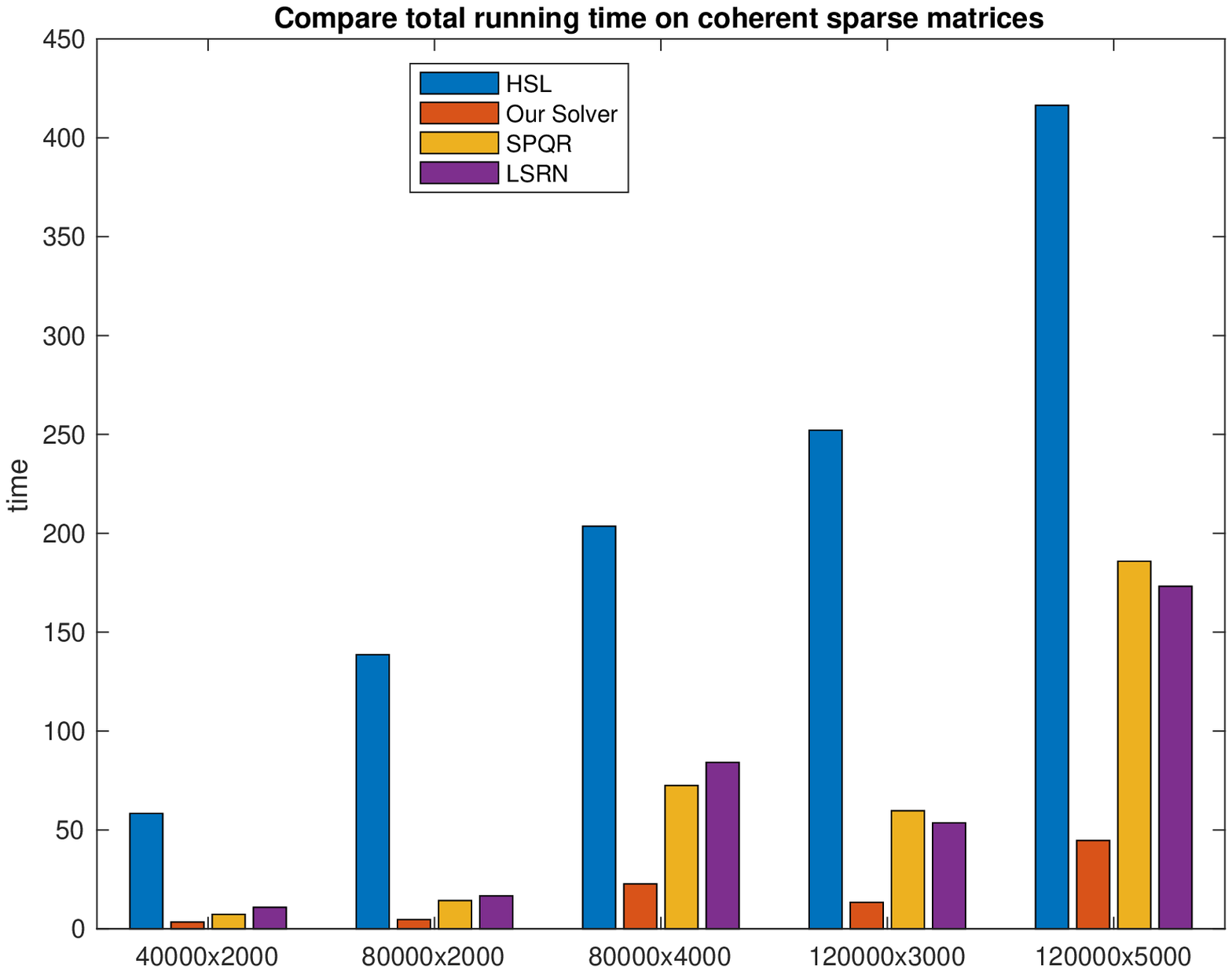}{\runningTimeSparseRandomCaption{coherent}}{fig::sparse_rand_co}

\subsection{Benchmarking \solverName{} on the Florida Matrix Collection (Test Set 3)}
\label{section:Florida_section}

We present the results of  benchmarking \solverNameSparse{} against LSRN, SPQR and HSL on the large scale Test Set 3.
Figure
\ref{fig::all_solver_2} shows these results for the entire collection considered here, namely the matrices or their transpose with $n\geq 2d$; we find that neither of the two sketching solvers,
\solverName{} and LSRN, are sufficiently competitive with the direct solver SPQR and the fastest solver, preconditioned iterative one in HSL. This is perhaps unsurprising as we would only really expect sketching solvers to improve state of the art ones when sketching can play a role, and for that either substantially more rows must be present, or some form of difficulty/structure. Indeed, 
the situation substantially changes  when we subselect from the entire Florida set the performance of solvers on significantly overdetermined problems, or ill-conditioned/`difficult ones' or moderately sparse ones.

\paragraph{Highly over-determined matrices in the Florida Matrix Collection}
\autoref{fig::all_solver_30} shows \solverName{} is the fastest in 75\% of problems in the Florida matrix collection with $n\geq 30d$, outperforming the iterative HSL code (with incomplete Cholesky preconditioning).

\paragraph{On moderately over-determined Florida collection problems}
\autoref{fig::all_solver_10} shows LS\_HSL is the fastest for the largest percentage of problems in the Florida matrix collection with $n\geq 10d$. However, \solverName{} is still competitive and noticeably faster than LSRN. As the proportion of rows versus columns further decreases, however, the profiles would approach those in Figure
\ref{fig::all_solver_2}.

\paragraph{Effect of condition number}
Many of the matrices in the Florida matrix condition are well-conditioned. Thus  LSQR (without preconditioning) converges in  few iterations. Then the benefits of the high quality preconditioner computed by   \solverName{} are lost to the incomplete one in LS\_HSL.
However, if the problems are not well-conditioned/'difficult',
the situation reverses.

\autoref{fig::all_solver_10_LSQR_5} shows \solverName{} is the fastest in more than 50\% of the moderately over-determined  problems ($n\geq 10d$) if we only consider problems such that it takes LSQR more than 5 seconds to solve.

\paragraph{Effect of sparsity}
Figure \ref{fig::density001} shows \solverName{} is highly competitive, being the fastest in all but one moderately over-determined problems with moderate sparsity ($\text{nnz}(A) \geq 0.01nd)$. The effect of increasing the sparsity by allowing ($\text{nnz}(A) \geq 0.001nd)$
is shown in Figure \ref{fig::density0001}.

\renewcommand{\mysize}{0.42}
\begin{figure}
    \centering
    \begin{minipage}{\mysize\textwidth}
        \centering
        \includegraphics[width=\textwidth]{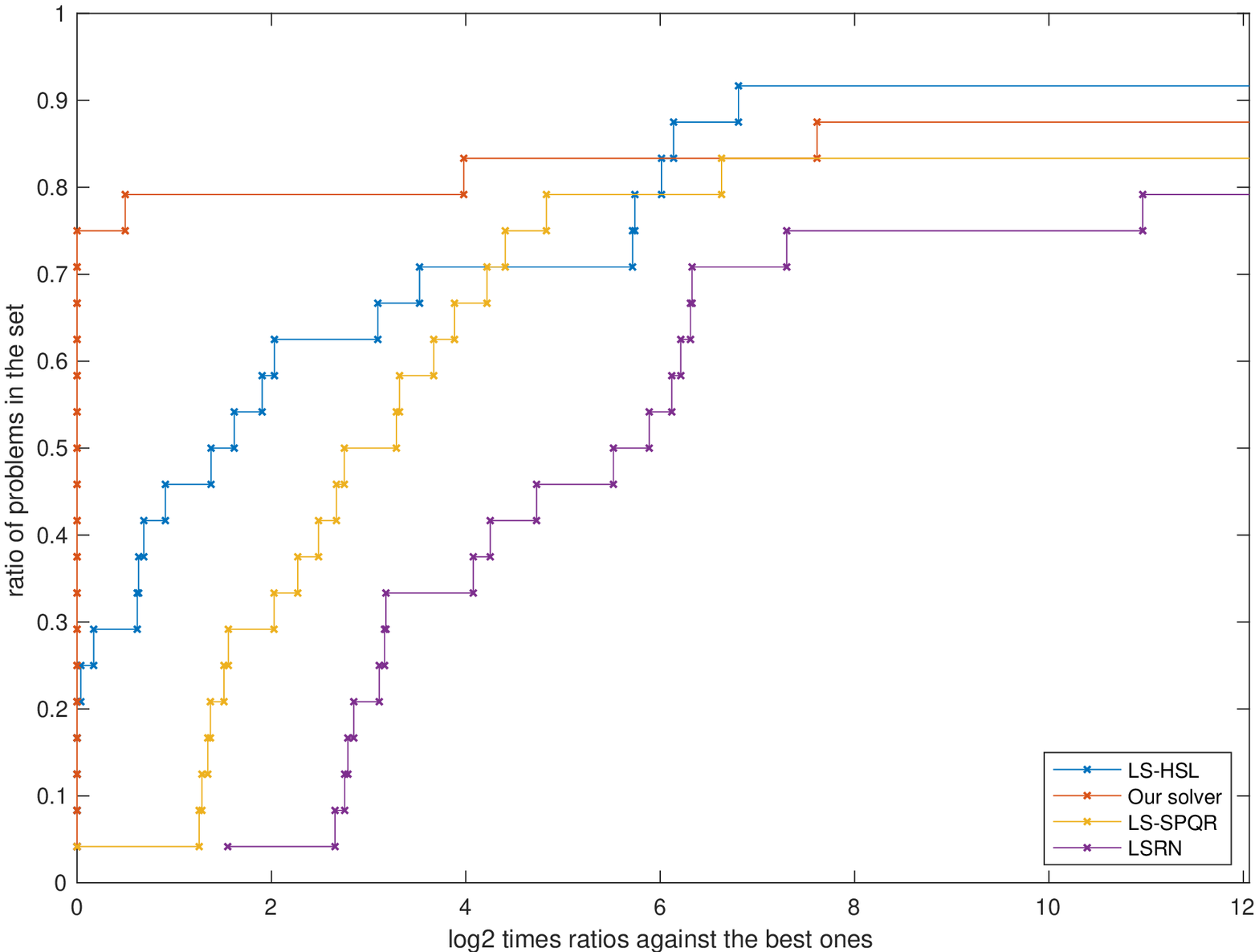} 
        \caption{\performanceProfileCaption{$n\geq 30d$}. }
        \label{fig::all_solver_30}
    \end{minipage}\hfill  
    \centering
    \begin{minipage}{\mysize\textwidth}
        \centering
        \includegraphics[width=\textwidth]{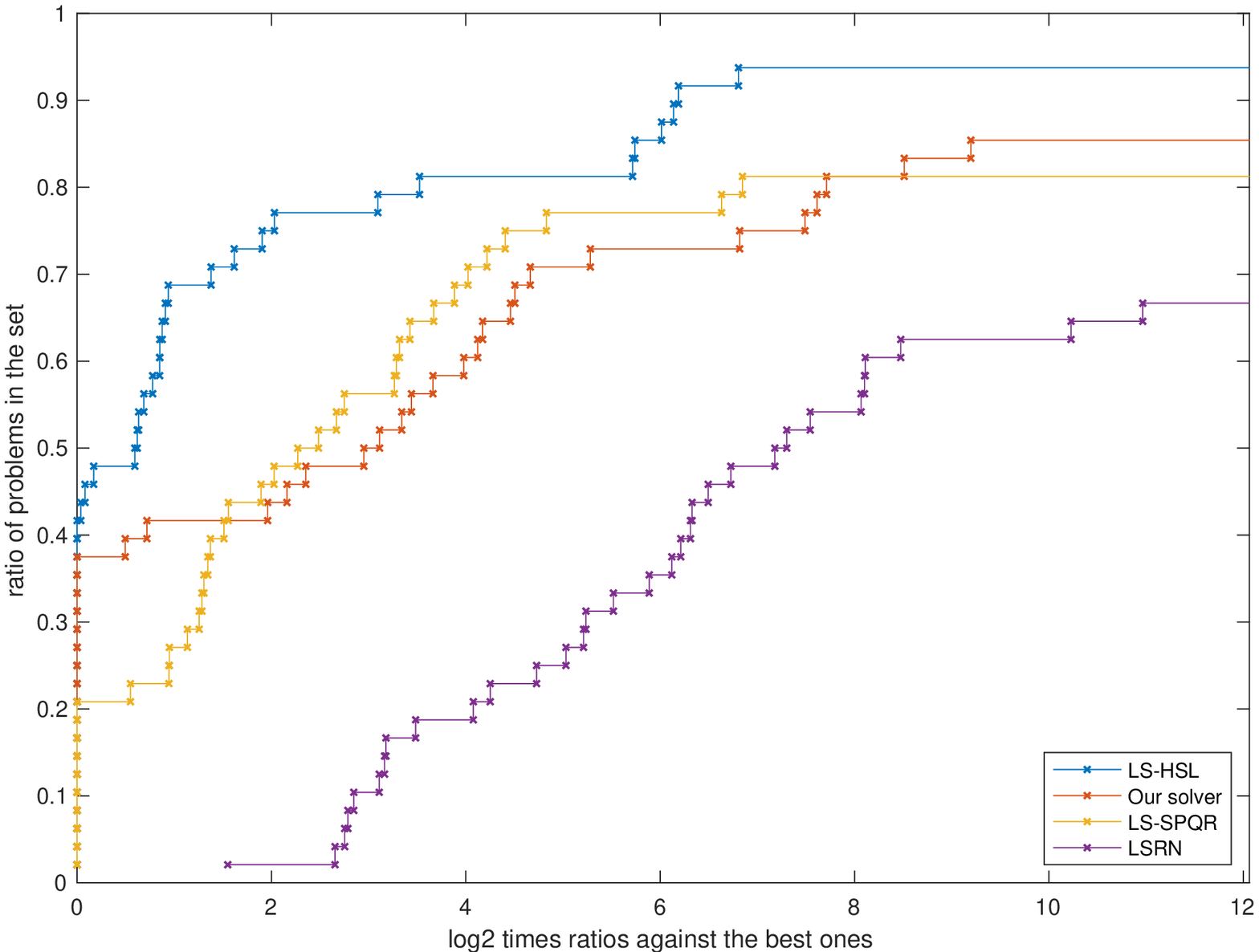} 
        \caption{\performanceProfileCaption{$n\geq 10d$}.}
        \label{fig::all_solver_10}
    \end{minipage}\hfill  
\end{figure}

\begin{figure}
    \centering
    \begin{minipage}{\mysize\textwidth}
        \centering
        \includegraphics[width=\textwidth]{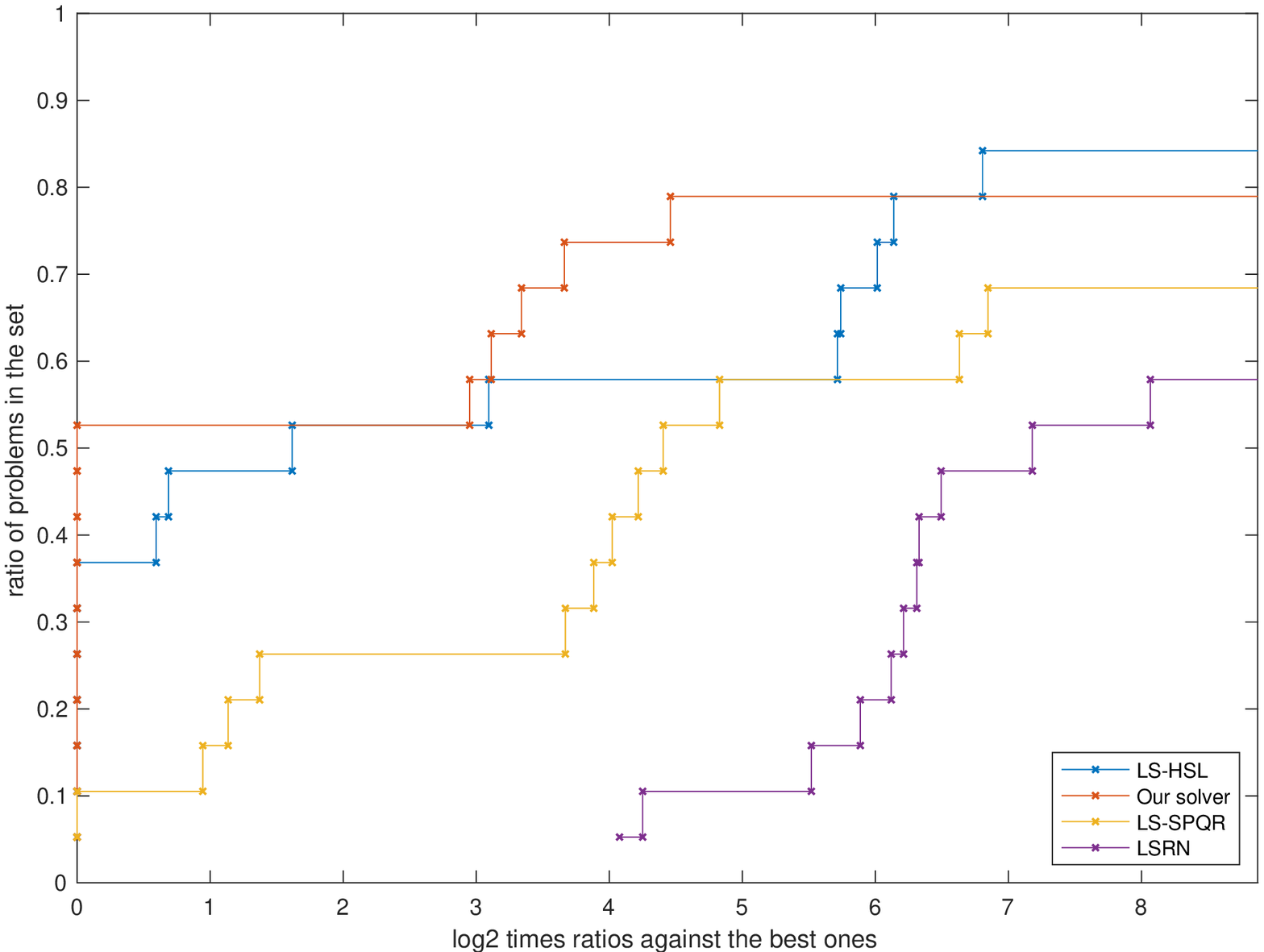} 
        \caption{\performanceProfileCaption{with $n\geq 10d$ and the unpreconditioned LSQR takes more than 5 seconds to solve}.}
        \label{fig::all_solver_10_LSQR_5}
    \end{minipage}\hfill  
    \begin{minipage}{\mysize\textwidth}
    \centering
    \includegraphics[width=\textwidth]{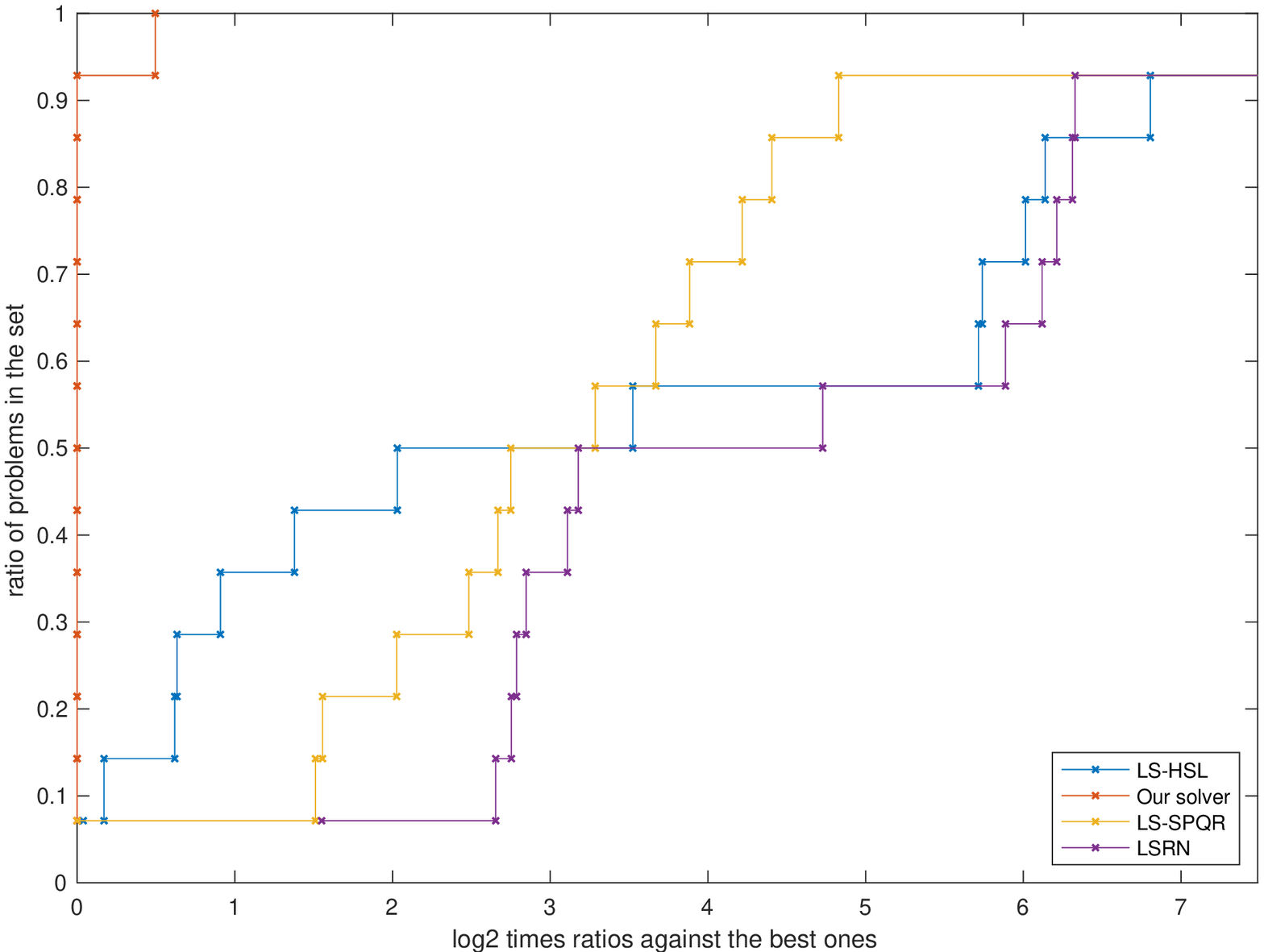}
    \caption{\performanceProfileCaption{$n \geq 10d$ and $\text{nnz}(A) \geq 0.01nd$}. }
    \label{fig::density001}
    \end{minipage}
\end{figure}

\bibliography{bib/2021_04_07_proper_ref.bib} 

\begin{thebibliography}{10}

\bibitem{10.1145/1132516.1132597}
N.~Ailon and B.~Chazelle.
\newblock Approximate nearest neighbors and the fast {J}ohnson-{L}indenstrauss
  transform.
\newblock In {\em S{TOC}'06: {P}roceedings of the 38th {A}nnual {ACM}
  {S}ymposium on {T}heory of {C}omputing}, pages 557--563. ACM, New York, 2006.

\bibitem{10.1145/2483699.2483701}
N.~Ailon and E.~Liberty.
\newblock An almost optimal unrestricted fast {J}ohnson-{L}indenstrauss
  transform.
\newblock {\em ACM Trans. Algorithms}, 9(3):Art. 21, 1--12, 2013.

\bibitem{doi:10.1137/090767911}
H.~Avron, P.~Maymounkov, and S.~Toledo.
\newblock Blendenpik: supercharging {L}apack's least-squares solver.
\newblock {\em SIAM J. Sci. Comput.}, 32(3):1217--1236, 2010.

\bibitem{Avron:2009aa}
H.~Avron, E.~Ng, and S.~Toledo.
\newblock Using perturbed {$QR$} factorizations to solve linear least-squares
  problems.
\newblock {\em SIAM J. Matrix Anal. Appl.}, 31(2):674--693, 2009.

\bibitem{Bjorck:1996uz}
A.~Bj\"{o}rck.
\newblock {\em Numerical methods for least squares problems}.
\newblock Society for Industrial and Applied Mathematics (SIAM), Philadelphia,
  PA, 1996.

\bibitem{Bourgain:2015tc}
J.~Bourgain, S.~Dirksen, and J.~Nelson.
\newblock Toward a unified theory of sparse dimensionality reduction in
  {E}uclidean space.
\newblock {\em Geom. Funct. Anal.}, 25(4):1009--1088, 2015.

\bibitem{CHEN2020105639}
L.~Chen, S.~Zhou, and J.~Ma.
\newblock Stable sparse subspace embedding for dimensionality reduction.
\newblock {\em Knowledge-Based Systems}, 195:105639, 2020.

\bibitem{10.1145/3019134}
K.~L. Clarkson and D.~P. Woodruff.
\newblock Low-rank approximation and regression in input sparsity time.
\newblock {\em J. ACM}, 63(6):Art. 54, 1--45, 2017.

\bibitem{10.5555/2884435.2884456}
M.~B. Cohen.
\newblock Nearly tight oblivious subspace embeddings by trace inequalities.
\newblock In {\em Proceedings of the {T}wenty-{S}eventh {A}nnual {ACM}-{SIAM}
  {S}ymposium on {D}iscrete {A}lgorithms}, pages 278--287. ACM, New York, 2016.

\bibitem{10.1145/3219819.3220098}
Y.~Dahiya, D.~Konomis, and D.~P. Woodruff.
\newblock An empirical evaluation of sketching for numerical linear algebra.
\newblock In {\em Proceedings of the 24th ACM SIGKDD International Conference
  on Knowledge Discovery \& Data Mining}, KDD '18, pages 1292--1300, New York,
  NY, USA, 2018. Association for Computing Machinery.

\bibitem{10.1145/2049662.2049670}
T.~A. Davis.
\newblock Algorithm 915, {S}uite{S}parse{QR}: multifrontal multithreaded
  rank-revealing sparse {QR} factorization.
\newblock {\em ACM Trans. Math. Software}, 38(1):Art. 1, 1--22, 2011.

\bibitem{10.1145/2049662.2049663}
T.~A. Davis and Y.~Hu.
\newblock The {U}niversity of {F}lorida sparse matrix collection.
\newblock {\em ACM Trans. Math. Software}, 38(1):Art. 1, 1--25, 2011.

\bibitem{dolan2002benchmarking}
E.~D. Dolan and J.~J. Mor{\'e}.
\newblock Benchmarking optimization software with performance profiles.
\newblock {\em Mathematical programming}, 91(2):201--213, 2002.

\bibitem{10.5555/1109557.1109682}
P.~Drineas, M.~W. Mahoney, and S.~Muthukrishnan.
\newblock Sampling algorithms for l2 regression and applications.
\newblock In {\em Proceedings of the Seventeenth Annual ACM-SIAM Symposium on
  Discrete Algorithm}, SODA '06, page 1127–1136, USA, 2006. Society for
  Industrial and Applied Mathematics.

\bibitem{10.5555/3327345.3327444}
C.~Freksen, L.~Kamma, and K.~G. Larsen.
\newblock Fully understanding the hashing trick.
\newblock In {\em Proceedings of the 32nd International Conference on Neural
  Information Processing Systems}, NIPS'18, pages 5394--5404, Red Hook, NY,
  USA, 2018. Curran Associates Inc.

\bibitem{10.5555/248979}
G.~H. Golub and C.~F. Van~Loan.
\newblock {\em Matrix computations}.
\newblock Johns Hopkins Studies in the Mathematical Sciences. Johns Hopkins
  University Press, Baltimore, MD, third edition, 1996.

\bibitem{Gould:2016vg}
N.~Gould and J.~Scott.
\newblock The state-of-the-art of preconditioners for sparse linear
  least-squares problems: the complete results.
\newblock Technical report, STFC Rutherford Appleton Laboratory, 2015.
\newblock Available at ftp://cuter.rl.ac.uk/pub/nimg/pubs/GoulScot16b_toms.pdf.

\bibitem{10.1145/3014057}
N.~Gould and J.~Scott.
\newblock The state-of-the-art of preconditioners for sparse linear
  least-square problems.
\newblock {\em ACM Trans. Math. Software}, 43(4):Art. 36, 1--35, 2017.

\bibitem{MR3432148}
R.~M. Gower and P.~Richt\'{a}rik.
\newblock Randomized iterative methods for linear systems.
\newblock {\em SIAM J. Matrix Anal. Appl.}, 36(4):1660--1690, 2015.

\bibitem{iwen2020lower}
M.~A. Iwen, D.~Needell, E.~Rebrova, and A.~Zare.
\newblock Lower {M}emory {O}blivious ({T}ensor) {S}ubspace {E}mbeddings with
  {F}ewer {R}andom {B}its: {M}odewise {M}ethods for {L}east {S}quares.
\newblock {\em SIAM J. Matrix Anal. Appl.}, 42(1):376--416, 2021.

\bibitem{Iyer:2016aa}
C.~Iyer, H.~Avron, G.~Kollias, Y.~Ineichen, C.~Carothers, and P.~Drineas.
\newblock A randomized least squares solver for terabyte-sized dense
  overdetermined systems.
\newblock {\em J. Comput. Sci.}, 36:100547, 2019.

\bibitem{10.5555/3019094.3019103}
C.~{Iyer}, C.~{Carothers}, and P.~{Drineas}.
\newblock Randomized sketching for large-scale sparse ridge regression
  problems.
\newblock In {\em 2016 7th Workshop on Latest Advances in Scalable Algorithms
  for Large-Scale Systems (ScalA)}, pages 65--72, 2016.

\bibitem{NIPS2019_9656}
M.~Jagadeesan.
\newblock Understanding sparse {JL} for feature hashing.
\newblock In {\em Advances in Neural Information Processing Systems},
  volume~32. Curran Associates, Inc., 2019.

\bibitem{Johnson:1984aa}
W.~B. Johnson and J.~Lindenstrauss.
\newblock Extensions of {L}ipschitz mappings into a {H}ilbert space.
\newblock In {\em Conference in modern analysis and probability ({N}ew {H}aven,
  {C}onn., 1982)}, volume~26 of {\em Contemp. Math.}, pages 189--206. Amer.
  Math. Soc., Providence, RI, 1984.

\bibitem{kahale2020leastsquares}
N.~{Kahale}.
\newblock {Least-squares regressions via randomized Hessians}.
\newblock {\em arXiv e-prints}, page arXiv:2006.01017, June 2020.

\bibitem{lacotte2019faster}
J.~{Lacotte} and M.~{Pilanci}.
\newblock {Faster Least Squares Optimization}.
\newblock {\em arXiv e-prints}, page arXiv:1911.02675, Nov. 2019.

\bibitem{lacotte2020optimal}
J.~{Lacotte} and M.~{Pilanci}.
\newblock {Optimal Randomized First-Order Methods for Least-Squares Problems}.
\newblock {\em arXiv e-prints}, page arXiv:2002.09488, Feb. 2020.

\bibitem{liu2021extending}
S.~{Liu}, T.~{Liu}, A.~{Vakilian}, Y.~{Wan}, and D.~P. {Woodruff}.
\newblock {Extending and Improving Learned CountSketch}.
\newblock {\em arXiv e-prints}, page arXiv:2007.09890, July 2020.

\bibitem{2019arXiv190912176L}
N.~{Loizou}.
\newblock {Randomized Iterative Methods for Linear Systems: Momentum,
  Inexactness and Gossip}.
\newblock {\em arXiv e-prints}, page arXiv:1909.12176, Sept. 2019.

\bibitem{MR4187148}
N.~Loizou and P.~Richt\'{a}rik.
\newblock Convergence analysis of inexact randomized iterative methods.
\newblock {\em SIAM J. Sci. Comput.}, 42(6):A3979--A4016, 2020.

\bibitem{lopes2018error}
M.~Lopes, S.~Wang, and M.~Mahoney.
\newblock Error estimation for randomized least-squares algorithms via the
  bootstrap.
\newblock In J.~Dy and A.~Krause, editors, {\em Proceedings of the 35th
  International Conference on Machine Learning}, volume~80 of {\em Proceedings
  of Machine Learning Research}, pages 3217--3226. PMLR, 10--15 Jul 2018.

\bibitem{10.1561/2200000035}
M.~W. Mahoney.
\newblock Randomized algorithms for matrices and data.
\newblock {\em Found. Trends Mach. Learn.}, 3(2):123--224, Feb. 2011.

\bibitem{martinsson2015blocked}
P.-G. Martinsson.
\newblock {Blocked rank-revealing QR factorizations: How randomized sampling
  can be used to avoid single-vector pivoting}.
\newblock {\em arXiv e-prints}, page arXiv:1505.08115, May 2015.

\bibitem{Martinsson:2017eh}
P.-G. Martinsson, G.~Quintana~Ort\'{\i}, N.~Heavner, and R.~van~de Geijn.
\newblock Householder {QR} factorization with randomization for column pivoting
  ({HQRRP}).
\newblock {\em SIAM J. Sci. Comput.}, 39(2):C96--C115, 2017.

\bibitem{10.1145/2488608.2488621}
X.~Meng and M.~W. Mahoney.
\newblock Low-distortion subspace embeddings in input-sparsity time and
  applications to robust linear regression.
\newblock In {\em S{TOC}'13---{P}roceedings of the 2013 {ACM} {S}ymposium on
  {T}heory of {C}omputing}, pages 91--100. ACM, New York, 2013.

\bibitem{Meng:2014ib}
X.~Meng, M.~A. Saunders, and M.~W. Mahoney.
\newblock L{SRN}: a parallel iterative solver for strongly over- or
  underdetermined systems.
\newblock {\em SIAM J. Sci. Comput.}, 36(2):C95--C118, 2014.

\bibitem{Nelson:te}
J.~Nelson and H.~L. Nguyen.
\newblock O{SNAP}: faster numerical linear algebra algorithms via sparser
  subspace embeddings.
\newblock In {\em 2013 {IEEE} 54th {A}nnual {S}ymposium on {F}oundations of
  {C}omputer {S}cience---{FOCS} 2013}, pages 117--126. IEEE Computer Soc., Los
  Alamitos, CA, 2013.

\bibitem{10.1145/2488608.2488622}
J.~Nelson and H.~L. Nguyen.
\newblock Sparsity lower bounds for dimensionality reducing maps.
\newblock In {\em S{TOC}'13---{P}roceedings of the 2013 {ACM} {S}ymposium on
  {T}heory of {C}omputing}, pages 101--110. ACM, New York, 2013.

\bibitem{Nelson:2014uu}
J.~Nelson and H.~L. Nguyen.
\newblock Lower bounds for oblivious subspace embeddings.
\newblock In {\em Automata, languages, and programming. {P}art {I}}, volume
  8572 of {\em Lecture Notes in Comput. Sci.}, pages 883--894. Springer,
  Heidelberg, 2014.

\bibitem{Nocedal:2006uv}
J.~Nocedal and S.~J. Wright.
\newblock {\em Numerical optimization}.
\newblock Springer Series in Operations Research and Financial Engineering.
  Springer, New York, second edition, 2006.

\bibitem{10.1145/355984.355989}
C.~C. Paige and M.~A. Saunders.
\newblock L{SQR}: an algorithm for sparse linear equations and sparse least
  squares.
\newblock {\em ACM Trans. Math. Software}, 8(1):43--71, 1982.

\bibitem{Rokhlin:2008wb}
V.~Rokhlin and M.~Tygert.
\newblock A fast randomized algorithm for overdetermined linear least-squares
  regression.
\newblock {\em Proc. Natl. Acad. Sci. USA}, 105(36):13212--13217, 2008.

\bibitem{10.1109/FOCS.2006.37}
T.~{Sarlos}.
\newblock Improved approximation algorithms for large matrices via random
  projections.
\newblock In {\em 2006 47th Annual IEEE Symposium on Foundations of Computer
  Science (FOCS'06)}, pages 143--152, 2006.

\bibitem{Tropp:wr}
J.~A. Tropp.
\newblock Improved analysis of the subsampled randomized {H}adamard transform.
\newblock {\em Adv. Adapt. Data Anal.}, 3(1-2):115--126, 2011.

\bibitem{10.1561/0400000060}
D.~P. Woodruff.
\newblock Sketching as a tool for numerical linear algebra.
\newblock {\em Found. Trends Theor. Comput. Sci.}, 10(1-2):1--157, 2014.

\bibitem{zhu2018gradientbased}
R.~Zhu.
\newblock Gradient-based sampling: An adaptive importance sampling for
  least-squares.
\newblock In {\em Proceedings of the 30th International Conference on Neural
  Information Processing Systems}, NIPS'16, pages 406--414, Red Hook, NY, USA,
  2016. Curran Associates Inc.

\end{thebibliography}
\bibliographystyle{abbrv}

\appendix
\section{Additional proofs}
\subsection{Proof of results in Section 2}
\begin{proof}[Proof of Lemma \ref{rank-of-sketched-equal-to-unsketched}]
Let $B\in \R^{n \times k}.$
By rank-nullity theorem, $\rank(B) + \dim \ker(B) = \rank(SB) + \dim \ker(SB) =k$. Clearly, $\dim \ker(SB) \geq \dim \ker(B)$. If the previous inequality is strict, then there exists $z\in \R^k$ such that $\|SBz\|_2 = 0$ and $\|Bz\|_2 >0$, contradicting the assumption that $S$ is an $\epsilon$-subspace embedding for $B$ according to \eqref{subspace_embedding_def1}.
\end{proof}

    \begin{proof}[Proof of Lemma \ref{subspace-embedding-def-2}]
\begin{itemize}
    \item[(i)]
Let $A = U \Sigma V^T$ be defined as in (\ref{thin-SVD}). If $S$ is an $\epsilon$-subspace embedding for $A$, let $z \in \R^{r}$ and define $x = V\Sigma^{-1}z \in \R^{d}$. Then we have $Uz = Ax$ and

\begin{equation}
\|SUz\|_2^2 = \|SAx\|_2^2 
			\leq (1+\epsilon) \|Ax\|_2^2 \\
			= (1+\epsilon) \|Uz\|_2^2,
\end{equation}
where we have used $Uz = Ax$ and (\ref{subspace_embedding_def1}). Similarly, we have $\|SUz\|_2^2 \geq (1-\epsilon) \|Uz\|_2^2$. Hence $S$ is an $\epsilon$-subspace embedding for $U$.

Conversely, given $S$ is an $\epsilon$-subspace embedding for $U$, let $x \in \R^{d}$ and $z = \Sigma V^T x \in \R^{r}$. Then we have $Ax = Uz$, and $\|SAx\|_2^2 = \|SUz\|_2^2 \leq (1+\epsilon)\|Uz\|_2^2 = (1+\epsilon) \|Ax\|_2^2$. Similarly $\|SAx\|_2^2 \geq (1-\epsilon) \|Ax\|_2^2$. Hence $S$ is an $\epsilon$-subspace embedding for $A$.
\item[(ii)] Since the equivalence in (i) holds, note that \eqref{U-condition} clearly implies \eqref{U-condition-2}. The latter also implies the former if 
\eqref{U-condition-2} is applied to $z/\|z\|_2$ for any nonzero $z\in \R^r$.
\end{itemize}
\end{proof}

\begin{proof}[Proof of Lemma \ref{non_uniformity_col_subspace_coherence}]

Let $A = U\Sigma V^T$ be defined as in \eqref{thin-SVD}, and let $z = \Sigma V^T x \in \R^{r}$. Then $y = Ax = Uz$.
Therefore 
\begin{align}
\|y\|_{\infty} = \|Uz\|_{\infty} = \max_{1\leq i \leq n} | \langle U_i, z \rangle| \leq \max_{1 \leq i \leq n} \|U_i\|_2 \|z\|_2 \leq  \mu(A) \|z\|_2,
\end{align}
where $U_i$ denotes the $i^{th}$ row of $U$.
Furthermore, $\|y\|_2 = \|Uz\|_2 = \|z\|_2$ which then implies $\nu(y) = \|y\|_{\infty} / \|y\|_2 \leq \mu(A)$.
\end{proof}

\begin{proof}[Proof of Lemma \ref{point to JL plus}]
Let $Y = \yExpression$. Let $B_i$ be the event that $S$ is an $\epsilon$-JL embedding for $y_i\in Y$. Then $\probability{B_i} \geq 1-\delta$ by  assumption. We have
\begin{align}
    \probability{\text{S is an $\epsilon$-JL embedding for Y}}=
    \probability{\cap_i B_i} & = 1 - \probability{\complement{\cap_i B_i}} \\
    & = 1 - \probability{\cup_i B_i^c} \\
    & \geq 1 - \sum_i \probability{B_i^c} \\
    & = 1 - \sum_i \squareBracket{1 - \probability{B_i}}\\
    & \geq 1 - \sum_i \squareBracket{1 - (1-\delta)} = 1-|Y|\delta.
\end{align}
\end{proof}

\begin{proof}[Proof of Lemma \ref{lem::lemma7}]
Let $y_1, y_2 \in Y$. We have that
\begin{align}
\left| \la Sy_1, Sy_2 \ra - \la y_1, y_2 \ra \right| &=  | ( \|S(y_1+y_2)\|^2 - \|S(y_1-y_2)\|^2 ) \\\nonumber
&\quad\quad - ( \|(y_1+y_2)\|^2 - \|(y_1-y_2)\|^2 ) | /4 \\\nonumber
&\leq \epsilon( \|(y_1+y_2)\|^2 + \|(y_1 -y_2)\|^2)/4 \\\nonumber
& = \epsilon( \|y_1\|^2 + \|y_2\|^2)/2 \\\nonumber
& = \epsilon,
\end{align}
where to obtain the inequality, we use that $S$ is an $\epsilon$-JL embedding for $\set{\union{\yPlus}{\yMinus}}$;  the last equality follows from $\|y_1\|_2=\|y_2\|_2=1$.
\end{proof}

  \subsection{Proof of results in Section 3}

\begin{proof}[Proof of Lemma \ref{Gamma-cover-existance}]
Let $\tilde{M} = \{z \in \R^r: \|z\|_2=1\}$. Let $\tilde{N} \subseteq \tilde{M}$ be the maximal set such that no two points in $\tilde{N}$ are within distance $\gamma$ from each other. Then it follows that the r-dimensional balls centred at points in $\tilde{N}$ with radius $\gamma/2$ are all disjoint and contained in the r-dimensional ball centred at the origin with radius $(1+\gamma/2)$. Hence
\begin{align}
\frac{\text{Volume of the r-dimensional ball centred at the origin with radius $(1+\gamma/2)$}}{\text{Total volume of the r-dimensional balls centred at points in $\tilde{N}$ with radius $\gamma/2$ }} \\ 
= \frac{1}{|\tilde{N}|} \frac{(1+\frac{\gamma}{2})^{r}}{(\frac{\gamma}{2})^{r}} \geq 1,
\end{align}
which implies $|\tilde{N}| \leq (1+\frac{2}{\gamma})^r$.

Let $N = \{Uz \in \R^n: z \in \tilde{N}\}$. Then $|N| \leq |\tilde{N}| \leq (1+\frac{2}{\gamma})^r$ and we show $N$ is a $\gamma$-cover for $M$. Given $y_M \in M$, there exists $z_M \in \tilde{M}$ such that $y_M = Uz_M$. By definition of $\tilde{N}$, there must be $z_N \in \tilde{N}$ such that $\|z_M - z_N\|_2 \leq \gamma$ as otherwise $\tilde{N}$ would not be maximal. 
Let $y_N = Uz_N \in N$. Since $U$ has orthonormal columns, we have $\|y_M -y_N\|_2 = \|z_M-z_N\|_2 \leq \gamma$.

\end{proof}

To prove Lemma \ref{approximation-of-net}, we need the following Lemma. 

\begin{lemma} \label{approximation_unit_ball}
Let $\gamma \in (0,1)$, $U\in\R^{n\times r}$ having orthonormal columns and $M\subseteq \R^{n}$ associated with $U$ be defined in \eqref{MU}. Let $N$ be a $\gamma$-cover of $M$, $y\in M$. Then for any $k \in \N$, there exists 
 $\alpha_0, \alpha_1, \dots, \alpha_k \in \R$, $y_0, y_1, y_2, \dots, y_k \in N$ such that
\begin{align}
 \|y - \sum_{i=0}^k \alpha_i y_i \|_2 \leq \gamma^{k+1}, \label{A-1-1}\\
 |\alpha_i| \leq \gamma^i, i=0,1,\dots, k. \label{A-1-2}
 \end{align}

\end{lemma}

\begin{proof}
We use induction. Let $k=0$. Then by definition of a $\gamma$-cover, there exists $y_0 \in N$ such that $\|y - y_0\| < \gamma$. Letting $\alpha_0=1$, we have covered the $k=0$ case.

Now assume \eqref{A-1-1} and \eqref{A-1-2} are true for $k = K \in \N$. Namely there exists $\alpha_0, \alpha_1, \dots, \alpha_K \in \R$, $y_0, y_1, y_2, \dots, y_K \in N$ such that
\begin{align}
 \|y - \sum_{i=0}^K \alpha_i y_i \|_2 \leq \gamma^{K+1} \\
 |\alpha_i| \leq \gamma^i, i=0,1,\dots, K.
 \end{align}
Because $y, y_0, y_1 \dots, y_K \in N \subseteq M$, there exists $z, z_0, z_1, \dots, z_K \in \R^r$ such that $y=Uz, y_0=Uz_0, y_1=Uz_1, \dots y_K=Uz_K$ with $\|z\| = \|z_0\| = \dots = \|z_K\|=1$. Therefore

\begin{align}
 &\frac{y - \sum_{i=0}^K \alpha_i y_i}{ \|y - \sum_{i=0}^K \alpha_i y_i\|_2} =
 \frac{ U \left( z - \sum_{i=0}^K \alpha_i z_i \right)  }{\|U \left( z - \sum_{i=0}^K \alpha_i z_i \right) \|_2} = 
 U \frac{z - \sum_{i=0}^K \alpha_i z_i}{ \|z - \sum_{i=0}^K \alpha_i z_i\|_2} \in M,
\end{align}
where we have used that the columns of $U$ are orthonormal. 

Since $N$ is a $\gamma$-cover for $M$,  there exists $y_{K+1} \in N$ such that

\begin{align}
 \left\| \frac{y - \sum_{i=0}^K \alpha_i y_i}{ \|y - \sum_{i=0}^K \alpha_i y_i\|_2} - y_{K+1} \right\|_2 \leq \gamma.
\end{align}

Multiplying both sides by $\alpha_{K+1}:=\|y - \sum_{i=0}^K \alpha_i y_i\|_2 \leq \gamma^{K+1}$, we have

\begin{align}
  \|y - \sum_{i=0}^{K+1} \alpha_i y_i \|_2 \leq \gamma^{K+2}, \\
 |\alpha_i| \leq \gamma^i, i=0,1, \dots, K+1.
\end{align}
 
\end{proof}

\begin{proof}[Proof of Lemma \ref{approximation-of-net}]
Let $y \in M$ and $k\in \N$, and consider the approximate representation of $y$ provided in Lemma \ref{approximation_unit_ball}, namely, assume that 
\eqref{A-1-1} and \eqref{A-1-2} hold. Then we have

\begin{align*}
 \|S \sum_{i=0}^k \alpha_i y_i\|_2^2 
		   & = \sum_{i=0}^k \|S\alpha_i y_i\|_2^2 + \sum_{0 \leq i < j \leq k} 2 \langle S\alpha_i y_i, S\alpha_j y_j \rangle \\
		   & = \sum_{i=0}^k \|S\alpha_i y_i\|_2^2 + \sum_{0 \leq i < j \leq k} 2 \langle \alpha_i y_i, \alpha_j y_j \rangle +\\
		   &\hspace*{2.5cm}
		   		+\left[  \sum_{0 \leq i < j \leq k} 2 \langle S\alpha_i y_i, S\alpha_j y_j \rangle - \sum_{0 \leq i < j \leq k} 2 \langle \alpha_i y_i, \alpha_j y_j \rangle\right] \\
		   & \leq (1+\epsilon_1) \sum_{i=0}^k \|\alpha_i y_i\|_2^2 +\sum_{0 \leq i < j \leq k} 2 \langle \alpha_i y_i, \alpha_j y_j \rangle +
 				2 \sum_{0 \leq i <j \leq k} \epsilon_1 \|\alpha_i y_i\|_2 \|\alpha_j y_j\|_2 \\
 		   & = \|\sum_{i=0}^k \alpha_i y_i\|^2_2 + \epsilon_1 \left[
 		   \sum_{i=0}^k \|\alpha_i y_i\|_2^2 + 2 \sum_{0 \leq i <j \leq k}  \|\alpha_i y_i\|_2 \|\alpha_j y_j\|_2
 		   \right],
\end{align*}
where to deduce the inequality, we use that $S$ is a generalised $\epsilon_1$-JL embedding for $N$. Using $\|y_i\|_2=1$ and $|\alpha_i| \leq \gamma^i$, we have
\begin{align}\label{L9:gamma_i}
\frac{1}{\epsilon_1}\left\{\|S \sum_{i=0}^k \alpha_i y_i\|_2^2 - \|\sum_{i=0}^k \alpha_i y_i\|^2_2\right\} &=
    \sum_i^k \|\alpha_i y_i\|_2^2 + 2 \sum_{0 \leq i <j \leq k} \|\alpha_i y_i\|_2 \|\alpha_j y_j\|_2 \\\nonumber
   & \leq \sum_{i=0}^{k} \gamma^i + 2 \sum_{0 \leq i <j \leq k} \gamma^i \gamma^j\\\nonumber
   &\leq \frac{1-\gamma^{k+1}}{1-\gamma}+\frac{2\gamma (1-\gamma^{k-i})(1-\gamma^{2k})}{\left( 1-\gamma \right) \left( 1-\gamma^2 \right)},
\end{align}
where we have used 
\begin{align*}
    \sum_{0 \leq i <j \leq k} \gamma^i \gamma^j = 
\sum_{i=0}^{k-1} \gamma^i \sum_{j=i+1}^{k} \gamma^j &=
\sum_{i=0}^{k-1} \gamma^{2i+1} \sum_{j=0}^{k-i-1} \gamma^j \\
&=
\frac{\gamma(1-\gamma^{k-i})}{1-\gamma} \sum_{i=0}^{k-1} \gamma^{2i} = 
\frac{\gamma (1-\gamma^{k-i})(1-\gamma^{2k})}{ \left( 1-\gamma \right) \left( 1-\gamma^2 \right) }.
\end{align*} 
Letting $k\rightarrow \infty$ in \eqref{L9:gamma_i}, we deduce
\[
\frac{1}{\epsilon_1}\left\{\|S \sum_{i=0}^{\infty} \alpha_i y_i\|_2^2 - \|\sum_{i=0}^{\infty} \alpha_i y_i\|^2_2\right\}\leq 
\frac{1}{1-\gamma}+\frac{2\gamma}{\left( 1-\gamma \right) \left( 1-\gamma^2 \right)}=\frac{1+2\gamma-\gamma^2}{\left( 1-\gamma \right) \left( 1-\gamma^2 \right)},
\]
Letting $k\rightarrow \infty$ in  \eqref{A-1-1} implies 
 $y = \sum_{i=0}^\infty \alpha_i y_i$, and so the above gives
 \[
\|S y\|_2^2 - \|y\|^2_2\leq 
\epsilon_1\frac{1+2\gamma-\gamma^2}{\left( 1-\gamma \right) \left( 1-\gamma^2 \right)}=\epsilon \|y\|_2^2,
\]
where to get the first equality, we used
 $\|y\|_2=1$ and the definition of $\epsilon_1$.
The lower bound in the $\epsilon_1$-JL embedding follows similarly.
\end{proof}

        \subsection{Proof of results in Section 5}

\begin{proof}[Proof of Lemma \ref{Lemma::w}]
Note $SW$ has rank $\k$ because $SW = Q_1$, where $Q_1$ is defined in \eqref{SA-QRV-fac}. By rank-nullity theorem in $\R^\k$, $\rank(W) + \dim \ker(W) = \rank(SW) + \dim \ker(SW)$ where $\ker(W)$ denotes the null space of $W$; and since $\text{dim} \ker(SW) \geq \dim \ker(W)$, we have that $\rank(SW) \leq \rank(W)$. So $\rank(W) \geq \k$. It follows that $\rank(W)=\k$ because $W\in\R^{n\times \k}$ can have at most rank $\k$.
\end{proof}

\begin{proof}[Proof of Lemma \ref{Lemma::p_equals_r}]
Lemma \ref{rank-of-sketched-equal-to-unsketched} gives $r= \rank(A) = \rank(SA)= \k$.
\end{proof}

\begin{proof}[Proof of Lemma \ref{explicit-sketching-guarantee}]
We have that $x_s \in \argmin \|SAx-Sb\|_2$ by checking the optimality condition $(SA)^T SA x_s = (SA)^T Sb$. Hence we have that
\begin{align}
\norms{Ax_s-b} \leq \frac{1}{1-\epsilon} \norms{SAx_s-Sb} \leq \frac{1}{1-\epsilon} \norms{SAx^*-Sb}
 \leq \frac{1+\epsilon}{1-\epsilon} \norms{Ax^*-b},
\end{align}
where the first and the last inequality follow from $S$ being a subspace embedding for $\left(A \,\,b \right)$, while the second inequality is due to $x_s$ minimizing $\|SAx-Sb\|$.
\end{proof}

\begin{proof}[Proof of Lemma \ref{VA-cap}]
Let $z \in \ker(V_1^T) \cap {\rm range}(A^T)$. Then $V_1^T z = 0$ and $z = A^T w$ for some $w \in \R^n$. Let $U, \Sigma, V$ be the SVD factors of $A$ as defined in \eqref{thin-SVD}. Since $S$ is an $\epsilon$-subspace embedding for $A$, $\rank \left( (SU)^T \right) = \rank(SU) = \rank(SA) = r$, where $r$ is the rank of $A$ and hence there exists $\hat{w} \in \R^m$ such that $(SU)^T \hat{w} = U^T w$. Note that 
\begin{equation}
    0 = V_1^T z = V_1 A^T w = V_1^T V \Sigma U^T w = V_1^T V \Sigma U^T S^T \hat{w} = V_1^T A^T S^T \hat{w} = R_{11}^T Q_1^T \hat{w},
\end{equation}
which implies $Q_1^T \hat{w} = 0$ because $R_{11}^T$ is nonsingular.
It follows that
\begin{equation}
    z = A^T w = V \Sigma U^T w = V\Sigma U^T S^T \hat{w} = (SA)^T \hat{w} = V 
\left( \begin{smallmatrix} R_{11}^T Q_1^T \\ R_{12}^T Q_1^T \end{smallmatrix} \right) \hat{w} = 0, 
\end{equation}
where we have used $Q_1^T \hat{w} = 0$ for the last equality.
\end{proof}

\begin{proof}[Proof of Theorem \ref{Speed-of-convergence}]
(i) Using results in \cite{10.5555/248979},  LSQR applied to \eqref{ytau} converges as follows 
\begin{align}
\frac{\|y_j - y_*\|_{ W^T W}}{ \|y_0 - y_*\|_{ W^T W}} \leq 
2 \left( \frac{ \sqrt{\kappa \left[ W^T W \right]} -1 } 
{\sqrt{\kappa \left[ W^T W \right]} +1 } \right)^j, \label{CG_guarantee}
\end{align}
where $y_j$ denotes the $j$th iterate of LSQR and $\kappa(W^TW)$ refers to the condition number of $W^TW$.
Since $S$ is an $\epsilon$-subspace embedding for $A$, we have that the largest singular value of $W$ satisfies
\begin{align*}
\sigma_{\max}(W) = \max_{\|y\|=1} \|AV_1R_{11}^{-1}y\| \leq \sqrt{1+\epsilon} \max_{\|y\|=1} \|SAV_1R_{11}^{-1}y\| = \sqrt{1+\epsilon} \max_{\|y\|=1} \|Q_1y\| = \sqrt{1+\epsilon} ,
\end{align*}
where we have used that $SAV_1R_{11}^{-1} = Q_1$ from (\ref{SA-QRV-fac}). Similarly, it can be shown that the smallest singular
value of $W$ satisfies $\sigma_{\min}(W) \geq \sqrt{1-\epsilon}$. Hence
\begin{align}
\kappa(W^T W) \leq  \frac{1+\epsilon}{1-\epsilon} . \label{Low condition number}
\end{align}
Hence we have
\begin{align*}
    \frac{ \sqrt{\kappa \left[ W^T W \right]} -1 } 
{\sqrt{\kappa \left[ W^T W \right]} +1 } \leq \frac{\sqrt{1+\epsilon} - \sqrt{1-\epsilon}}{\sqrt{1+\epsilon} + \sqrt{1-\epsilon}} =  \frac{\bracket{\sqrt{1+\epsilon} - \sqrt{1-\epsilon}} \bracket{\sqrt{1+\epsilon} + \sqrt{1-\epsilon}} }{\bracket{\sqrt{1+\epsilon} + \sqrt{1-\epsilon}}^2} \leq \epsilon.
\end{align*}
Thus (\ref{CG_guarantee}) implies $\|y_{\tau} - y_*\|_{ W^T W} \leq \tau_r \|y_0 - y_*\|_{ W^T W}$ whenever $\tau \geq \frac{\log(2) + |\log\tau_r|}{|\log\epsilon|}$.

(ii) If we initialize $y_0 : = Q^TSb$ for the LSQR method in Step 4, then we have
\begin{align*}
\|y_0 - y_*\|_{ W^T W} &= \|Ax_s - Ax_*\|_2 = \|Ax_s-b - (Ax_*-b)\|_2 
					                 \leq \|Ax_s -b\| - \|Ax_* -b\| \\
					                 & \leq \left( \sqrt{\frac{1+\epsilon}{1-\epsilon}} - 1 \right) \|Ax_*-b\| 
					                  \leq \frac{2\epsilon}{1-\epsilon} \|Ax_*-b\|,
\end{align*}
where we have used $\sqrt{\frac{1+\epsilon}{1-\epsilon}} \leq \frac{1+\epsilon}{1-\epsilon}$ to get the last inequality.
Using part (i), after at most $\frac{\log(2) + |\log \tau_r|}{|\log\epsilon|}$  LSQR iterations, we have that 
\begin{align}
\|y_{\tau} - y_*\|_{ W^T W} \leq \frac{2\epsilon \tau_r}{1-\epsilon} \|Ax_*-b\|_2. 
\end{align}
Note that $\|Ax_{\tau} - Ax_*\|_2 = \|y_{\tau} - y^*\|_{W^T W}$. Using the triangle inequality, we deduce 
\begin{align}
\|Ax_{\tau} - b\|_2 = \| A x_{\tau} - Ax_* + Ax_* -b  \|_2 \leq \left( 1 + \frac{2\epsilon \tau_r}{1-\epsilon} \right) \|Ax_*-b\|_2.
\end{align}
\end{proof}

\section{Calibration of sketching dimension for dense solvers}

In Figures \ref{fig::new_blen_engineering}, \ref{fig::new_blen_noCPQR_engineering}, \ref{fig::blen_engineering} and \ref{fig::LSRN_engineering},
we display the calibration results for finding a good default value for $m$, the size of the hashing sketching matrix, in 
\solverNameDense{}, Blendenpik and LSRN.

\renewcommand{\mysize}{0.3}

\calibrationSixFigures{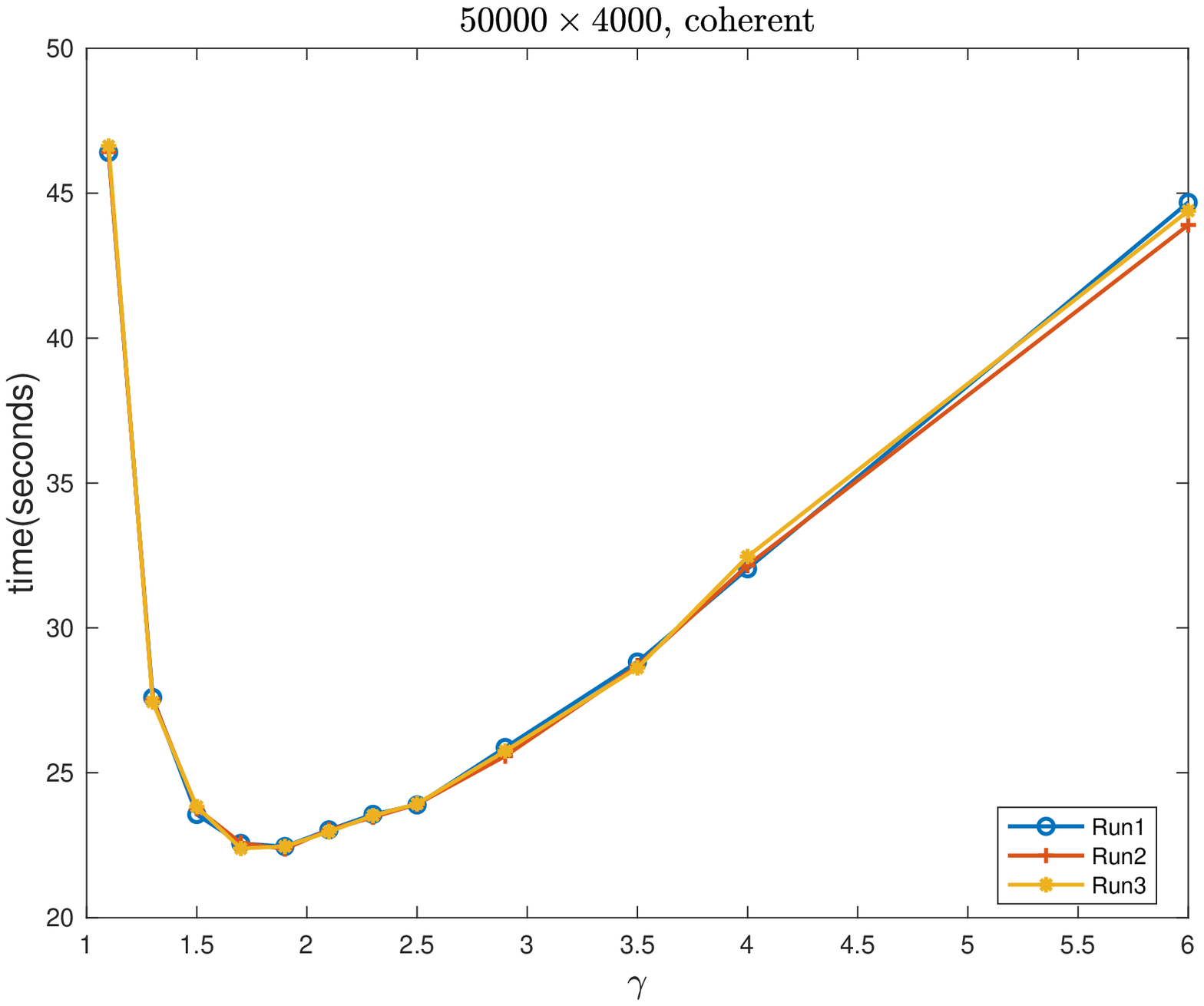}
{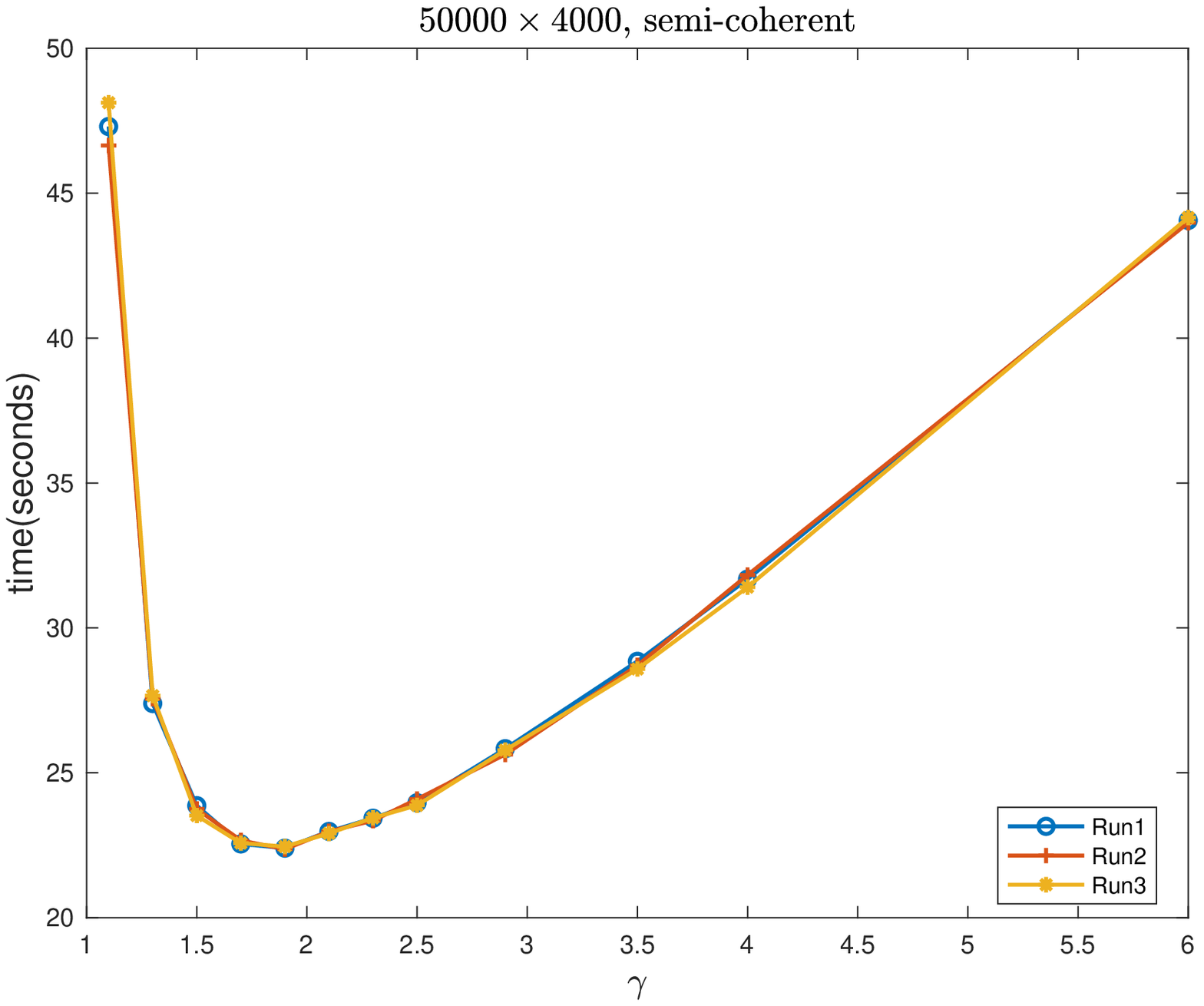}
{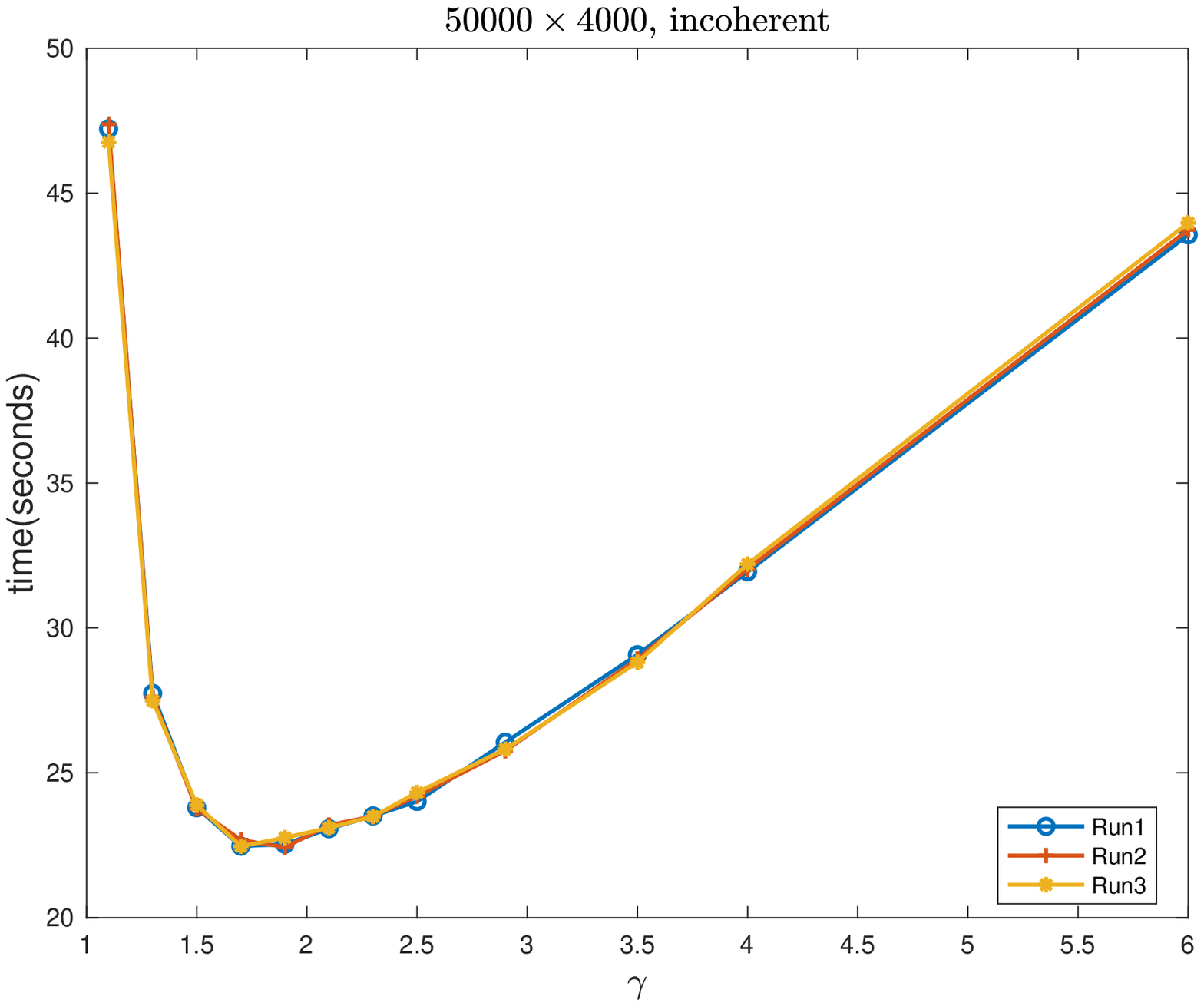}
{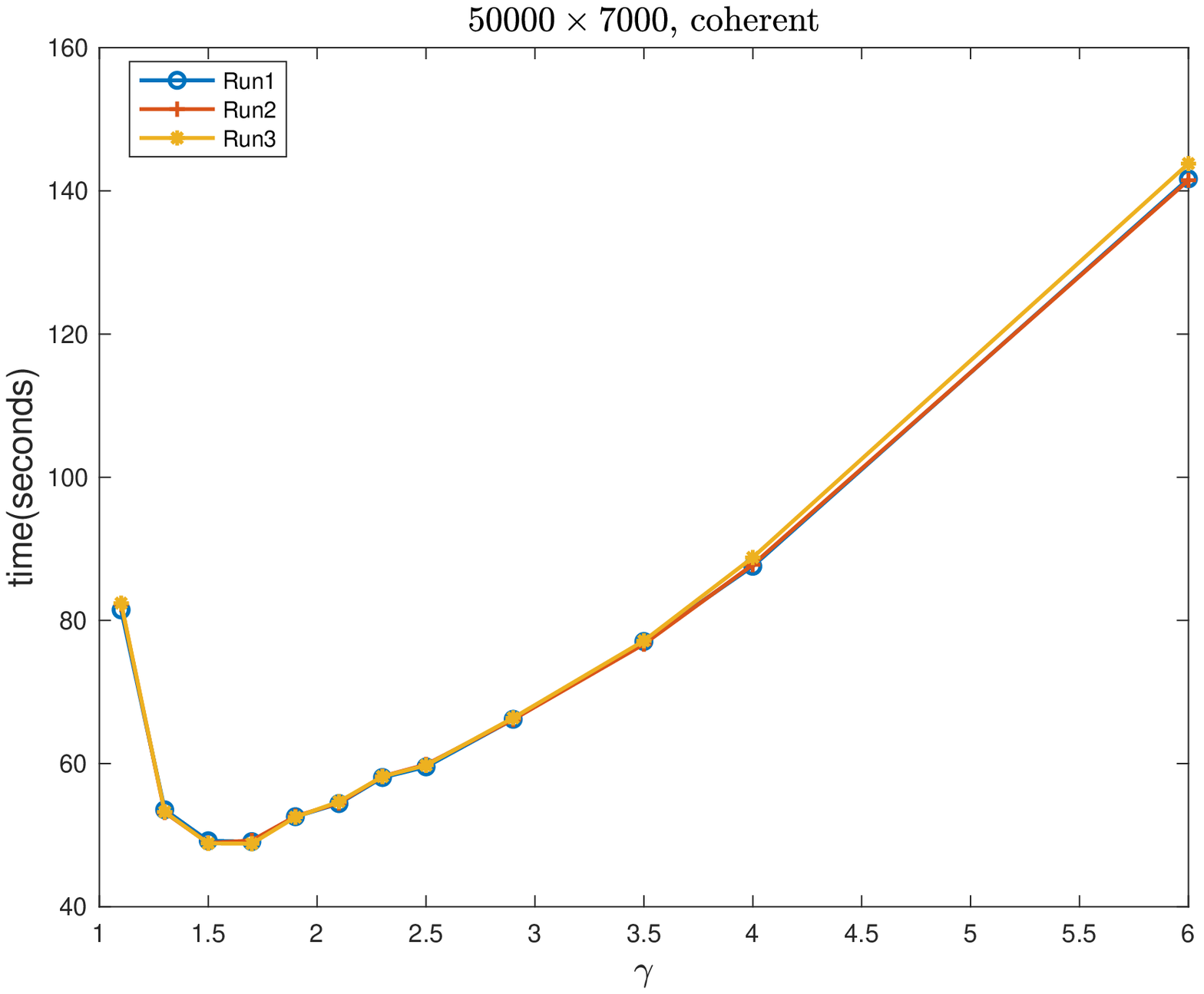}
{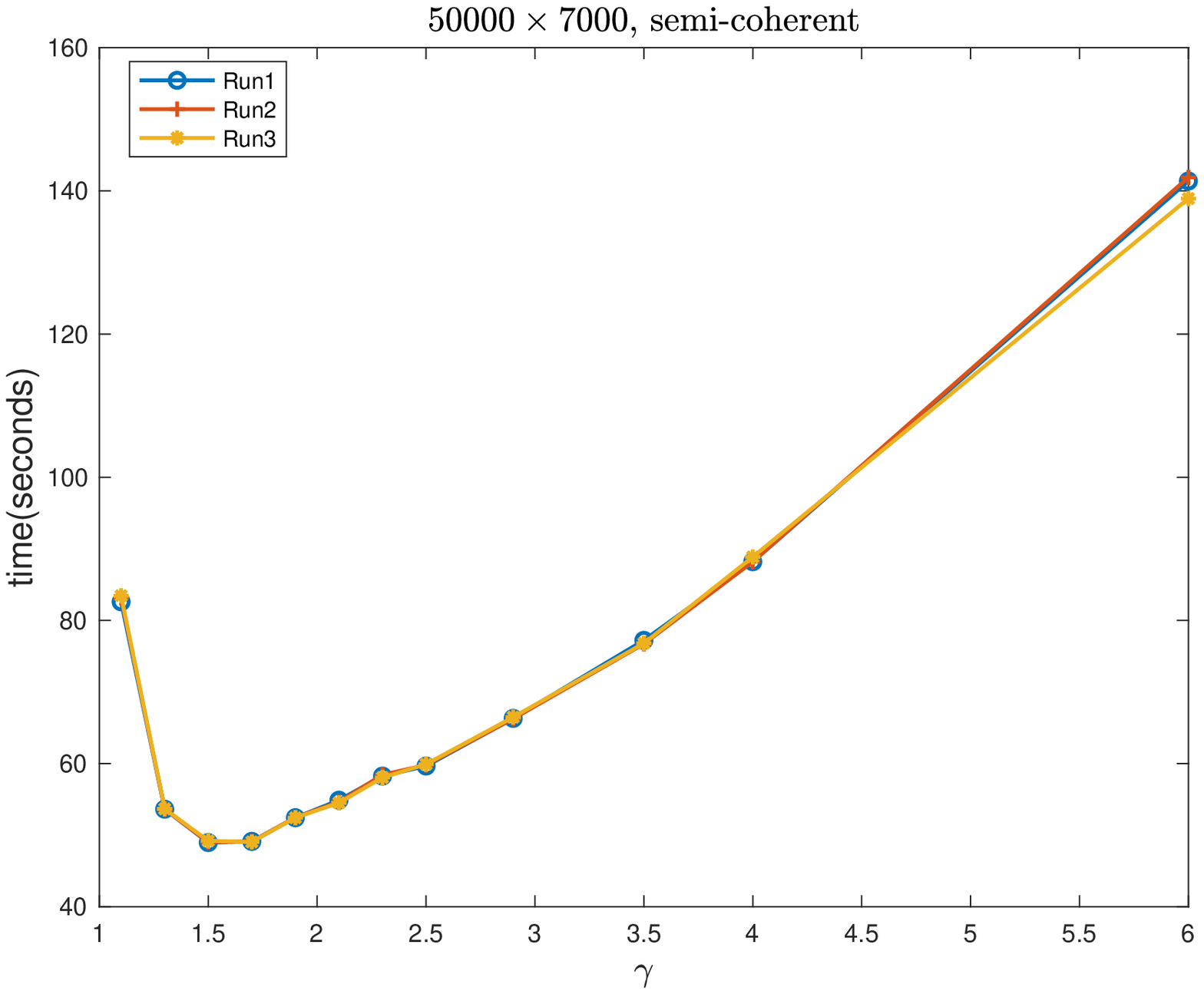}
{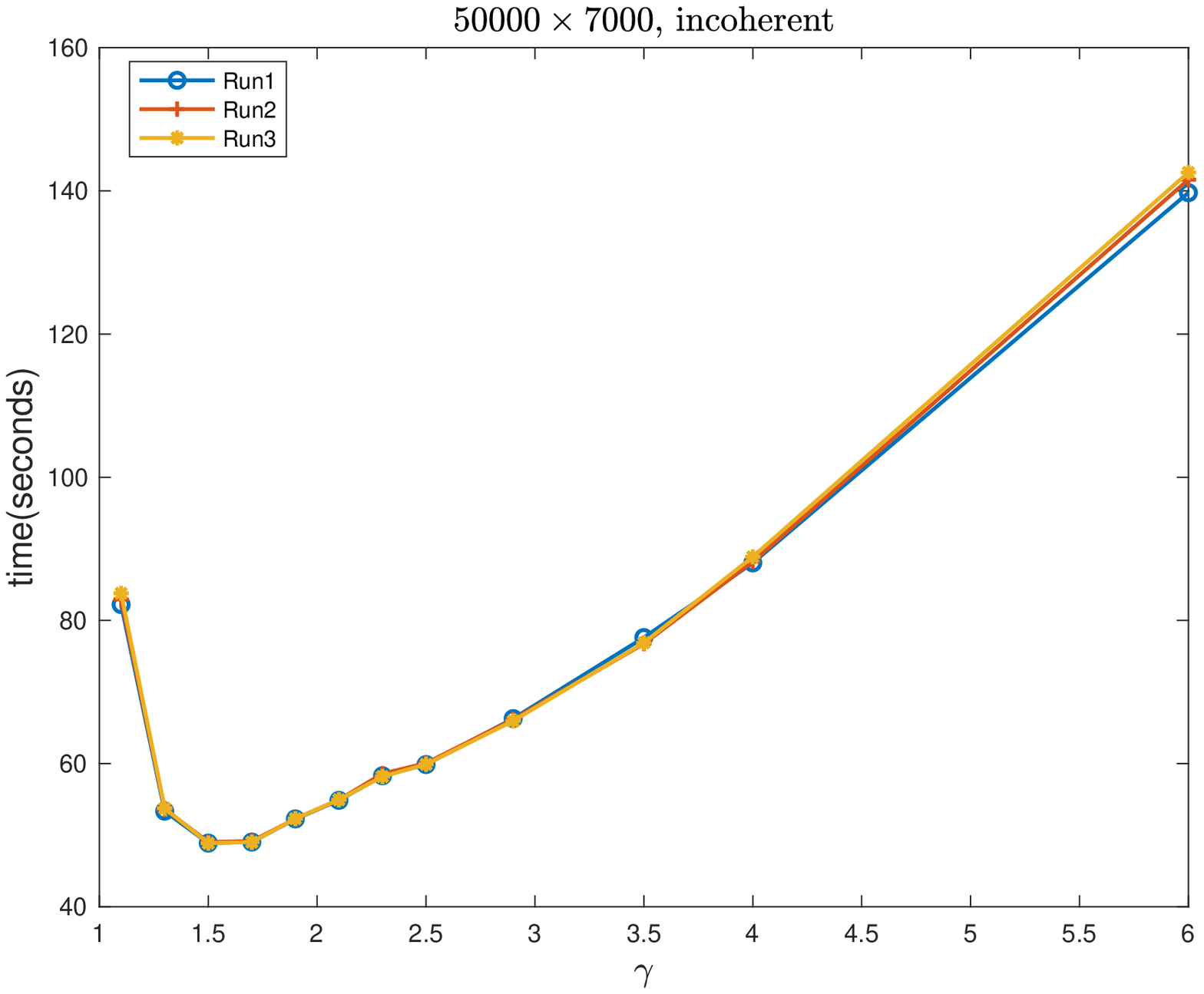}
{\calibrationDenseCaptionSentenceOne{\solverNameDense{}}. \calibrationDenseCaptionSentenceTwo{\solverNameDense{}}. \calibrationDenseCaptionSentenceThree{$
\gamma=1.7$}.}
{fig::new_blen_engineering}

\calibrationSixFigures{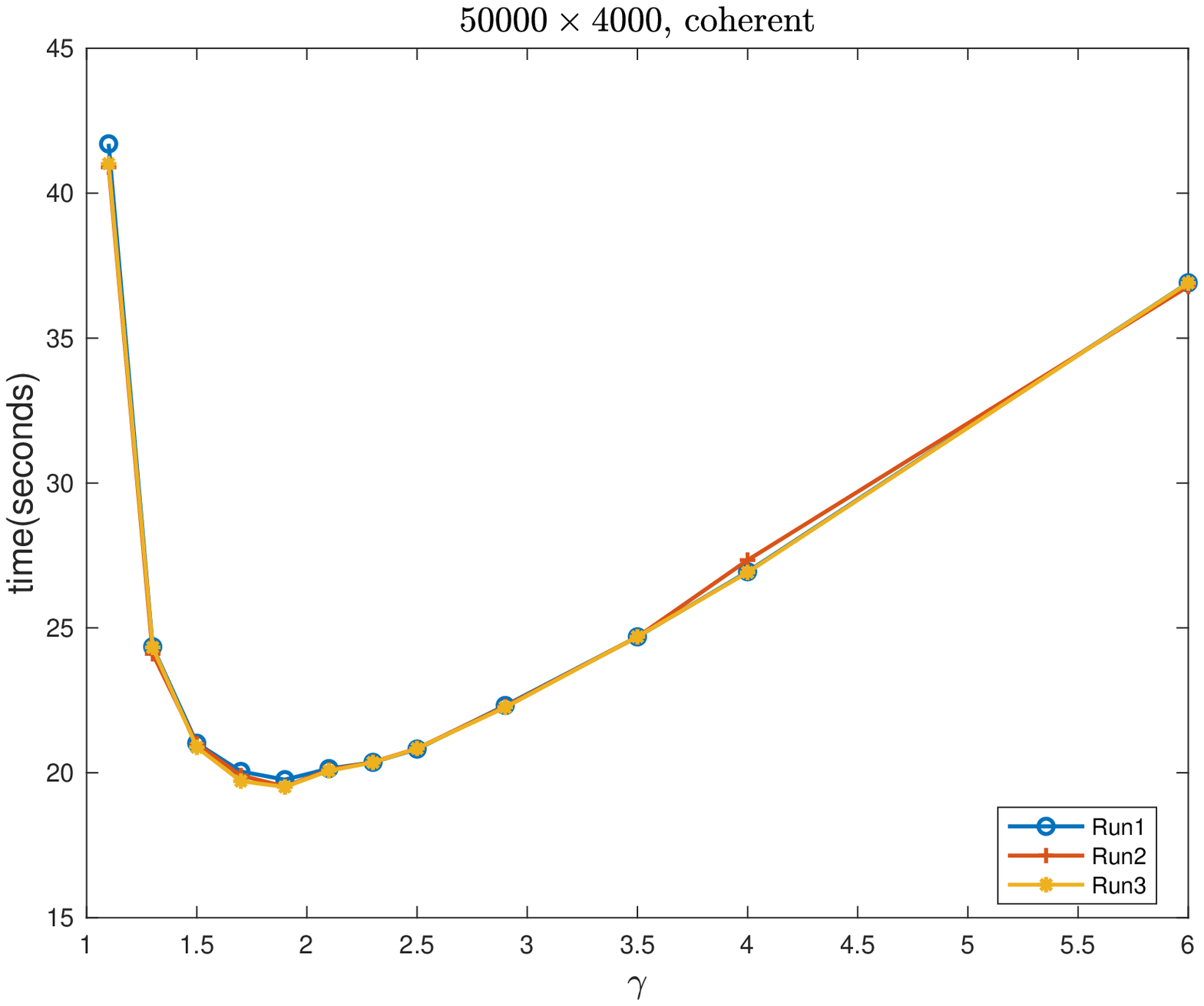}
{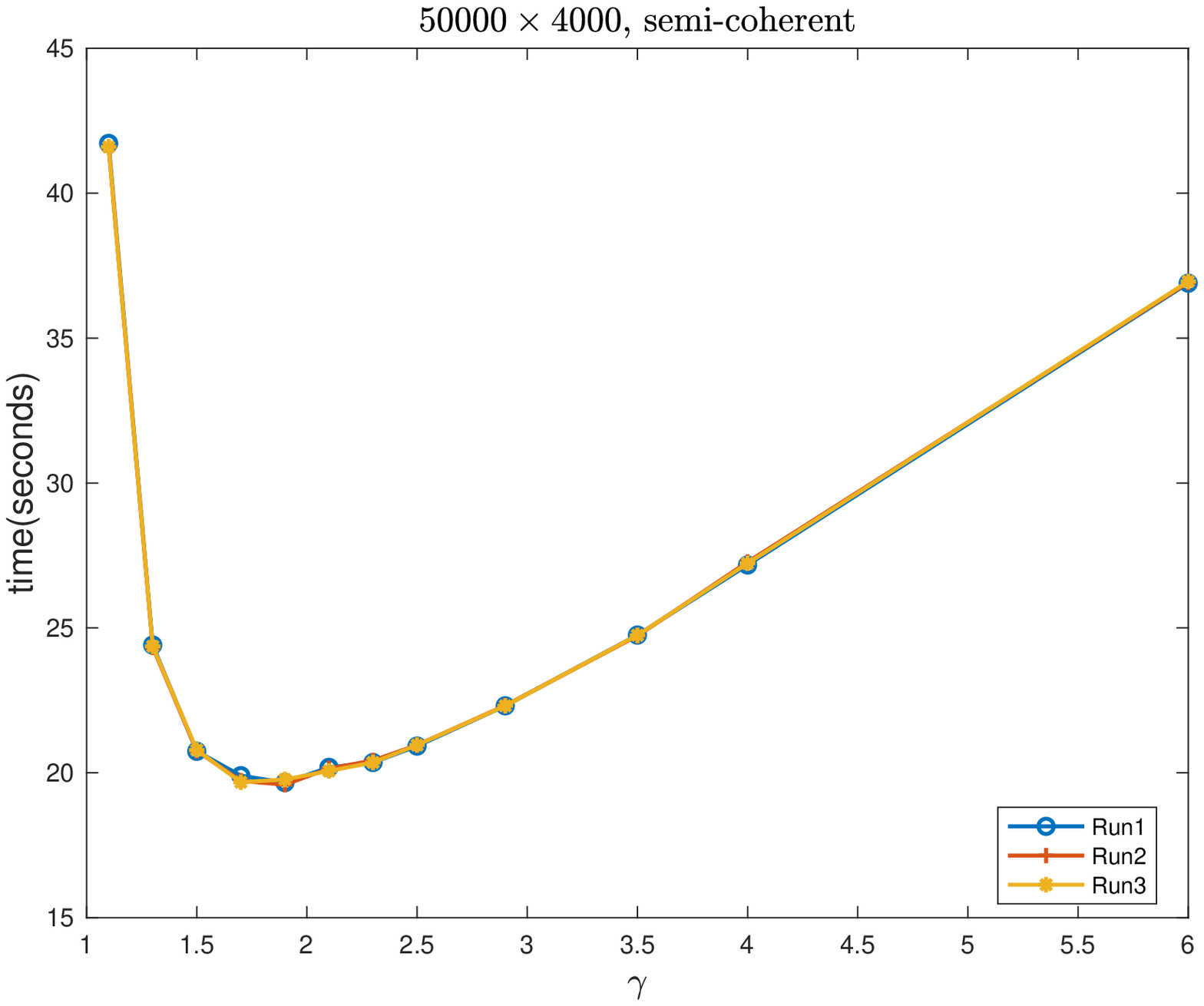}
{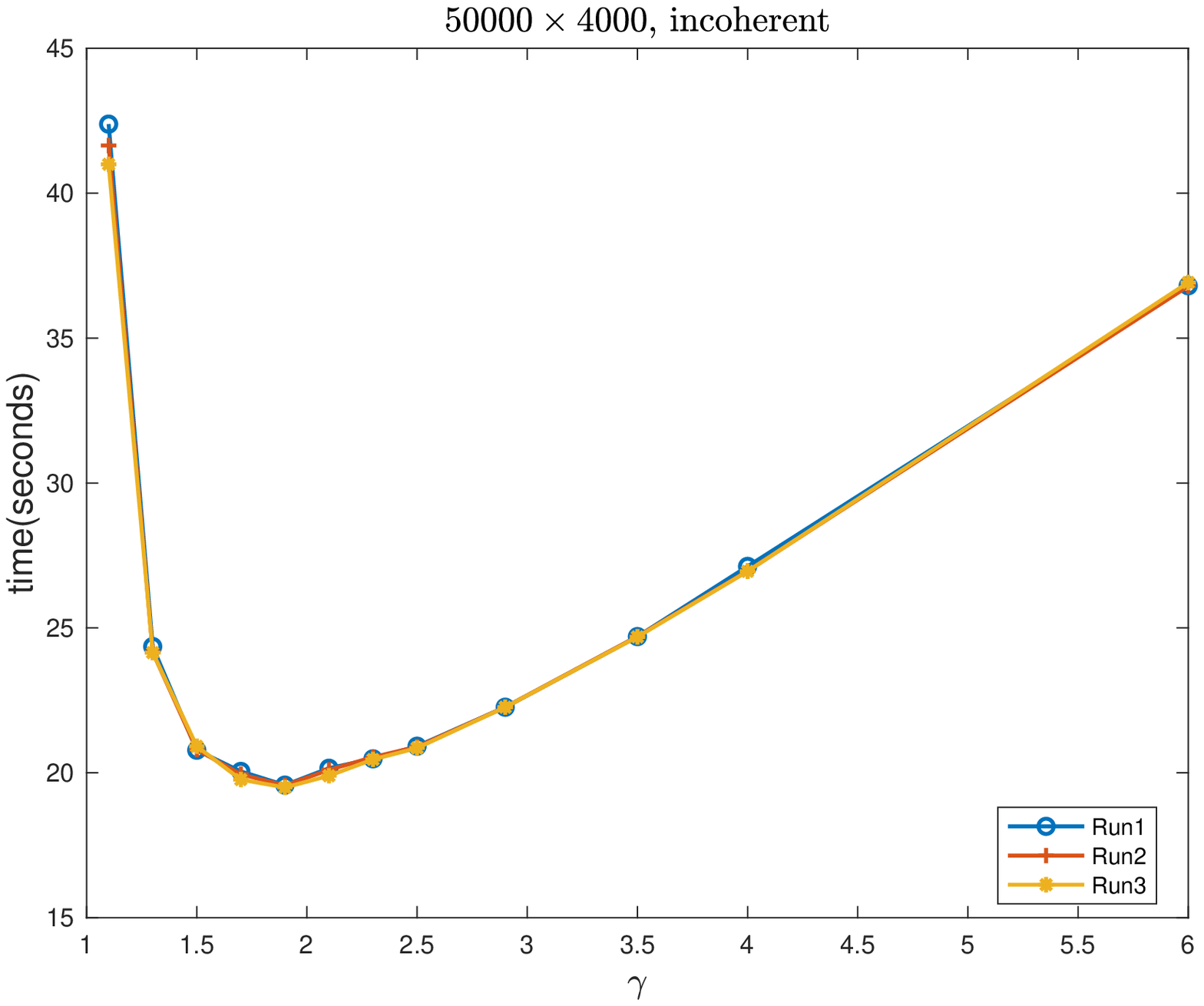}
{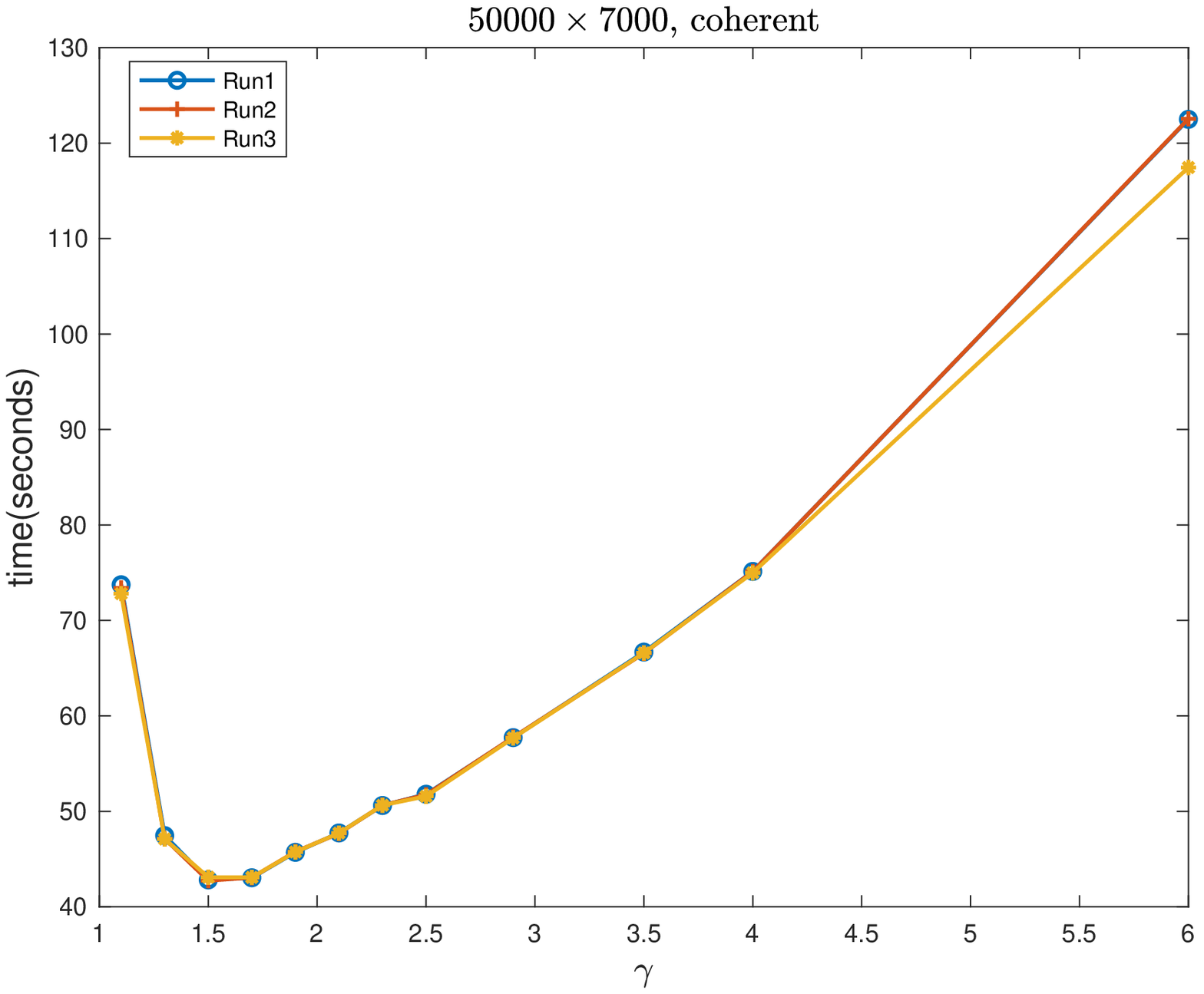}
{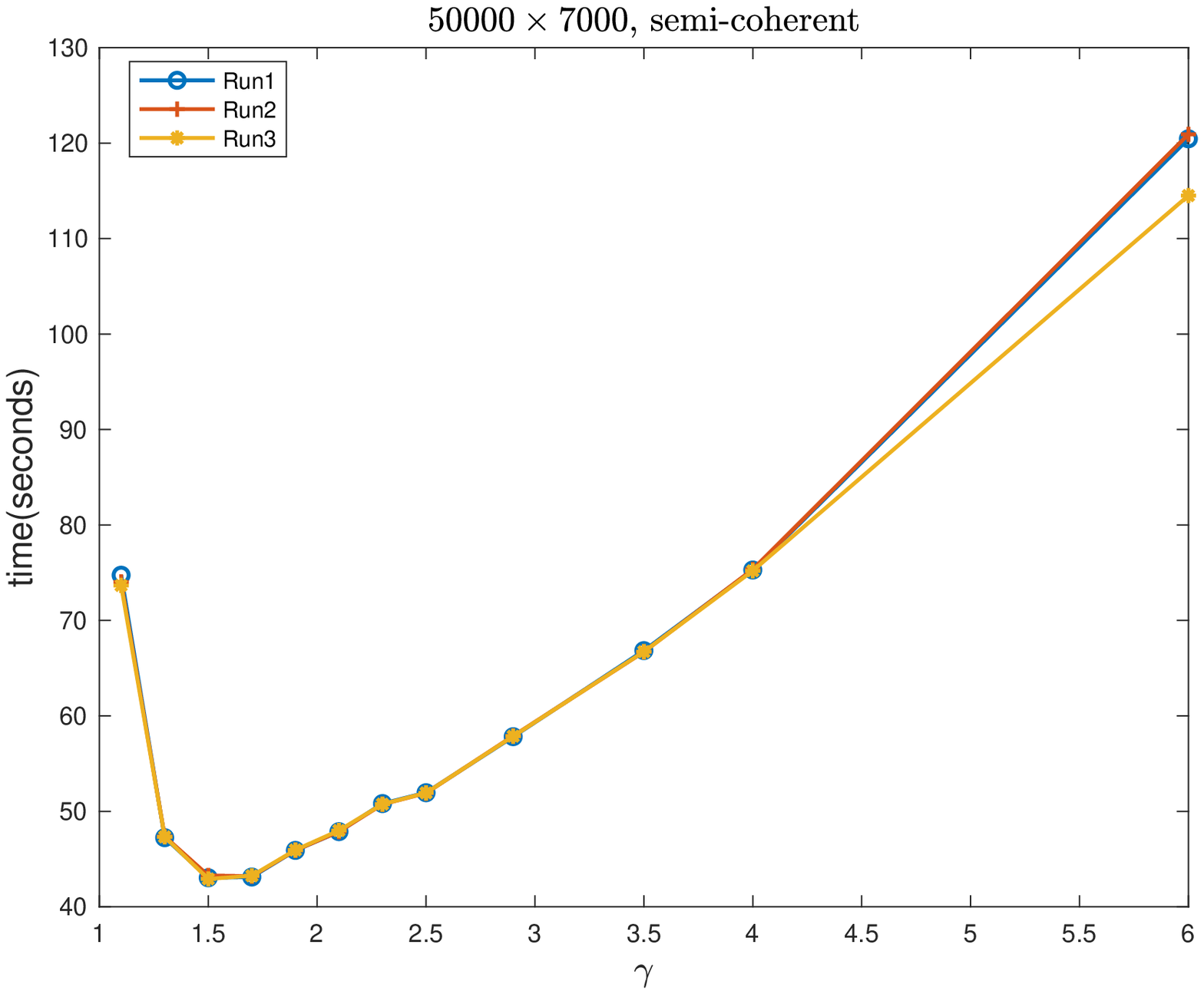}
{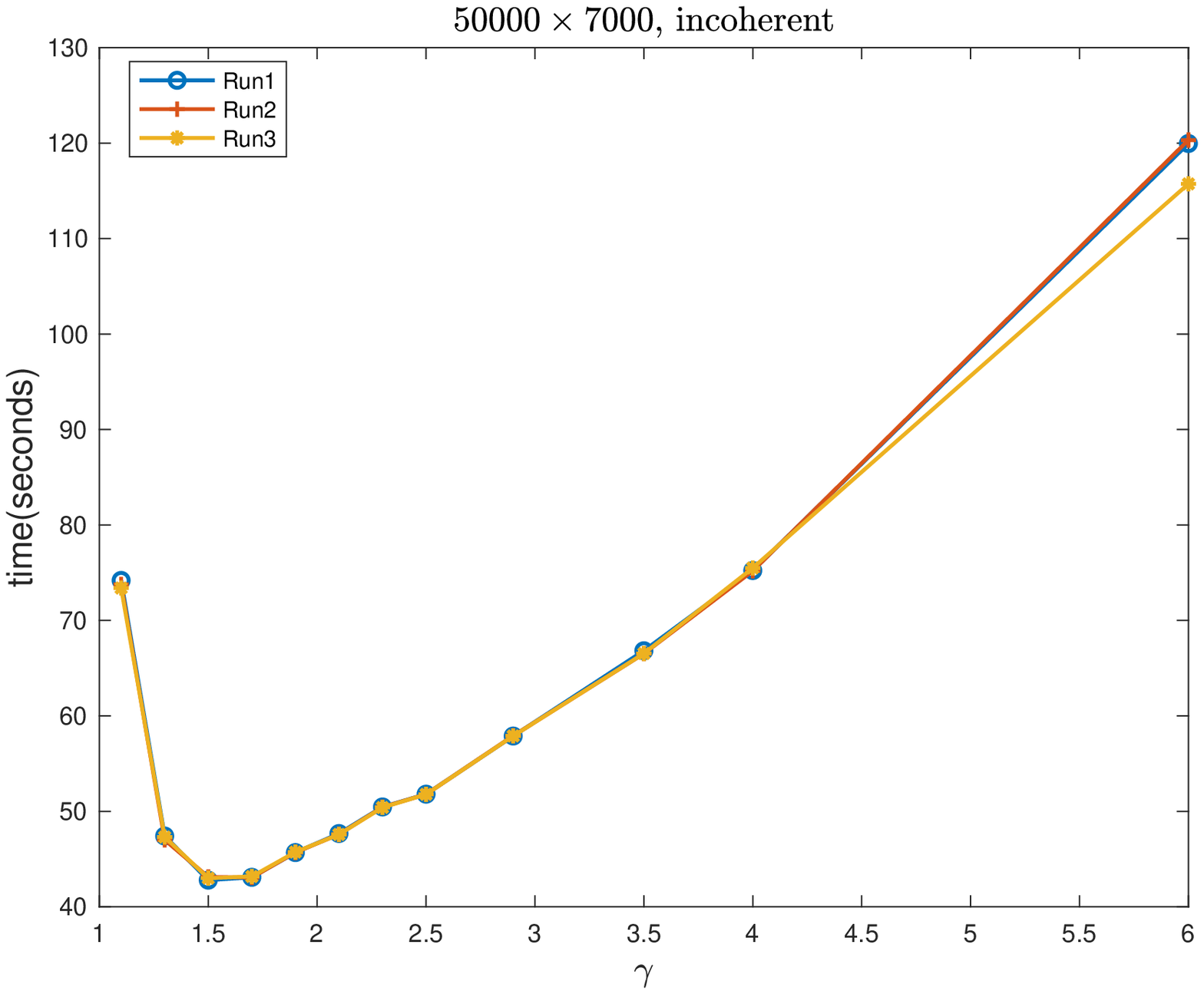}
{
\calibrationDenseCaptionSentenceOne{\solverNameDense{} without R-CPQR}. \calibrationDenseCaptionSentenceTwo{\solverNameDense{} without R-CPQR}. Note that using LAPACK QR instead of R-CPQR results in slightly shorter runtime (according to Figure \ref{fig::new_blen_engineering}). \calibrationDenseCaptionSentenceThree{$\gamma=1.7$}.}
{fig::new_blen_noCPQR_engineering}

\calibrationSixFigures{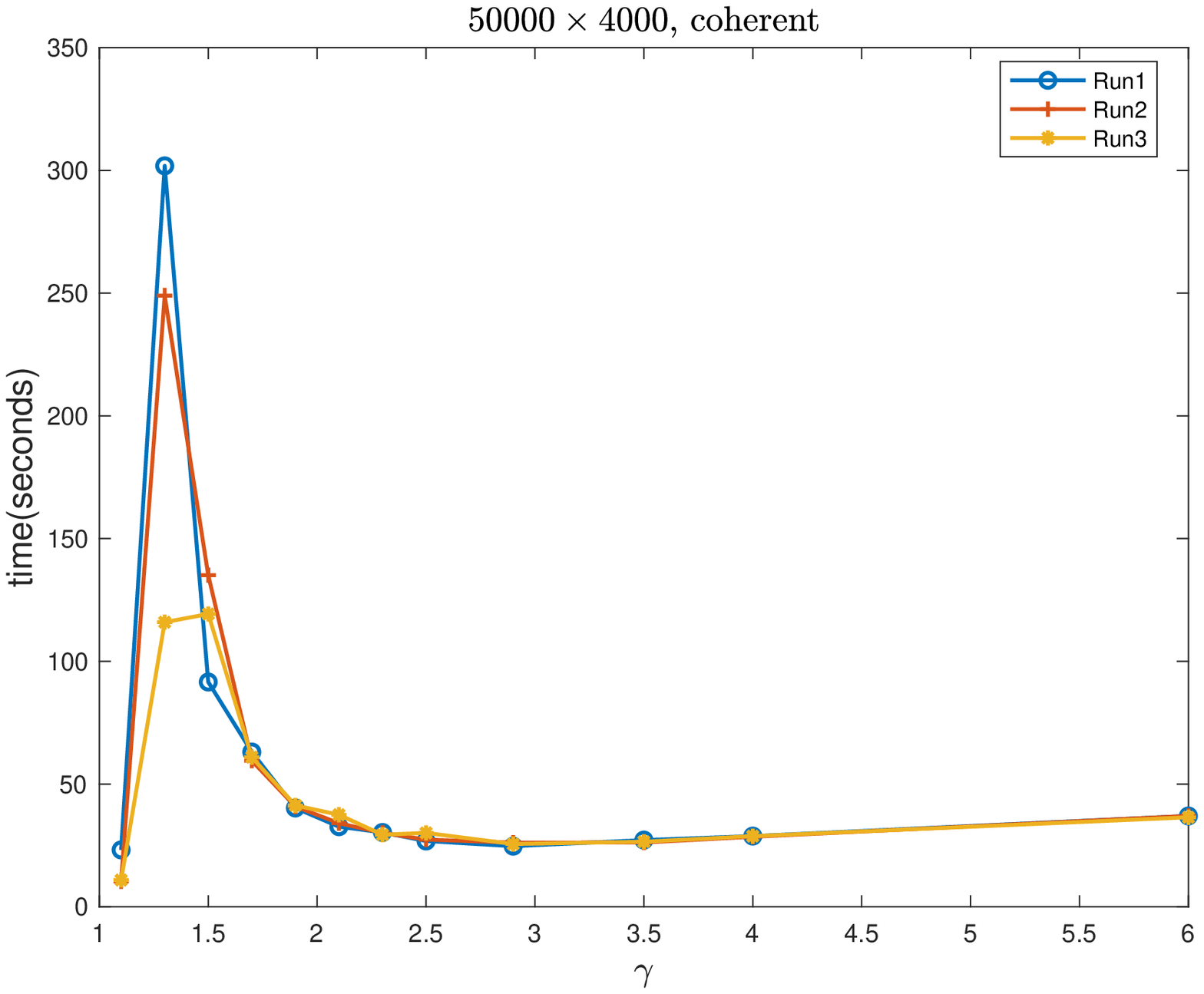}
{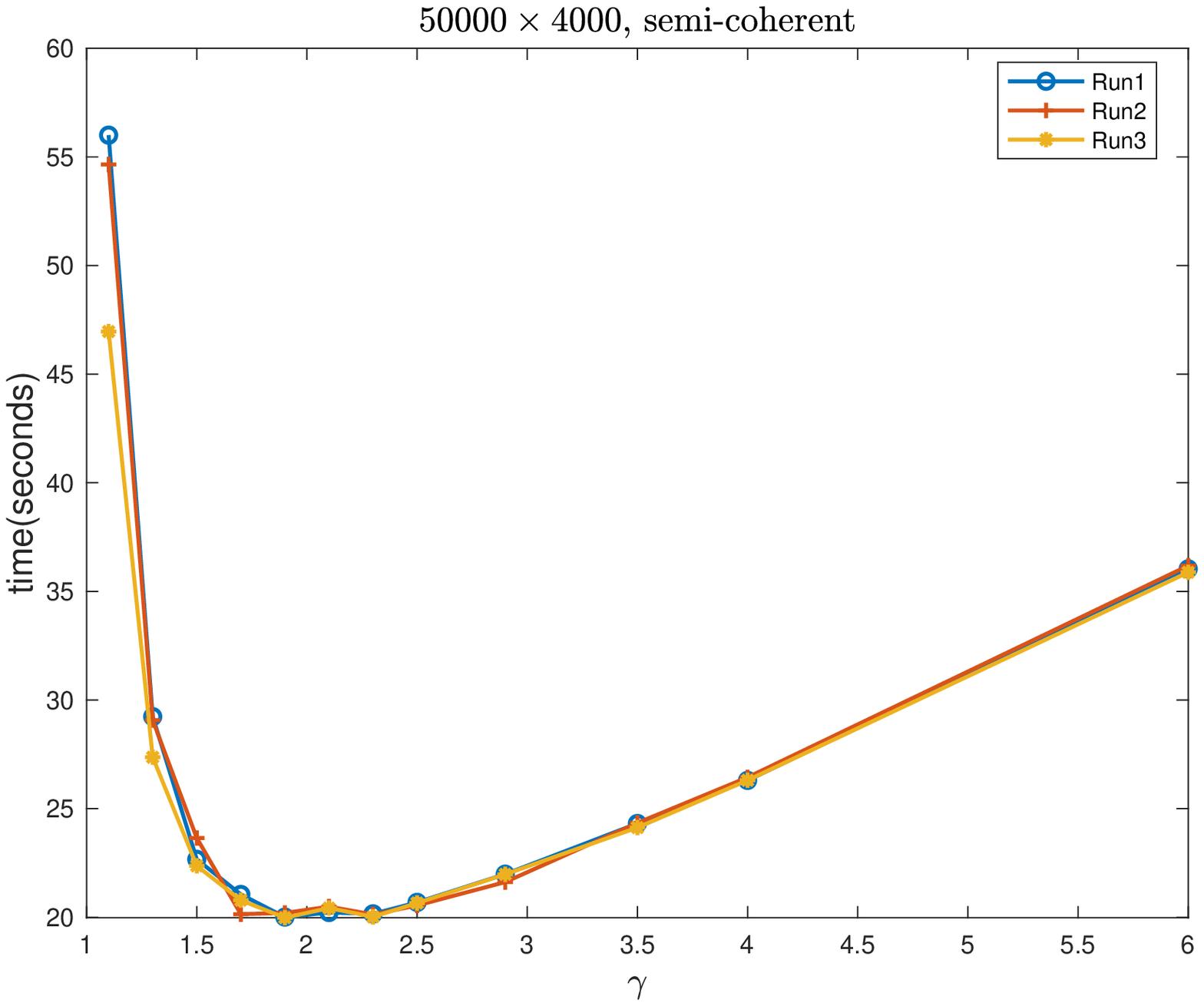}
{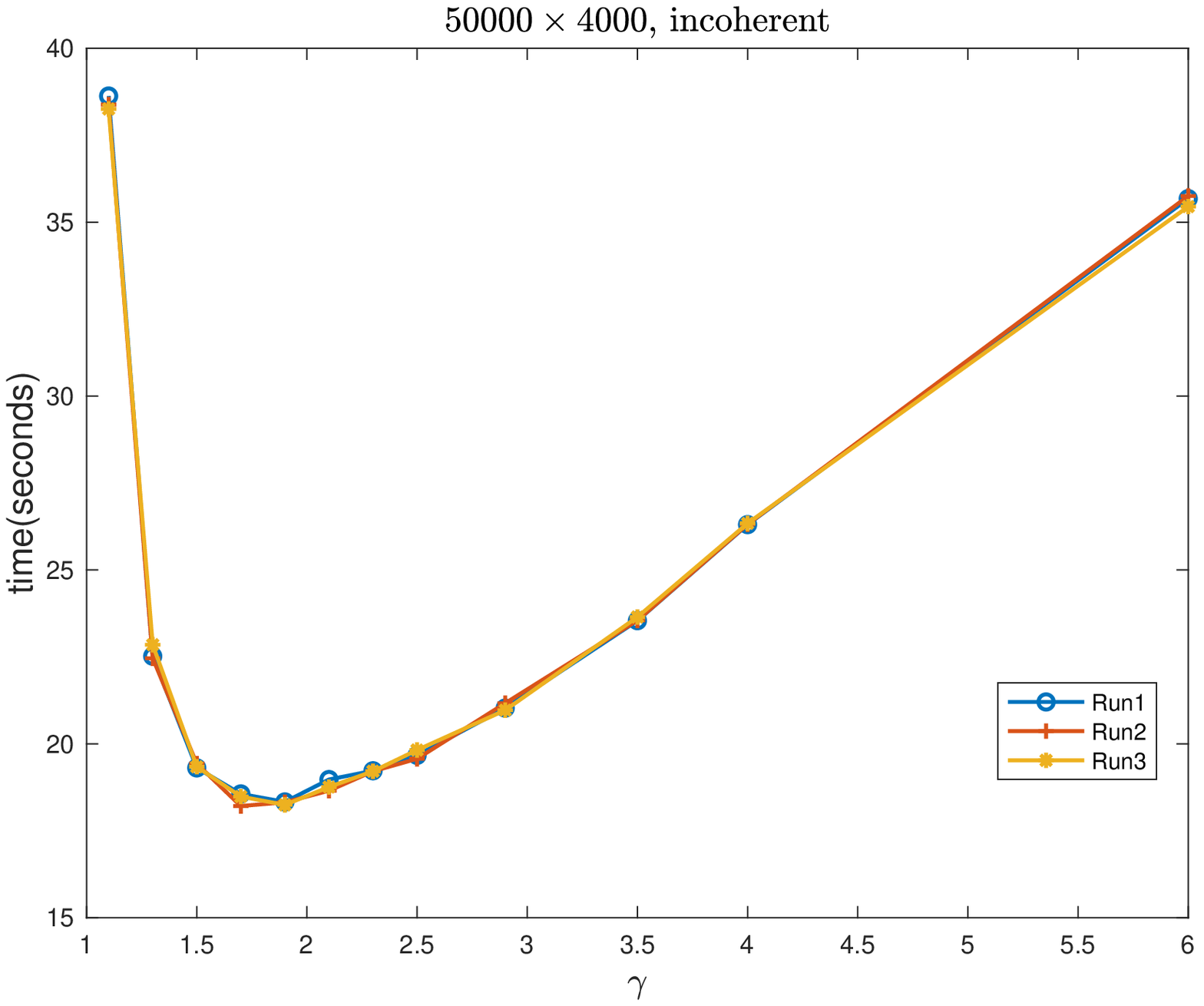}
{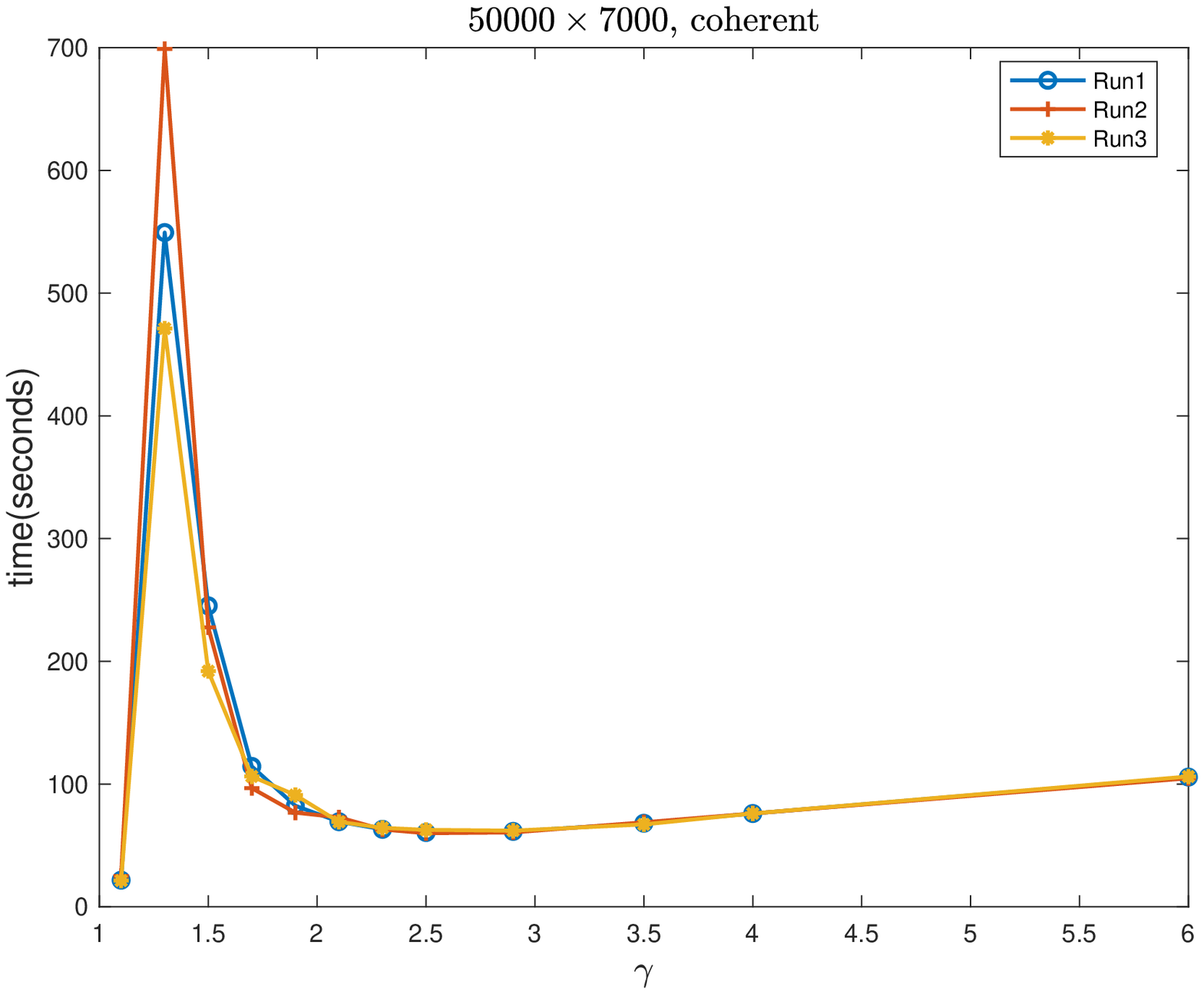}
{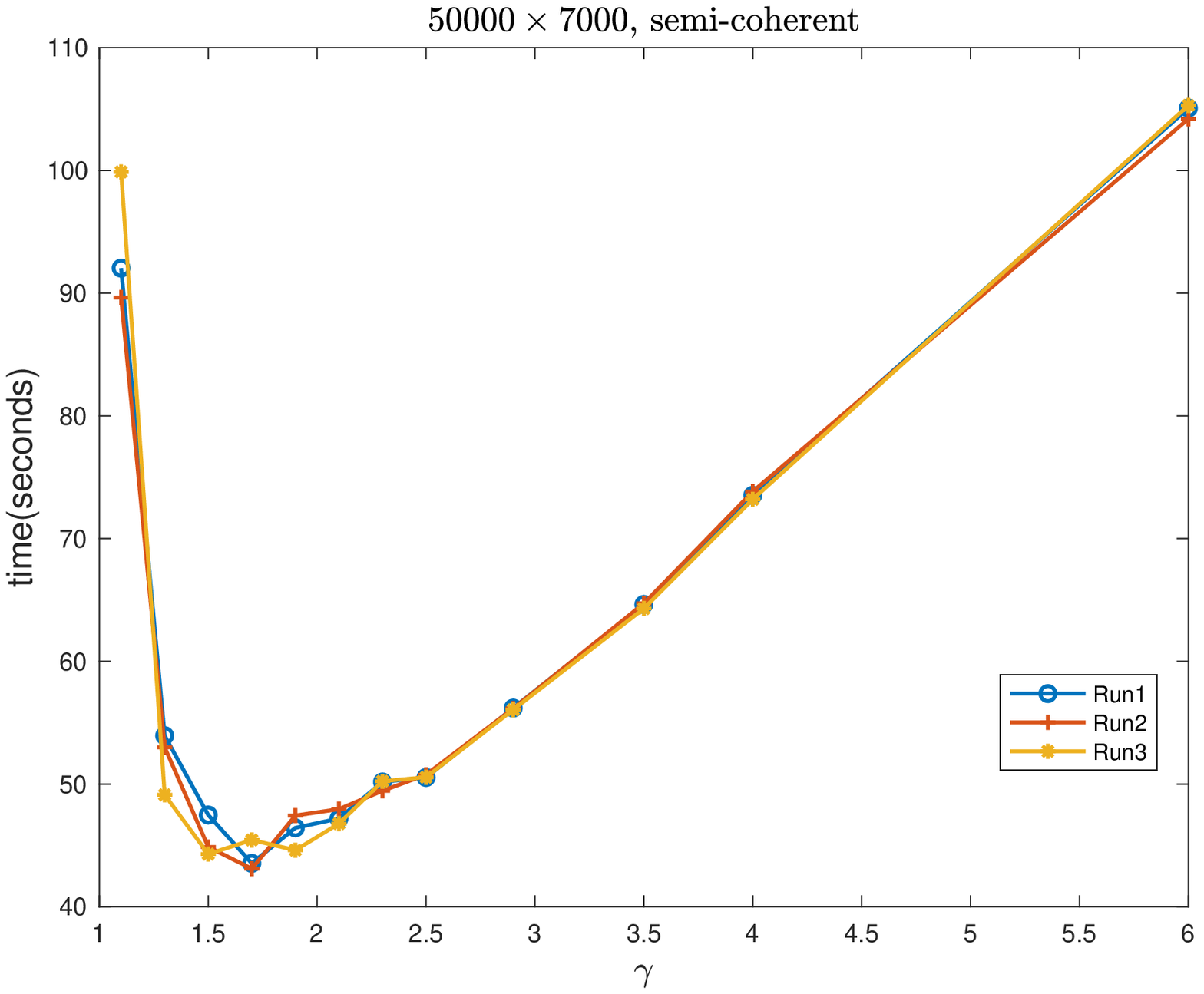}
{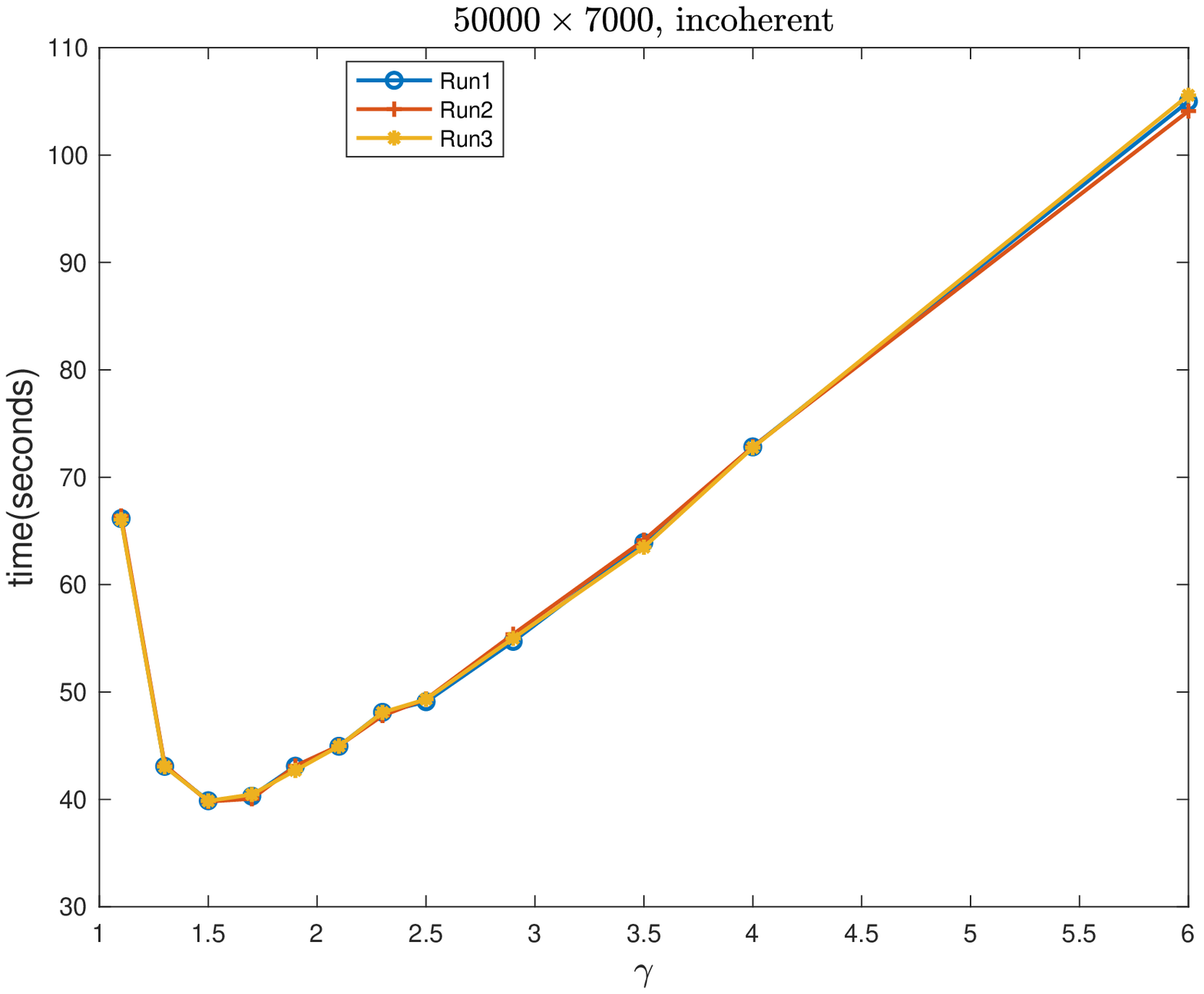}
{\calibrationDenseCaptionSentenceOne{Blendenpik}. \calibrationDenseCaptionSentenceTwo{Blendenpik}. Note that Blendenpik handles coherent dense $A$ significantly less well than \solverNameDense{}. \calibrationDenseCaptionSentenceThree{$\gamma=2.2$}.}
{fig::blen_engineering}

\calibrationSixFigures{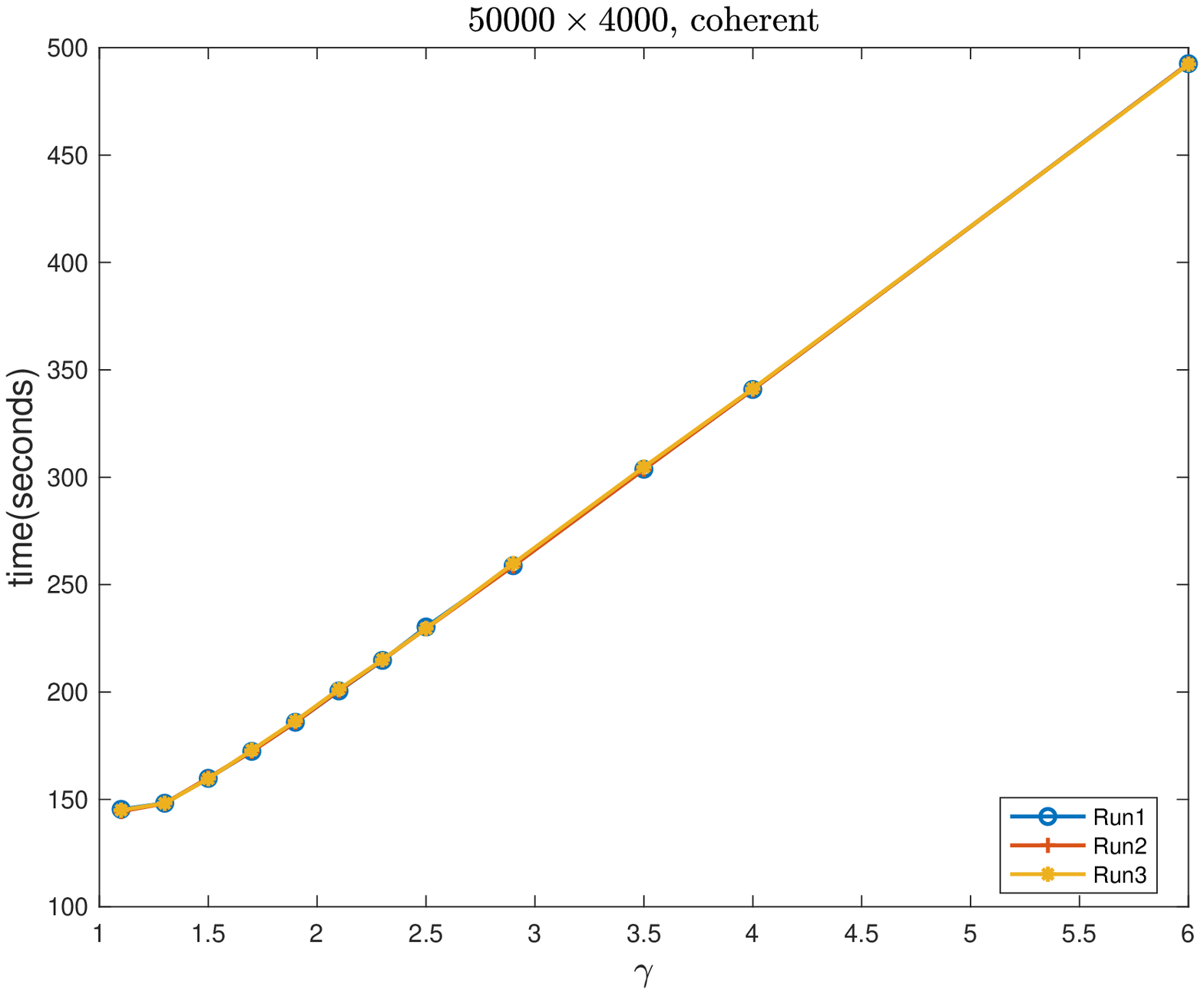}
{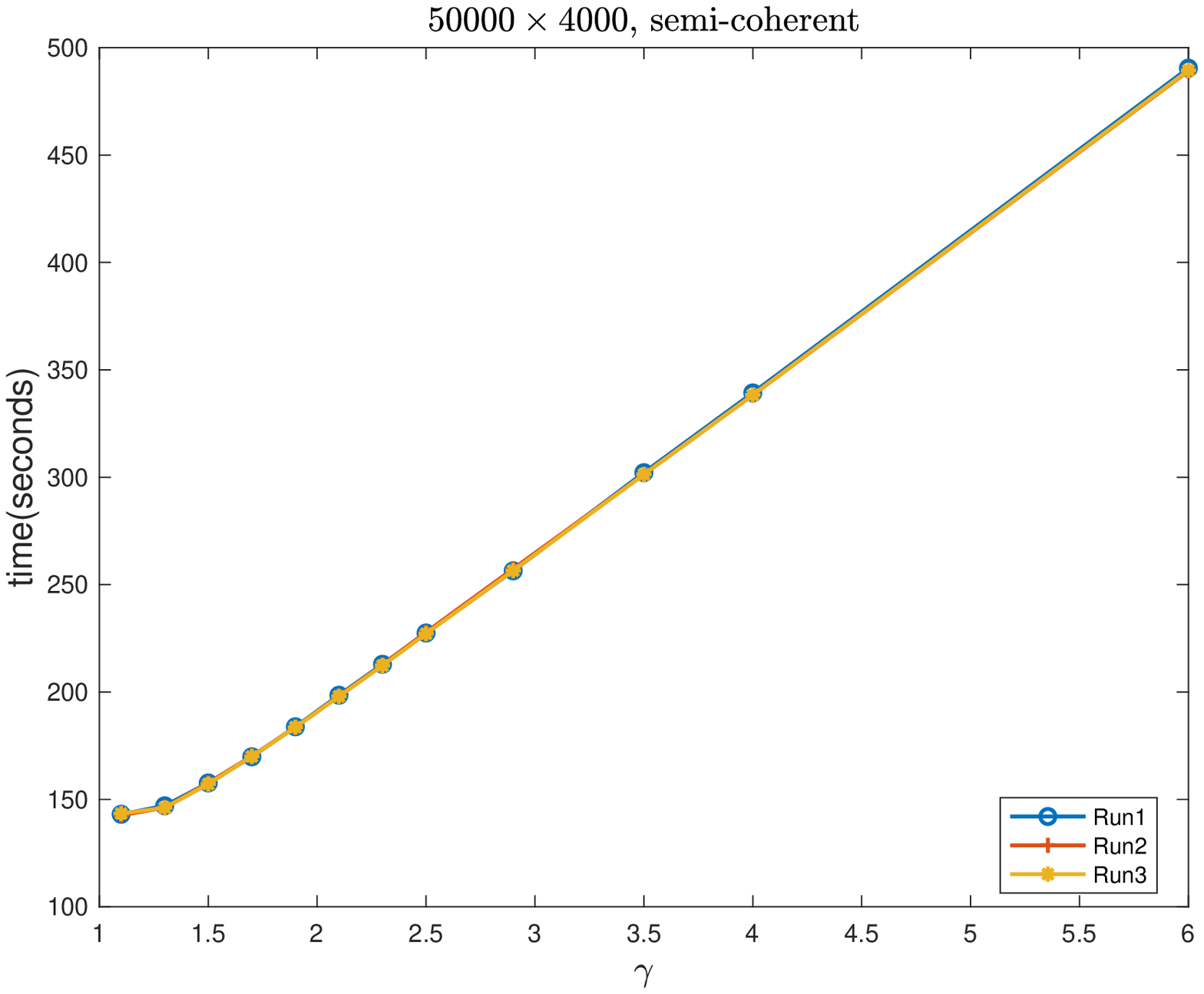}
{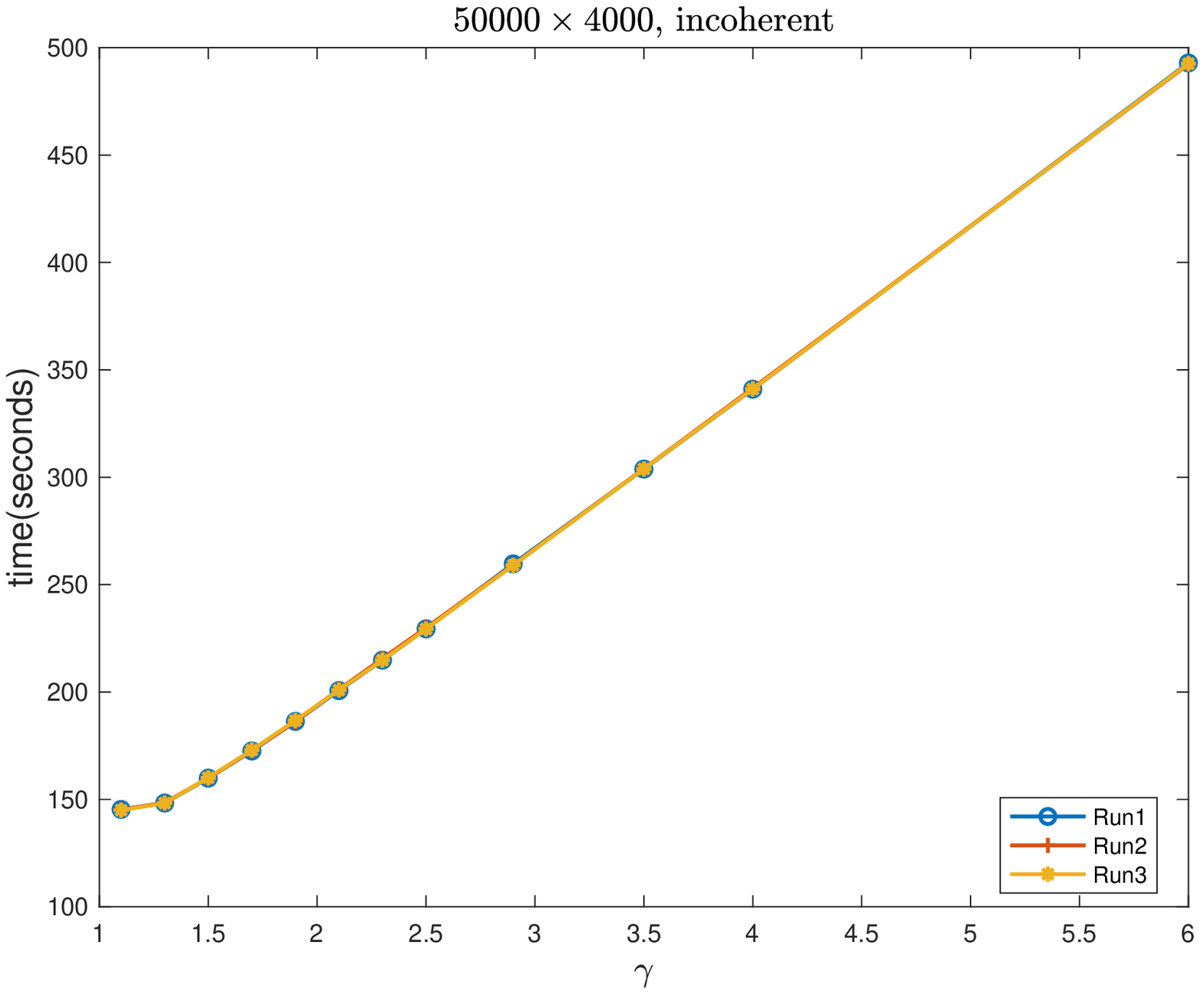}
{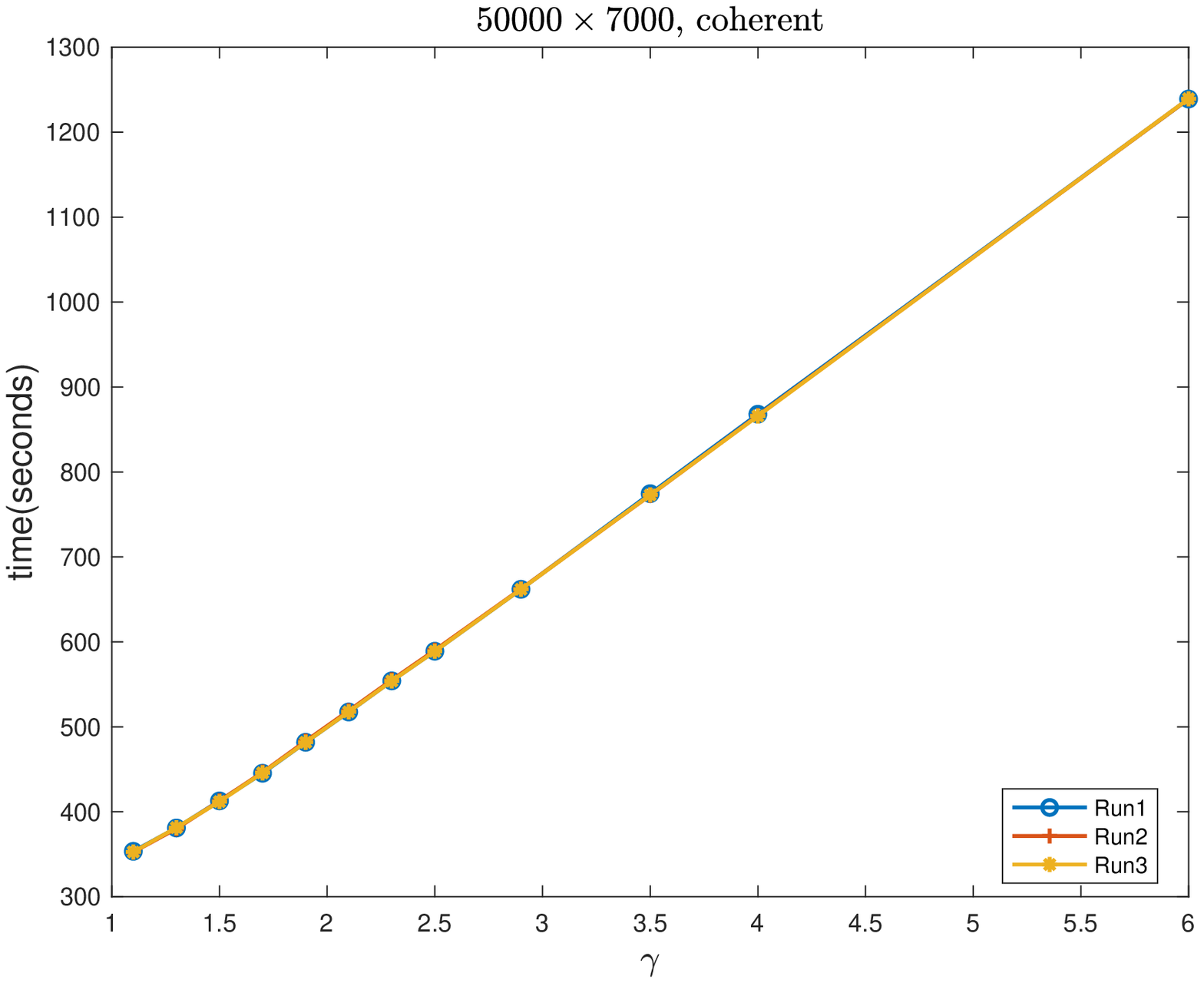}
{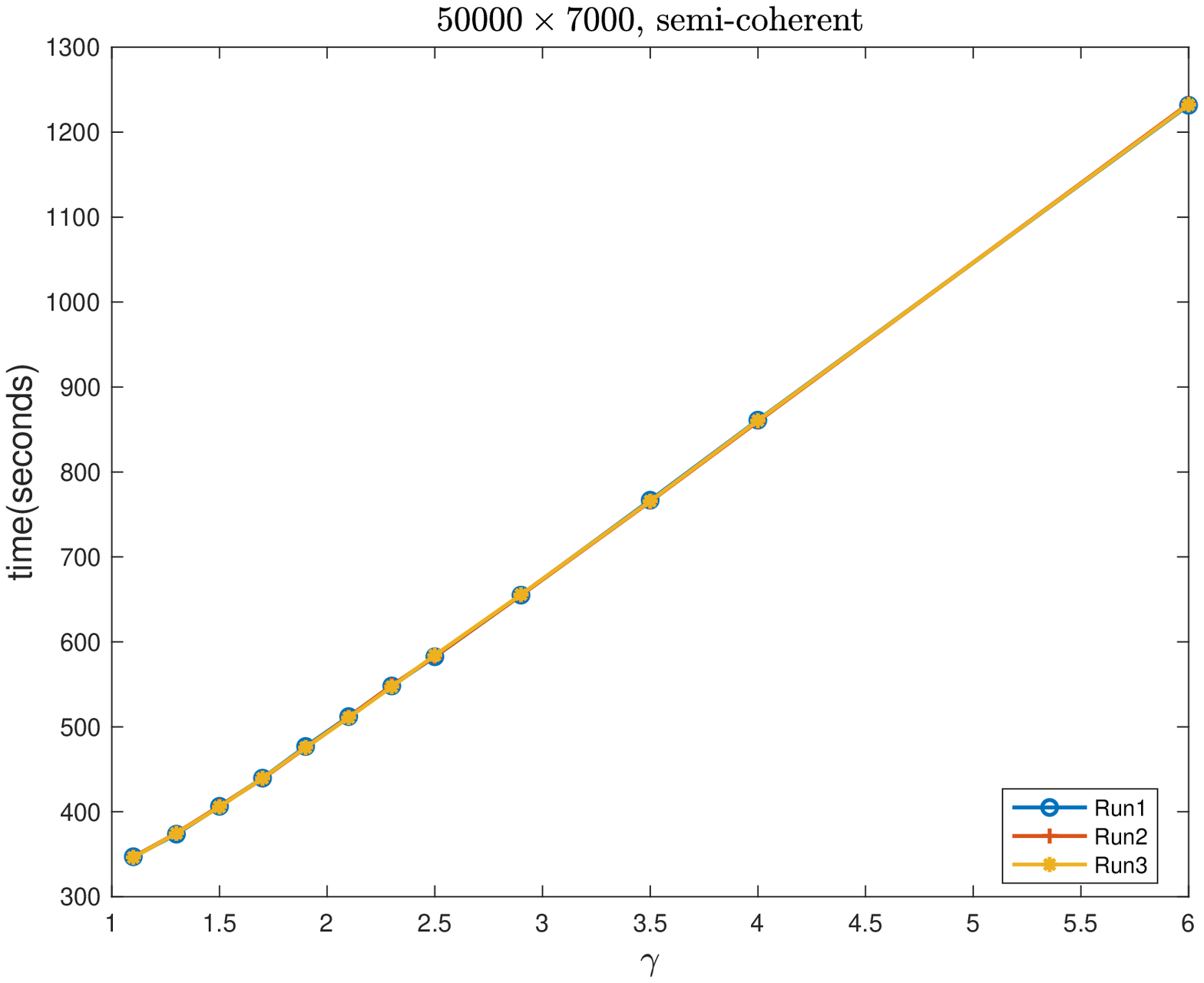}
{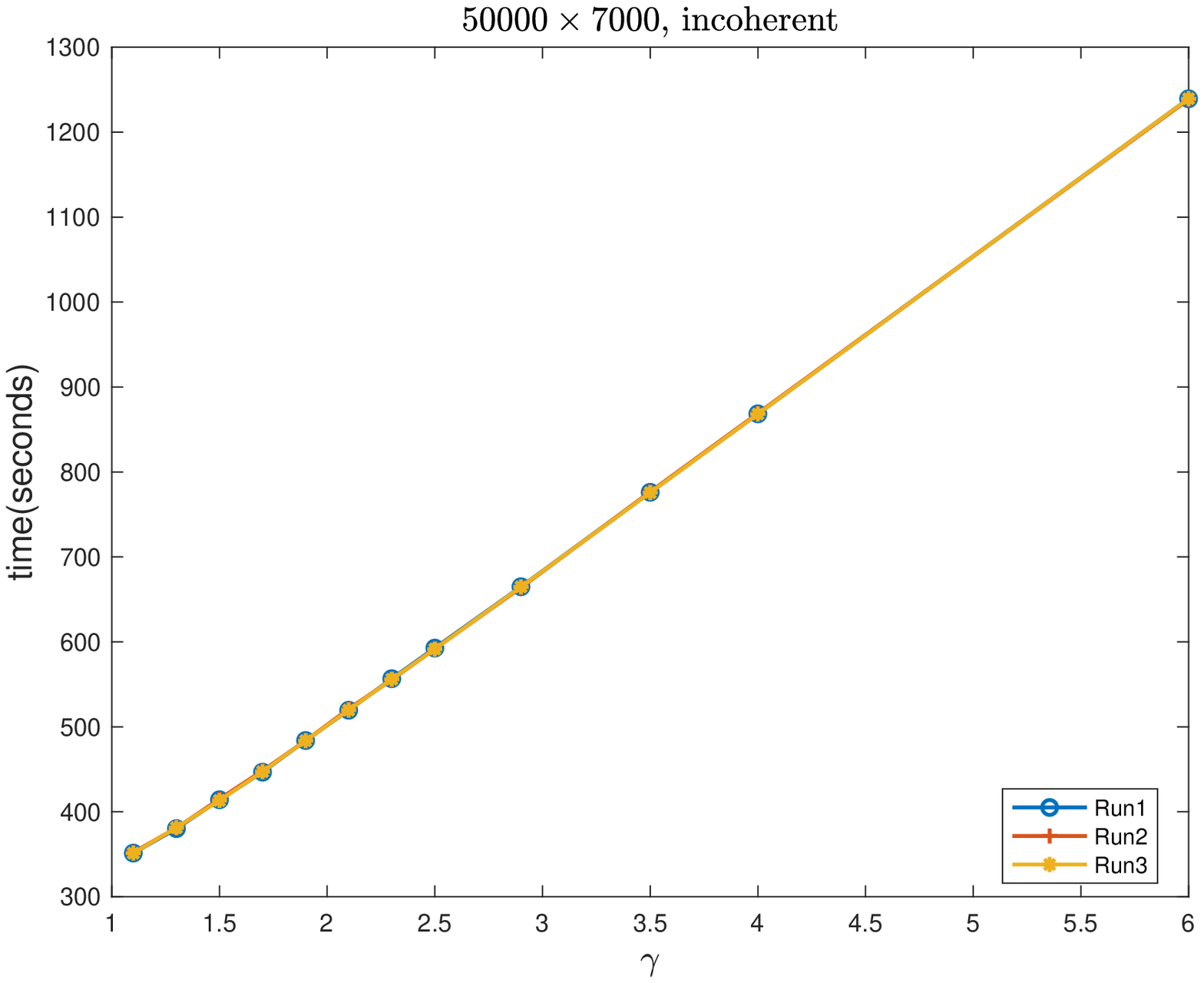}
{\calibrationDenseCaptionSentenceOne{LSRN}. \calibrationDenseCaptionSentenceTwo{LSRN}. Note that LSRN runs more than 5 times slower comparing to Blendenpik or \solverName{} in the serial testing environment, due to the use of SVD and Gaussian sketching. \calibrationDenseCaptionSentenceThree{$\gamma=1.1$}.} 
{fig::LSRN_engineering}

\section{Calibration of sketching dimension and number of hashing entries for sparse solvers}

In Figures \ref{fig::Ls_qr_engineering_time} and \autoref{fig::Ls_qr_engineering_residual} we display the calibration results for finding a good default value for $m$ (the size of the sketching matrix) and $s$ (the number of nonzero entries per sketching column), for \solverNameSparse{}; similar plots are given for LSRN in \autoref{fig::Ls_lsrn_engineering_time} and \autoref{fig::Ls_lsrn_engineering_residual}, respectively.

\calibrationSixFigures{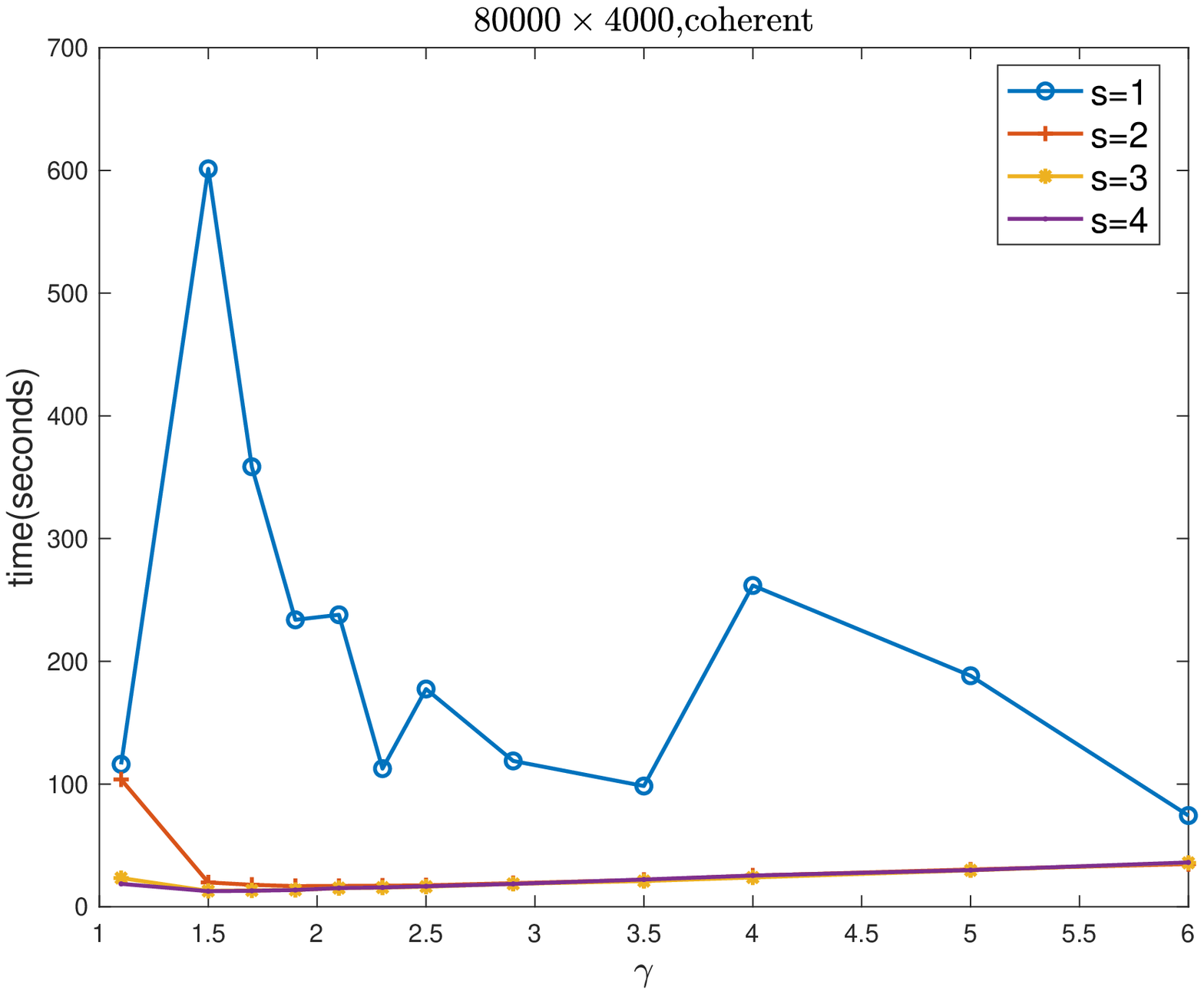}
{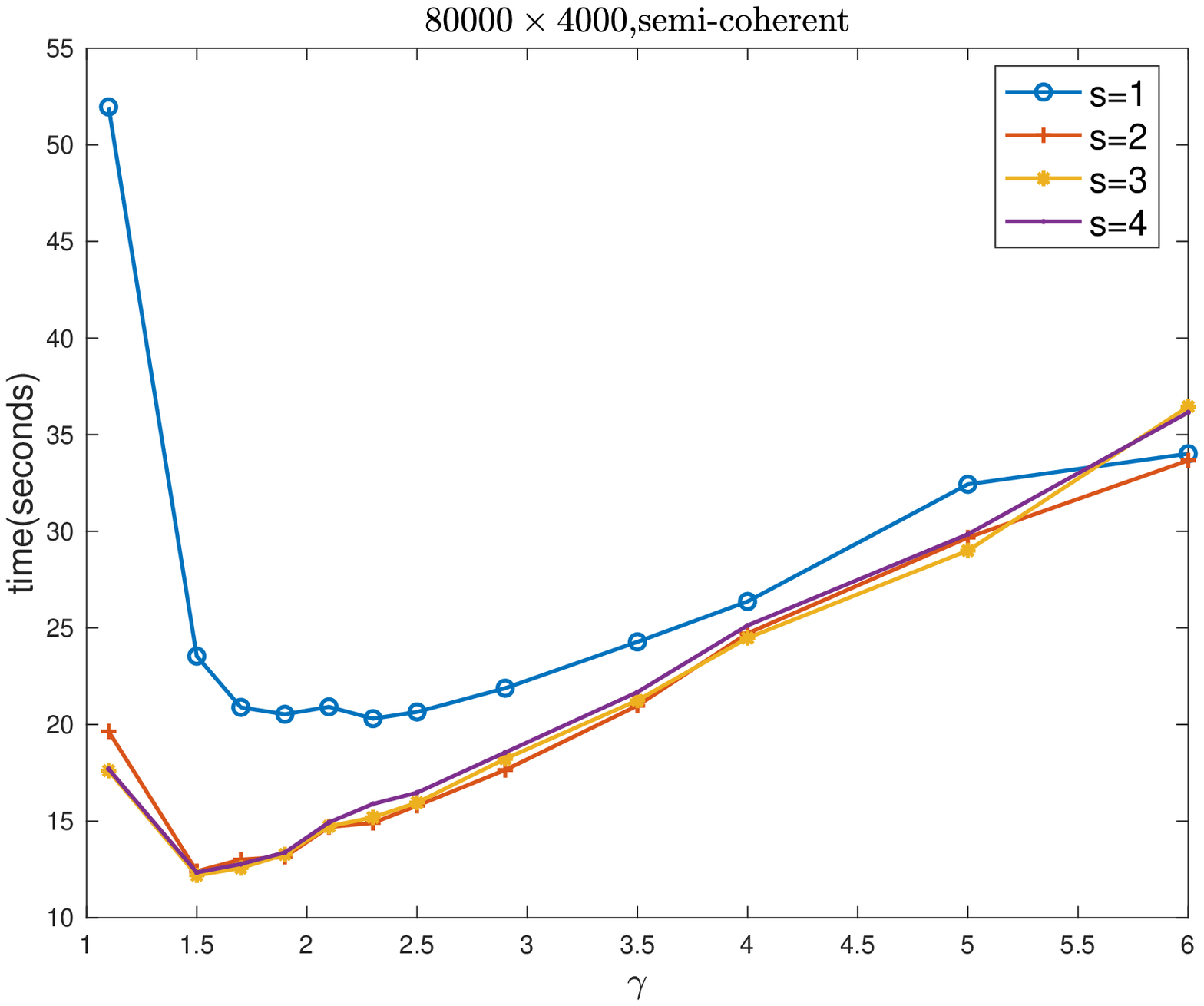}
{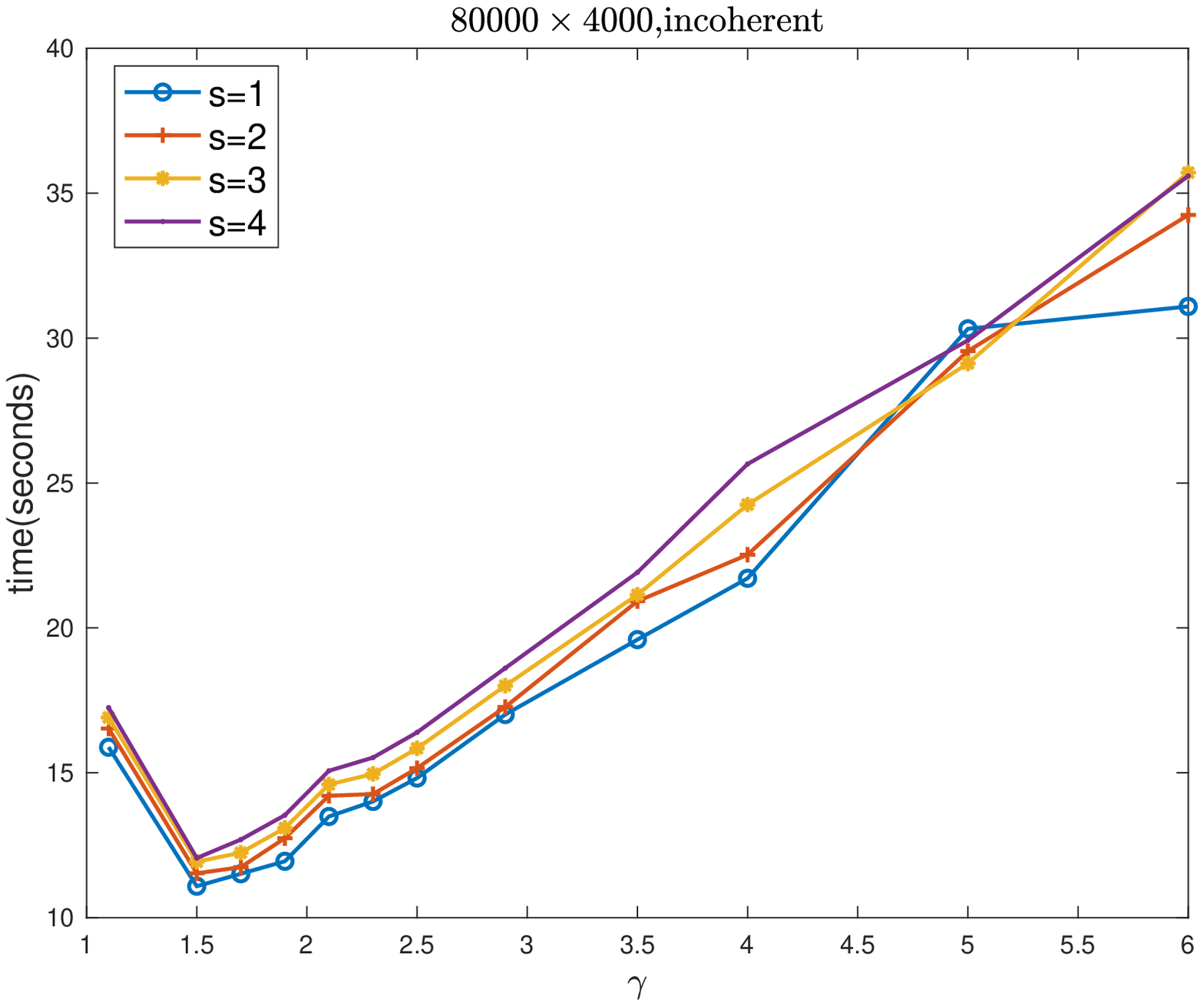}
{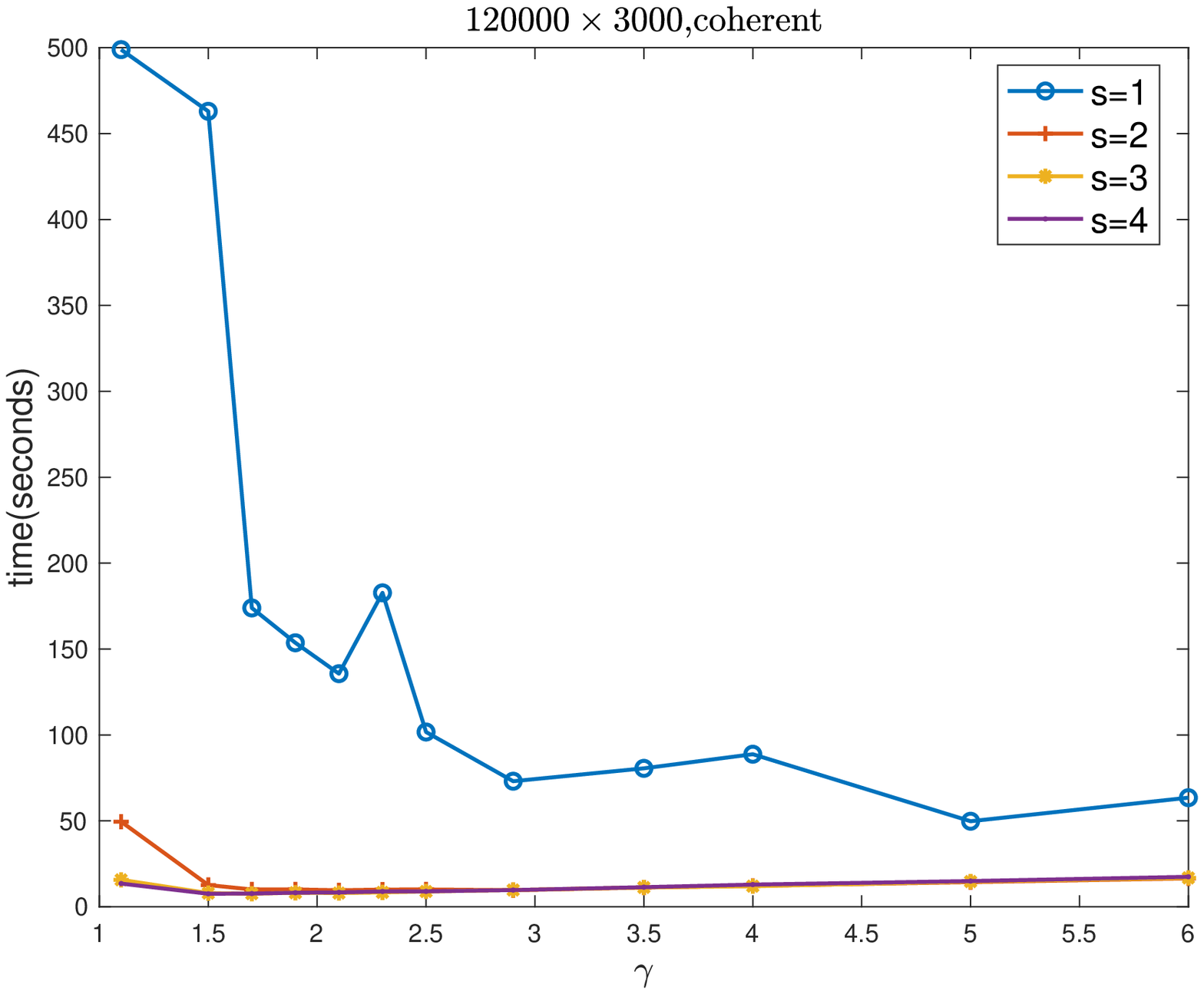}
{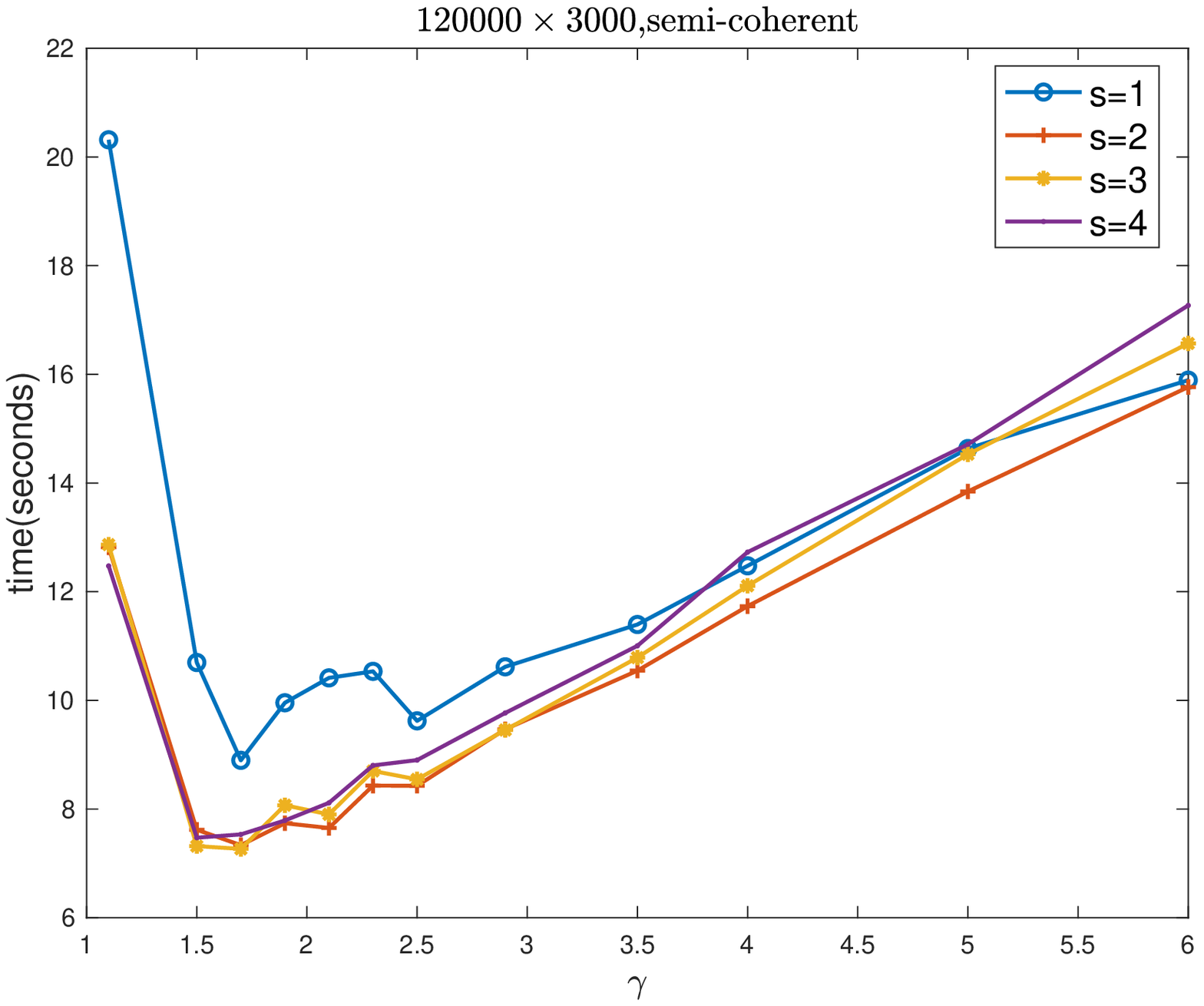}
{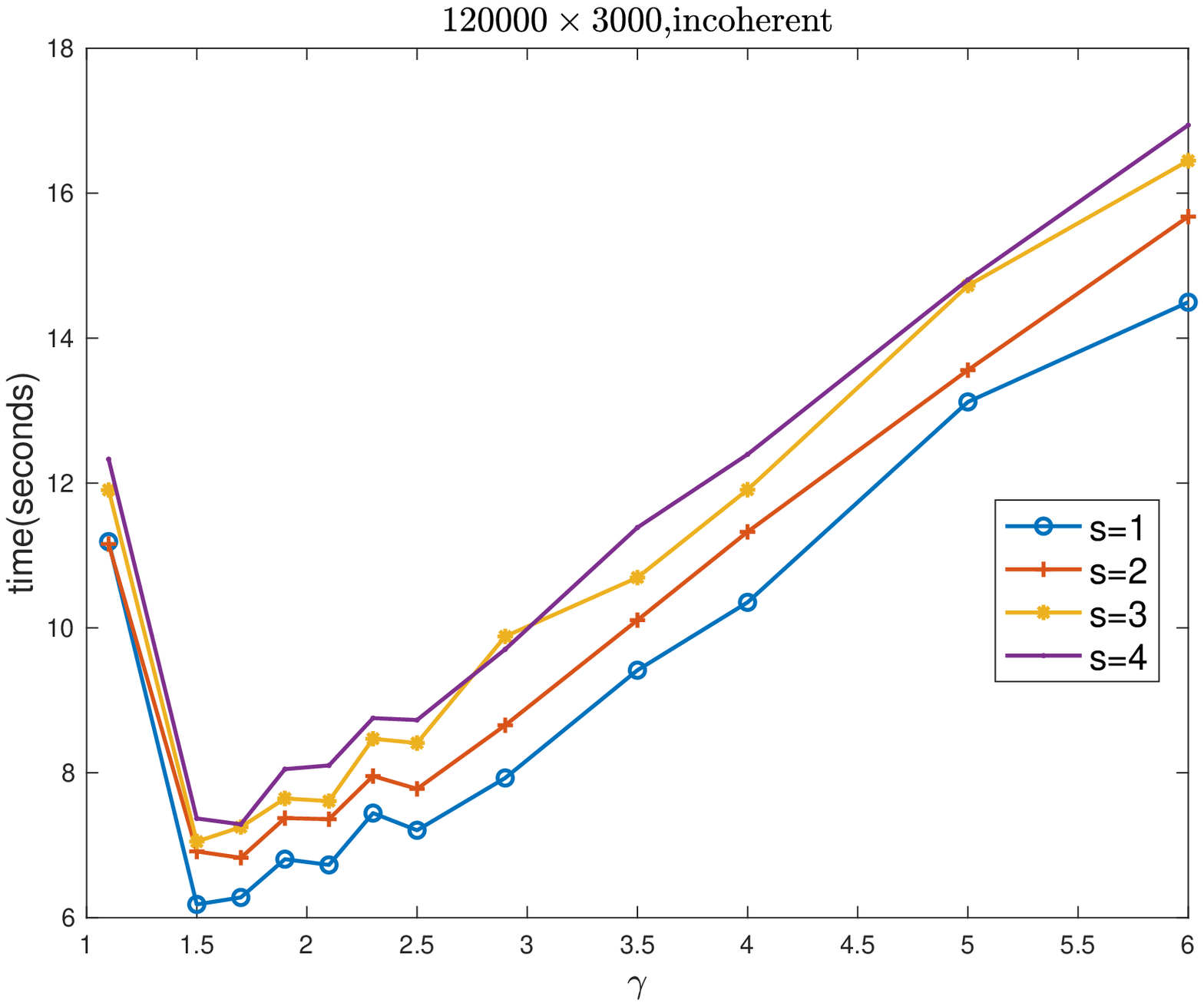}
{Runtime for \solverNameSparse{}  on  sparse $A \in \R^{n\times d}$ from Test Set 2 with $n=80000, d=4000$ and $n=120000, d=3000$ and different values of $\gamma=m/d$ 
 using different values of $s$ and $\gamma = m/d$. Using these plots and some experiments with Test Set 3 (Florida collection), we set $m = 1.4d$ and $s=2$ in \solverNameSparse{}.}
{fig::Ls_qr_engineering_time}

\calibrationSixFigures{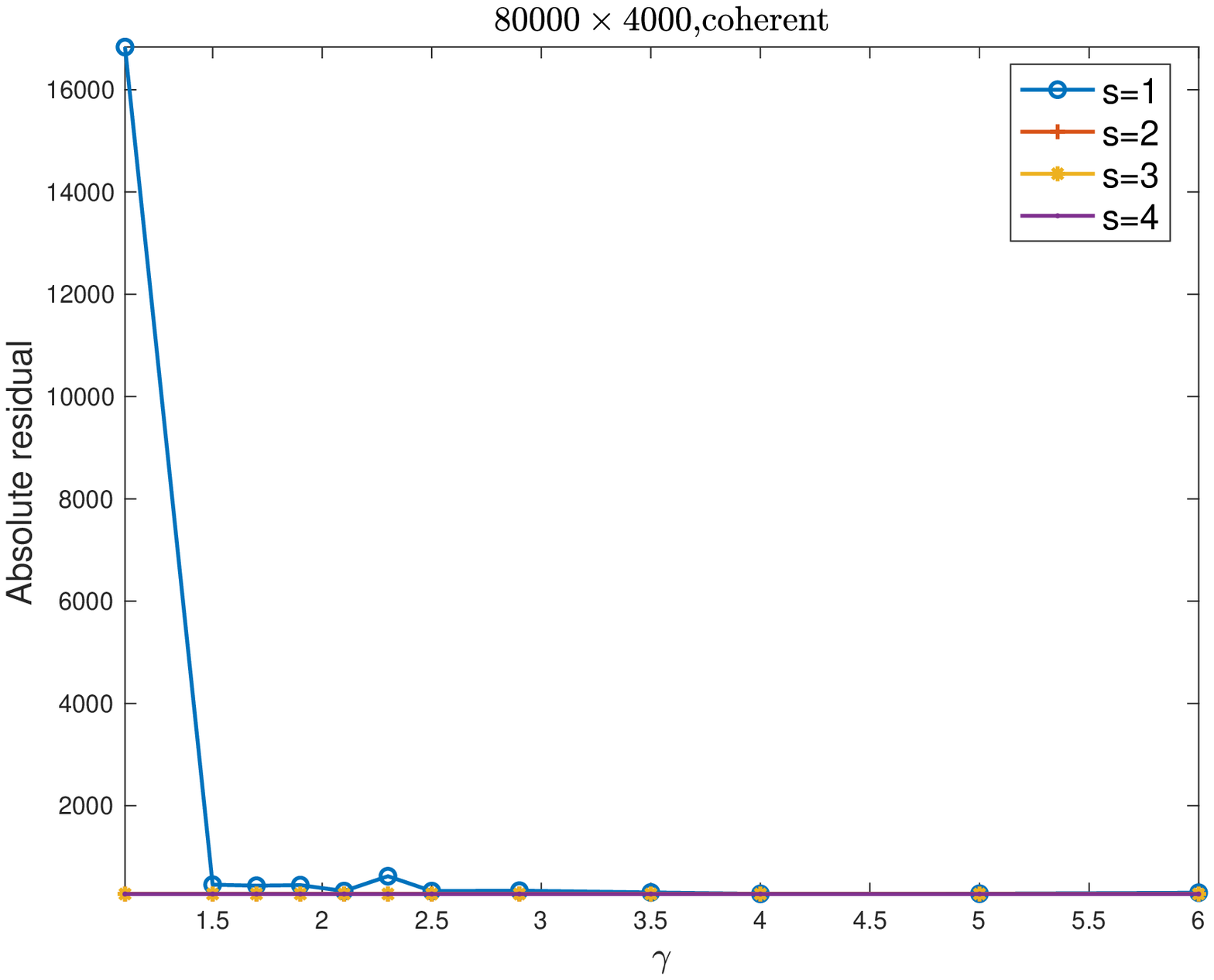}
{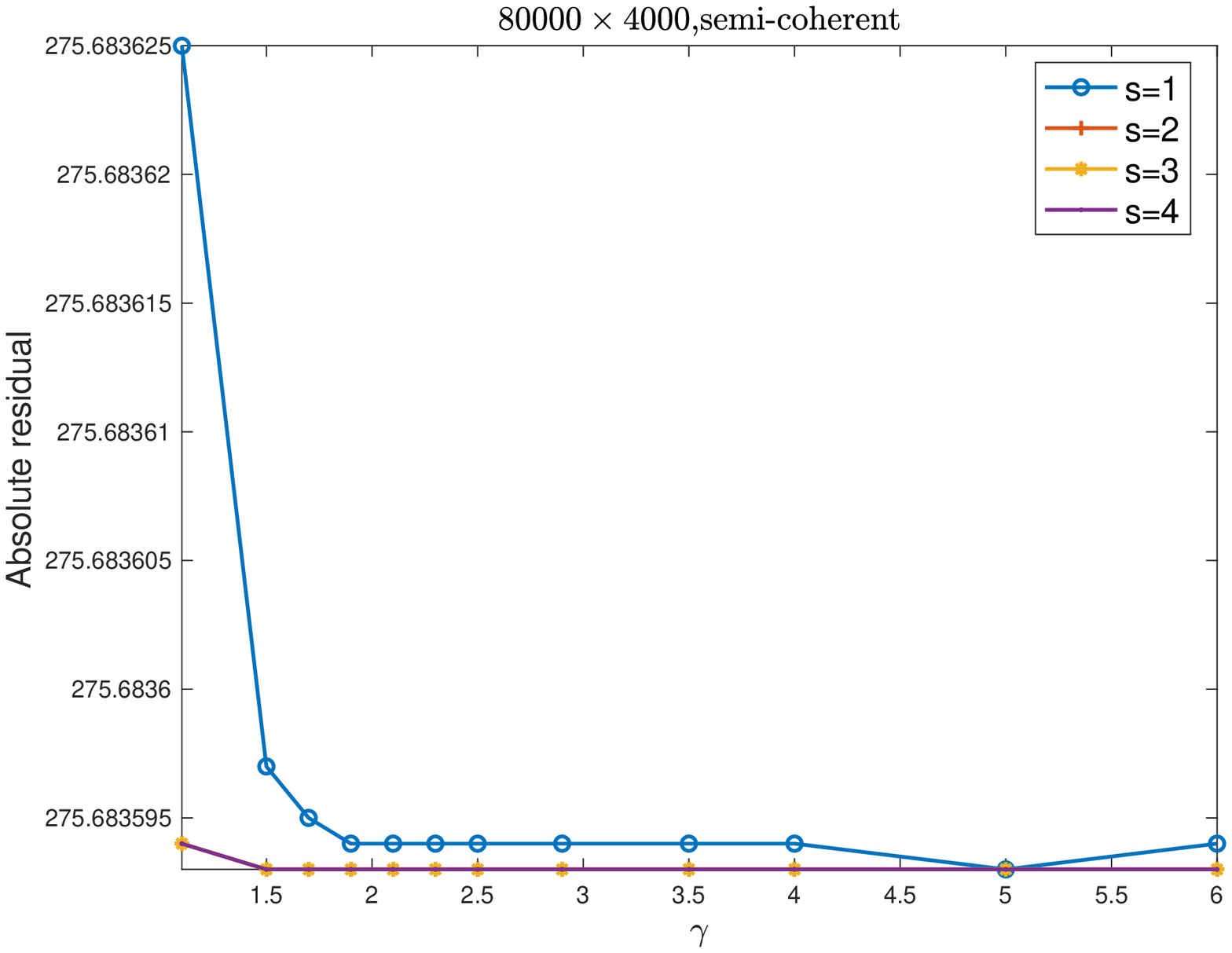}
{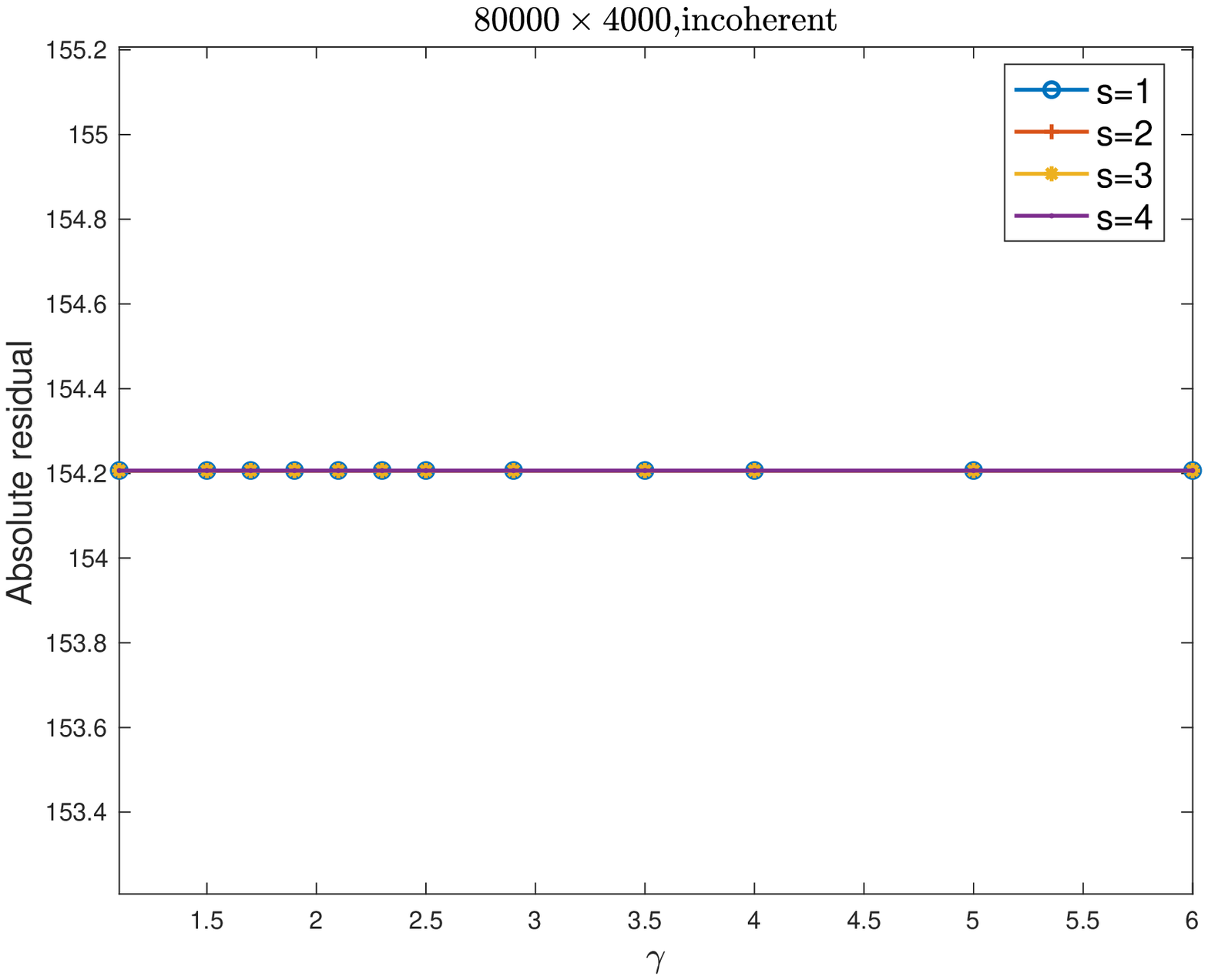}
{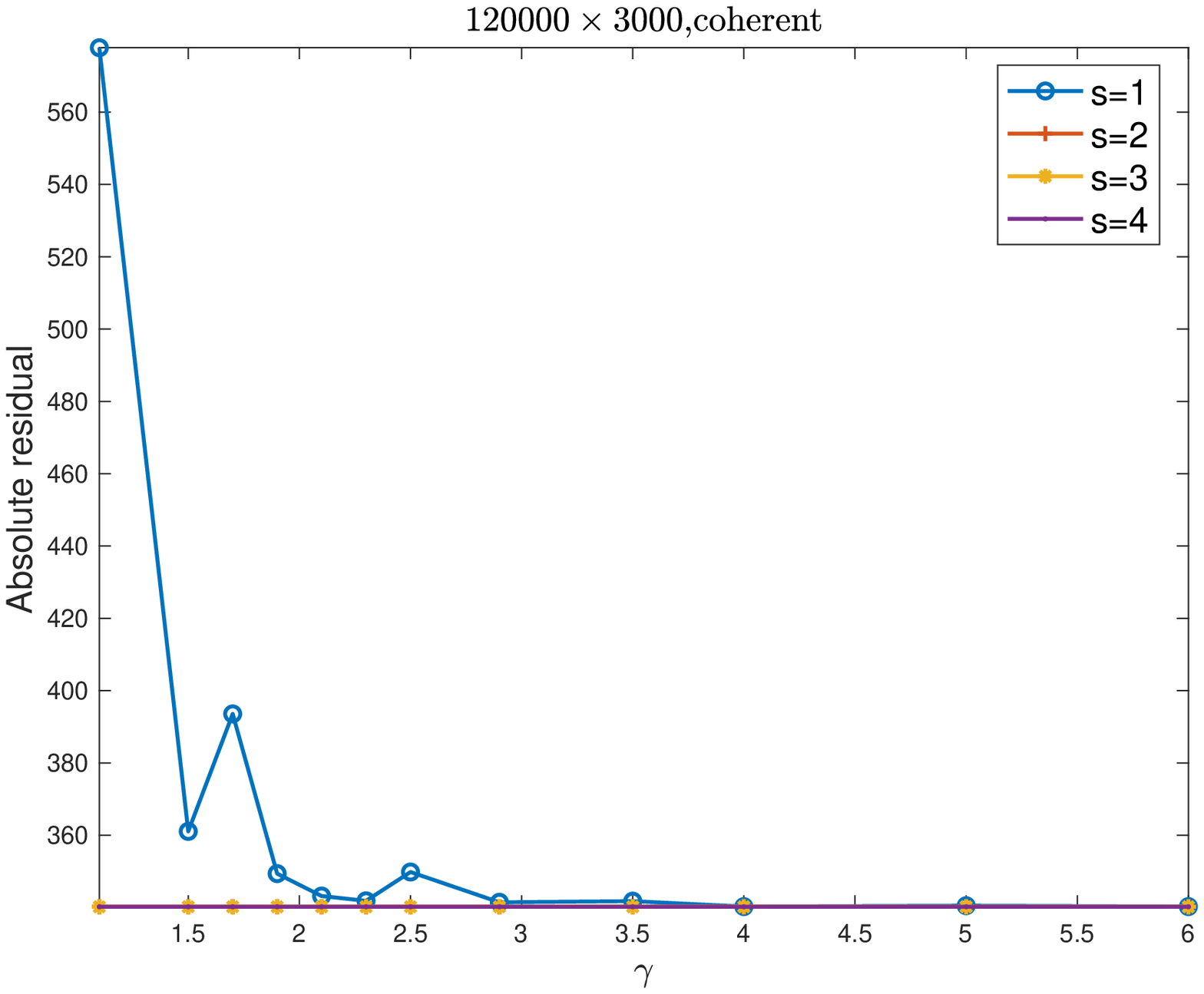}
{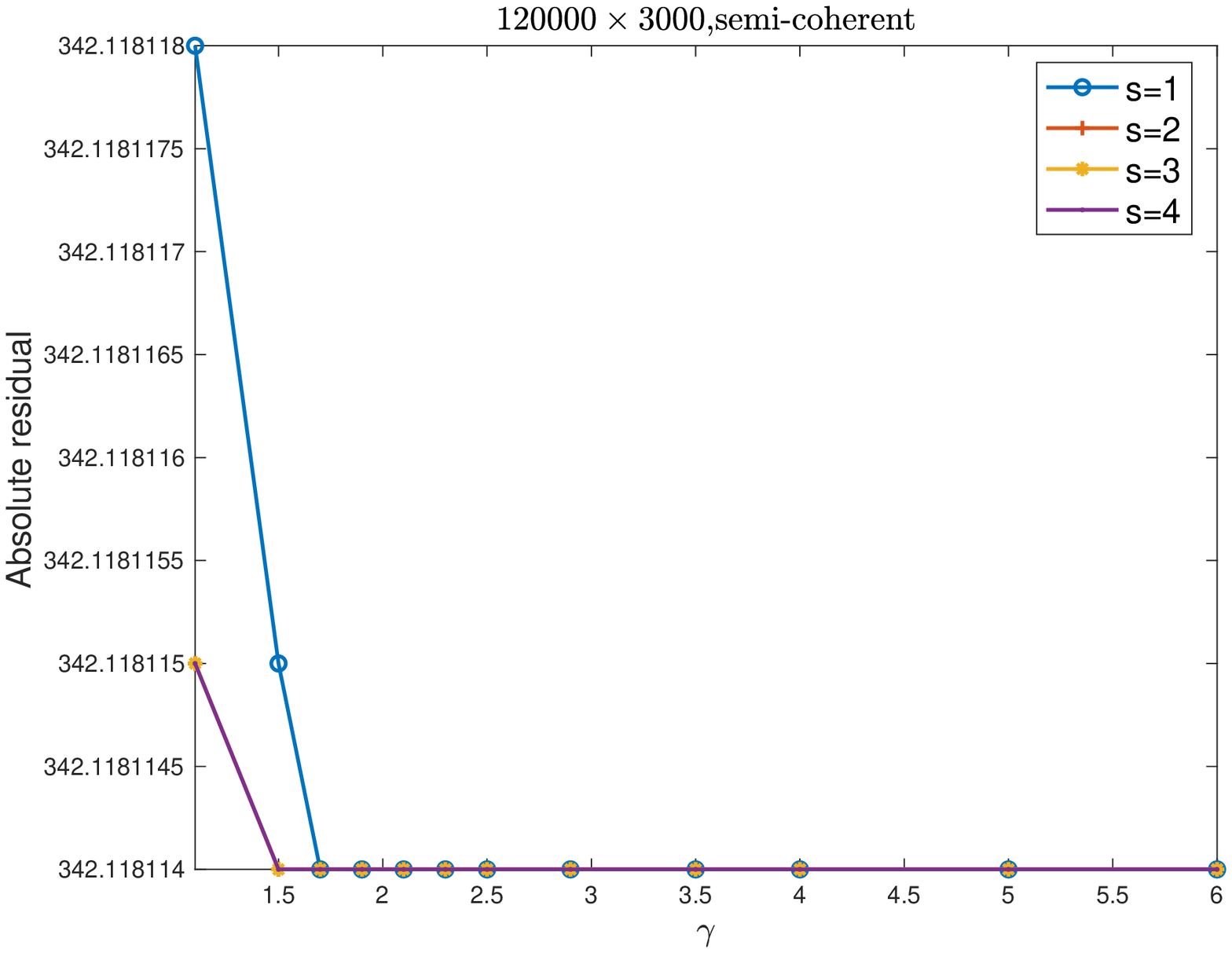}
{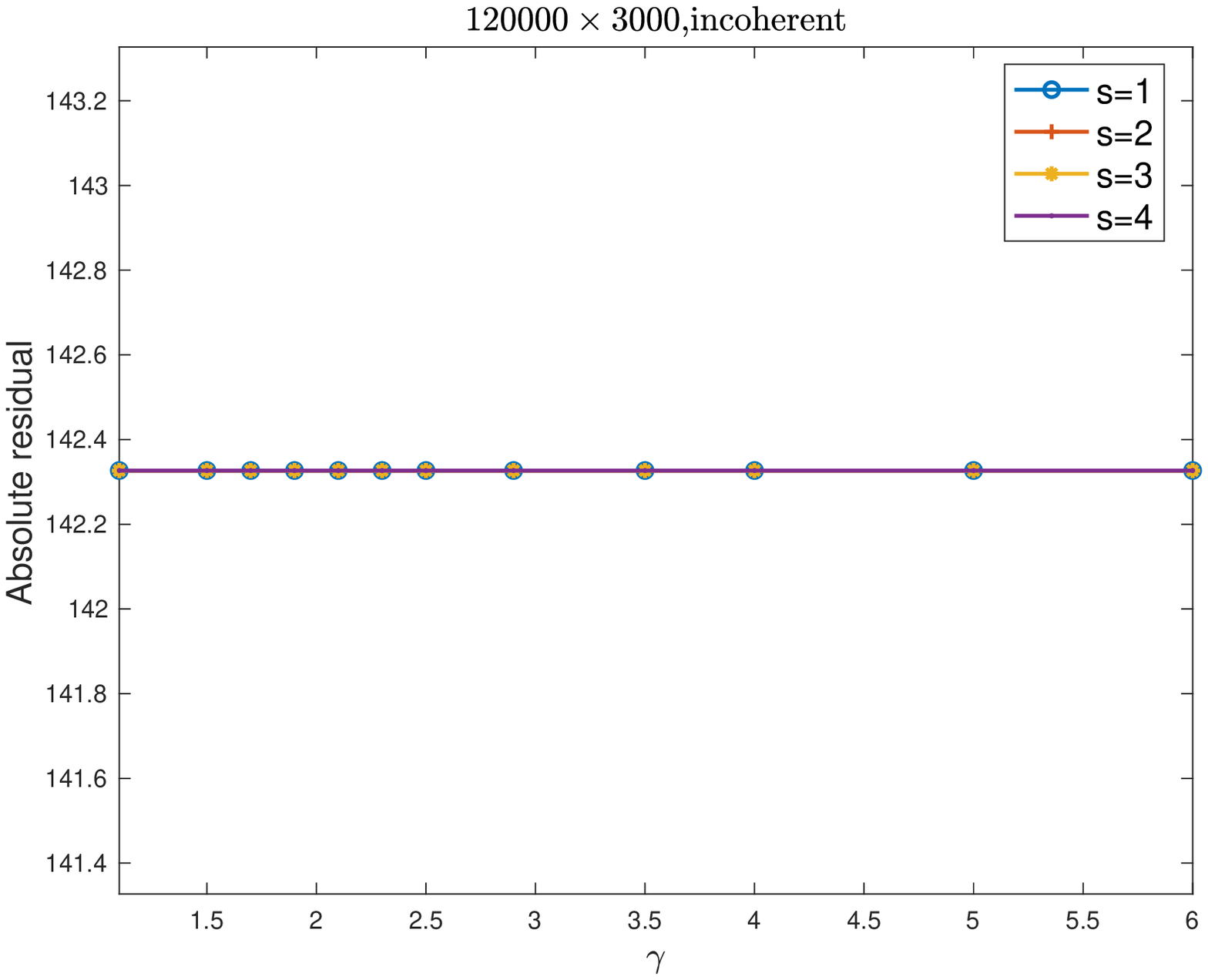}
{Residual values obtained by \solverNameSparse{} on
the same sparse problems as in Figure \ref{fig::Ls_qr_engineering_time}. Note that using $1$-hashing ($s=1$) results in inaccurate solutions. } 
{fig::Ls_qr_engineering_residual}

\calibrationSixFigures{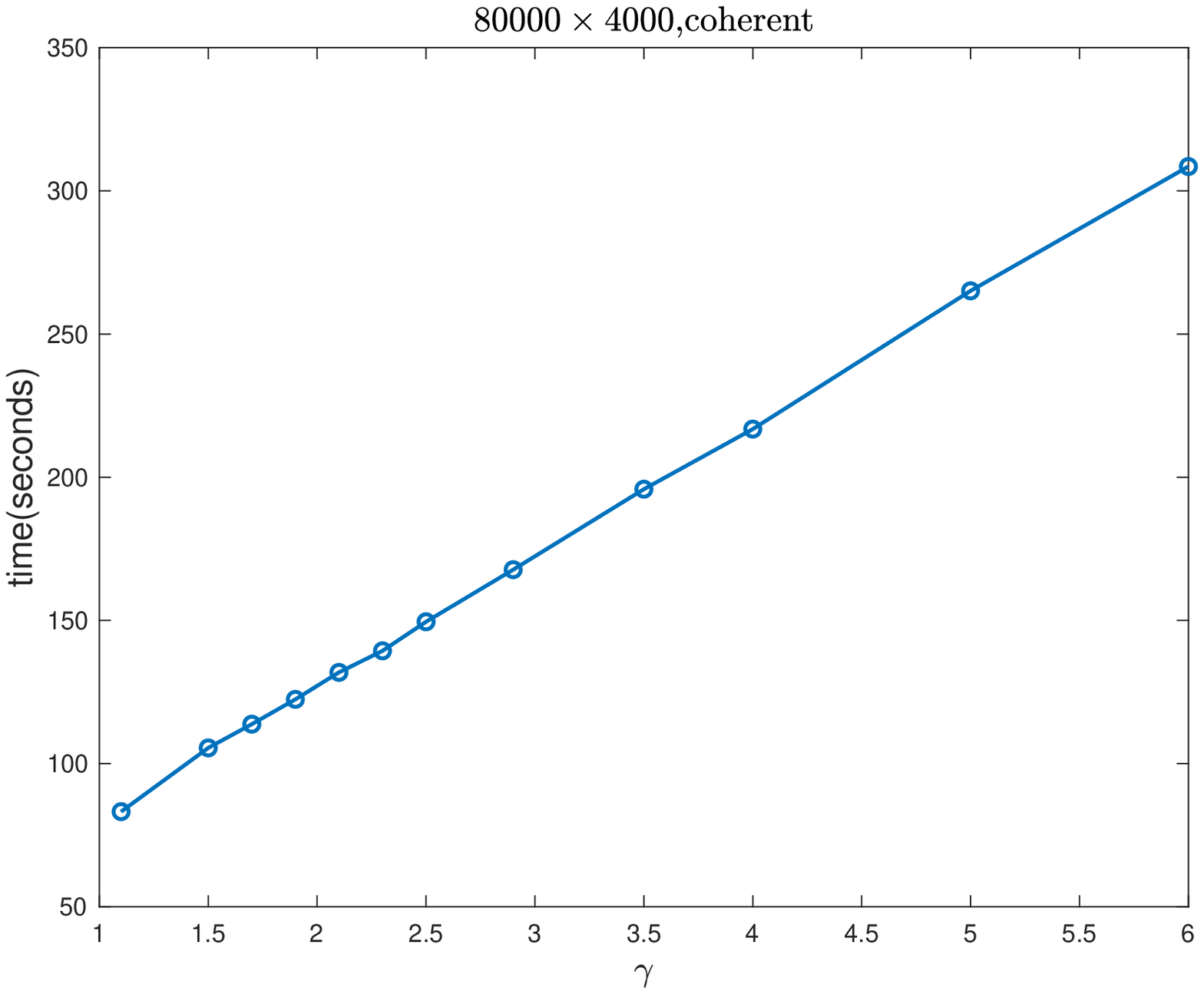}
{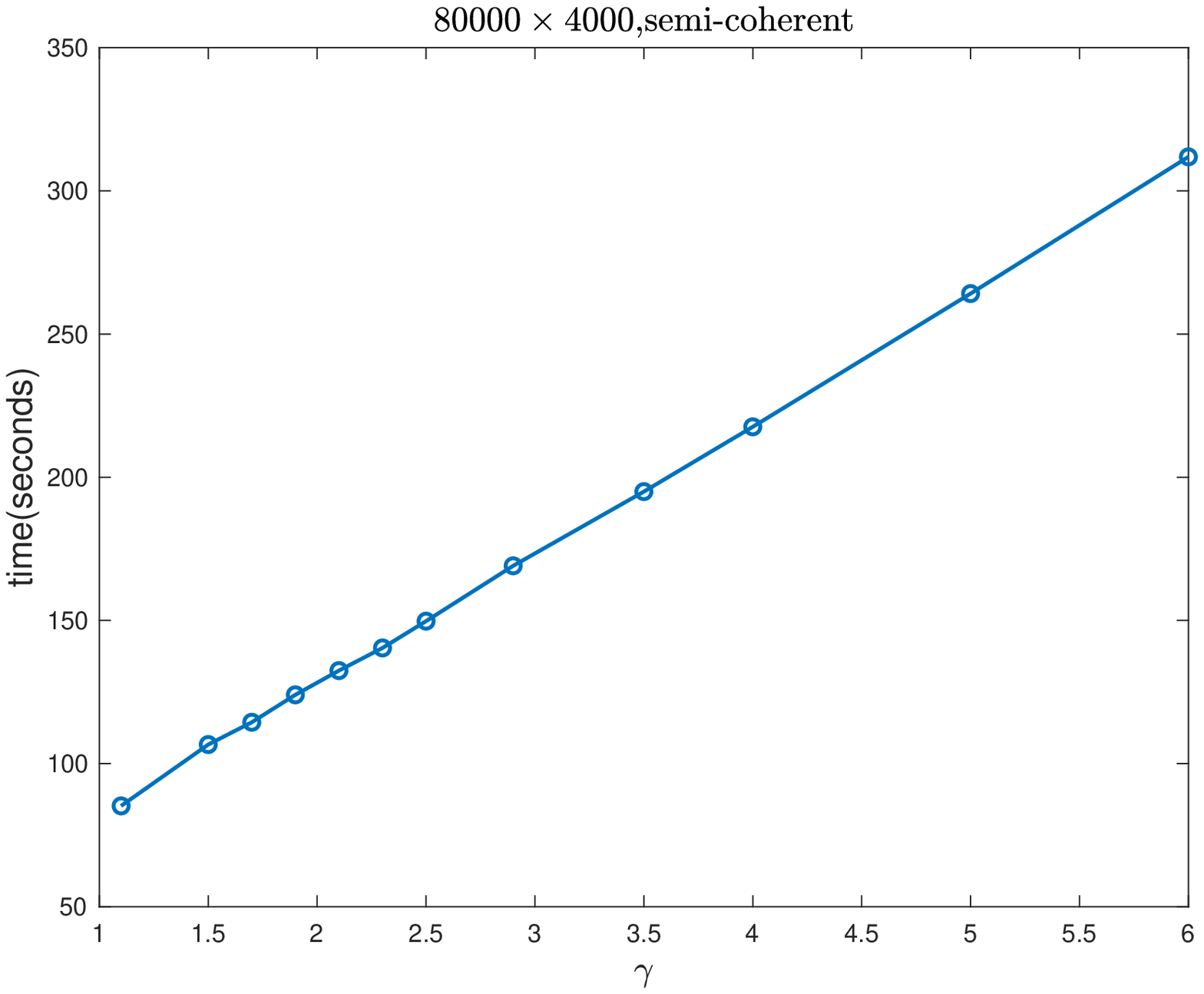}
{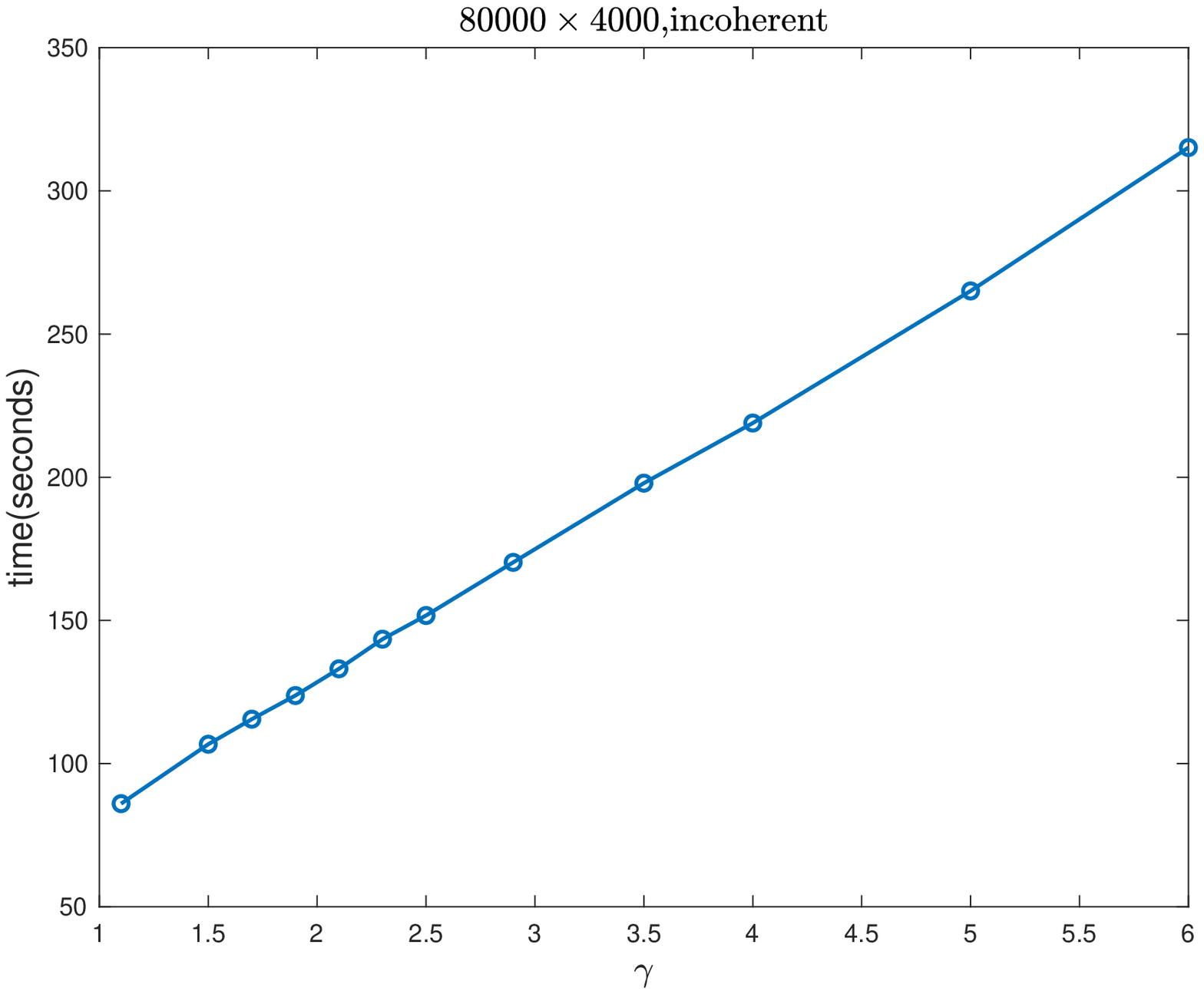}
{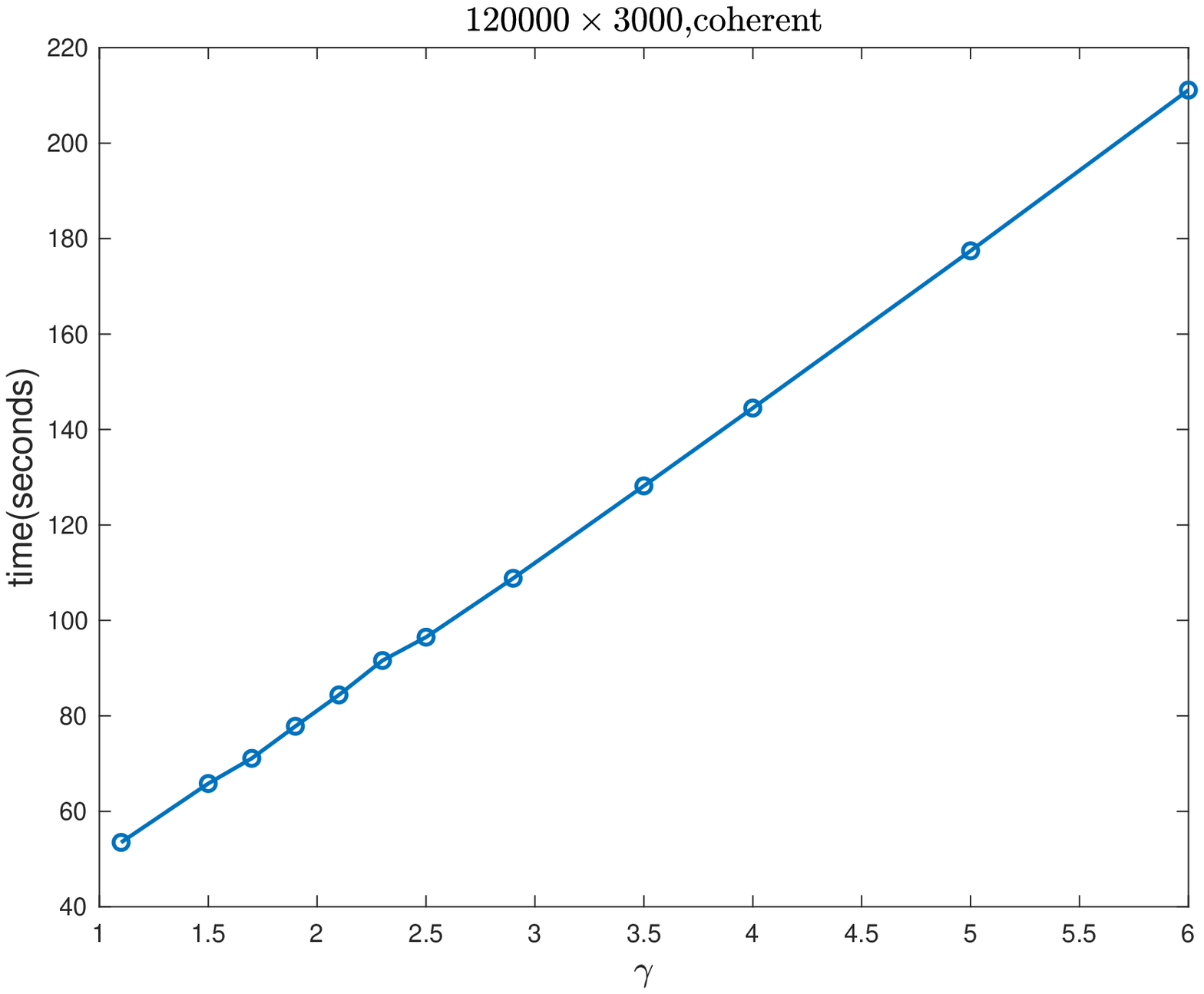}
{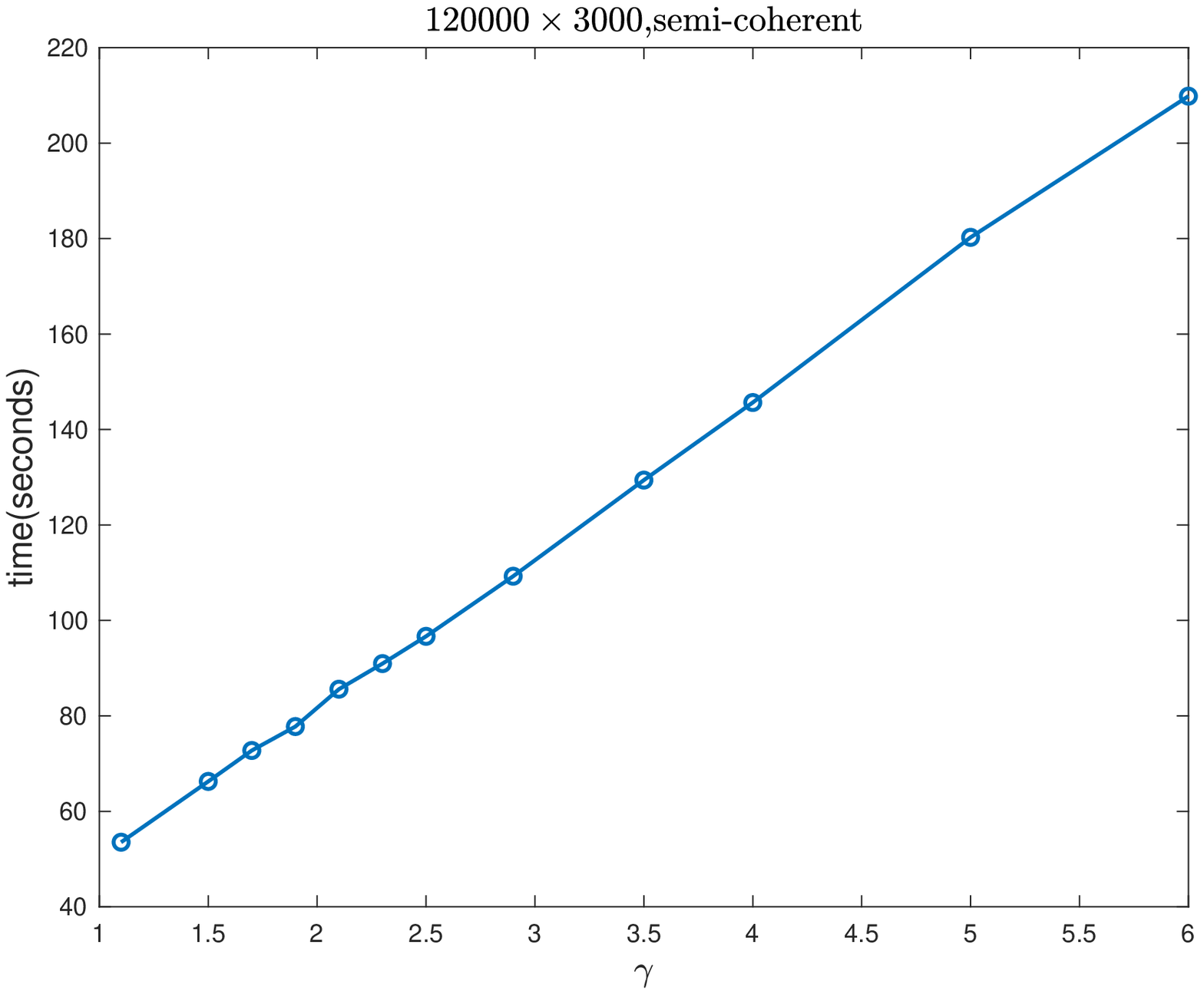}
{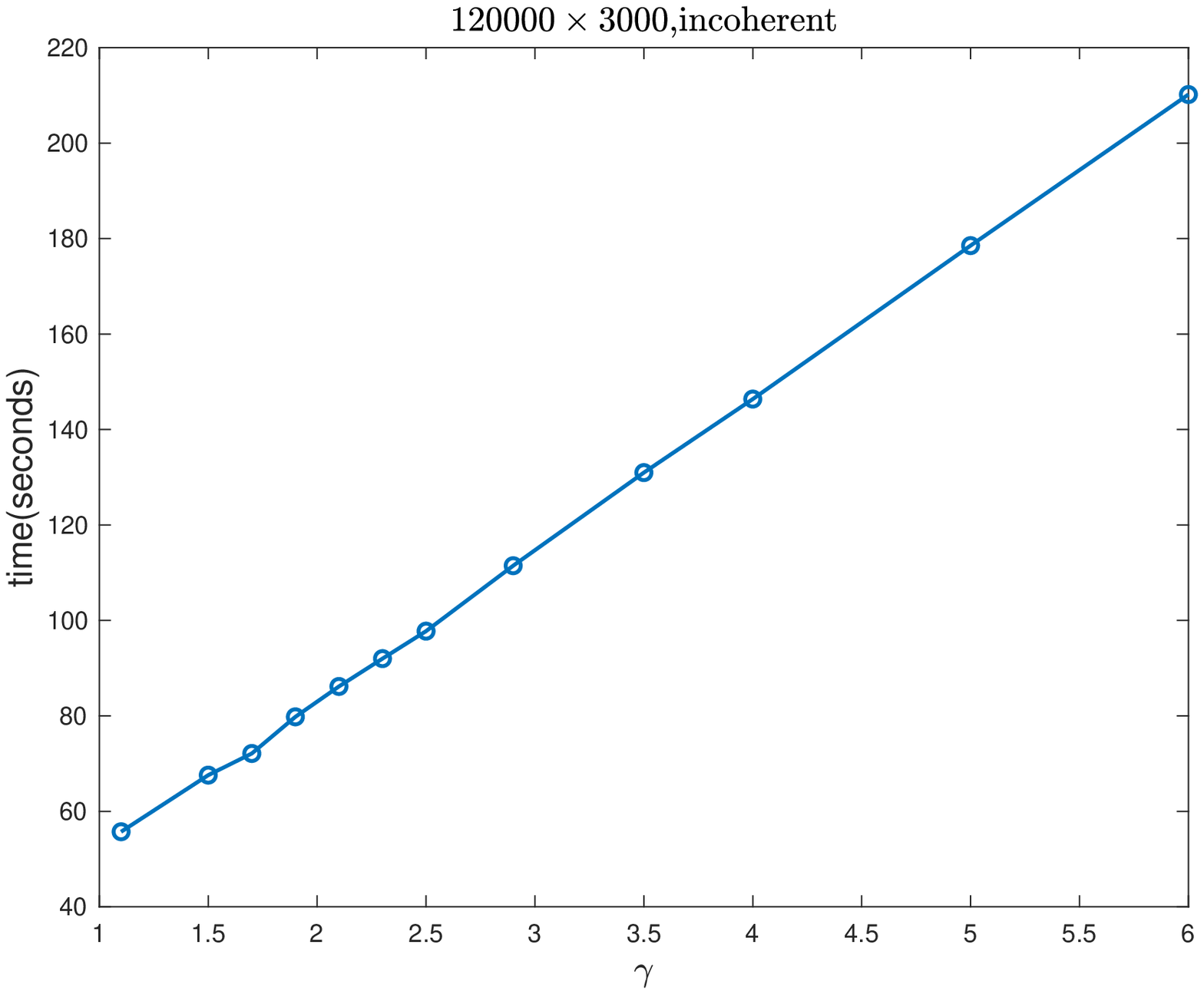}
{Runtime for LSRN  on  sparse $A \in \R^{n\times d}$ from Test Set 2 with $n=80000, d=4000$ and $n=120000, d=3000$ and different values of $\gamma=m/d$. The plots indicate that a value of  $m = 1.1d$
 is reasonable.}
{fig::Ls_lsrn_engineering_time}

\calibrationSixFigures{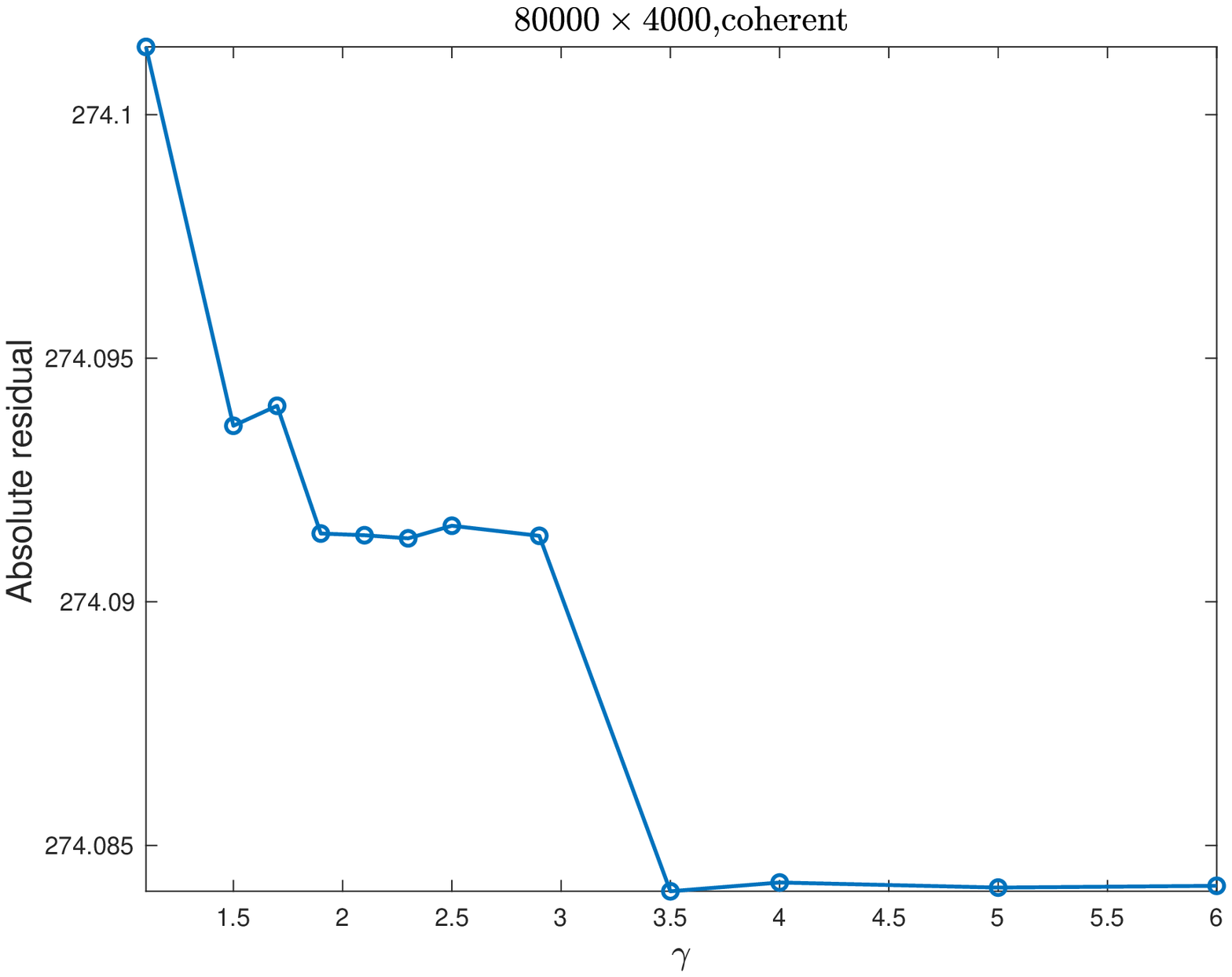}
{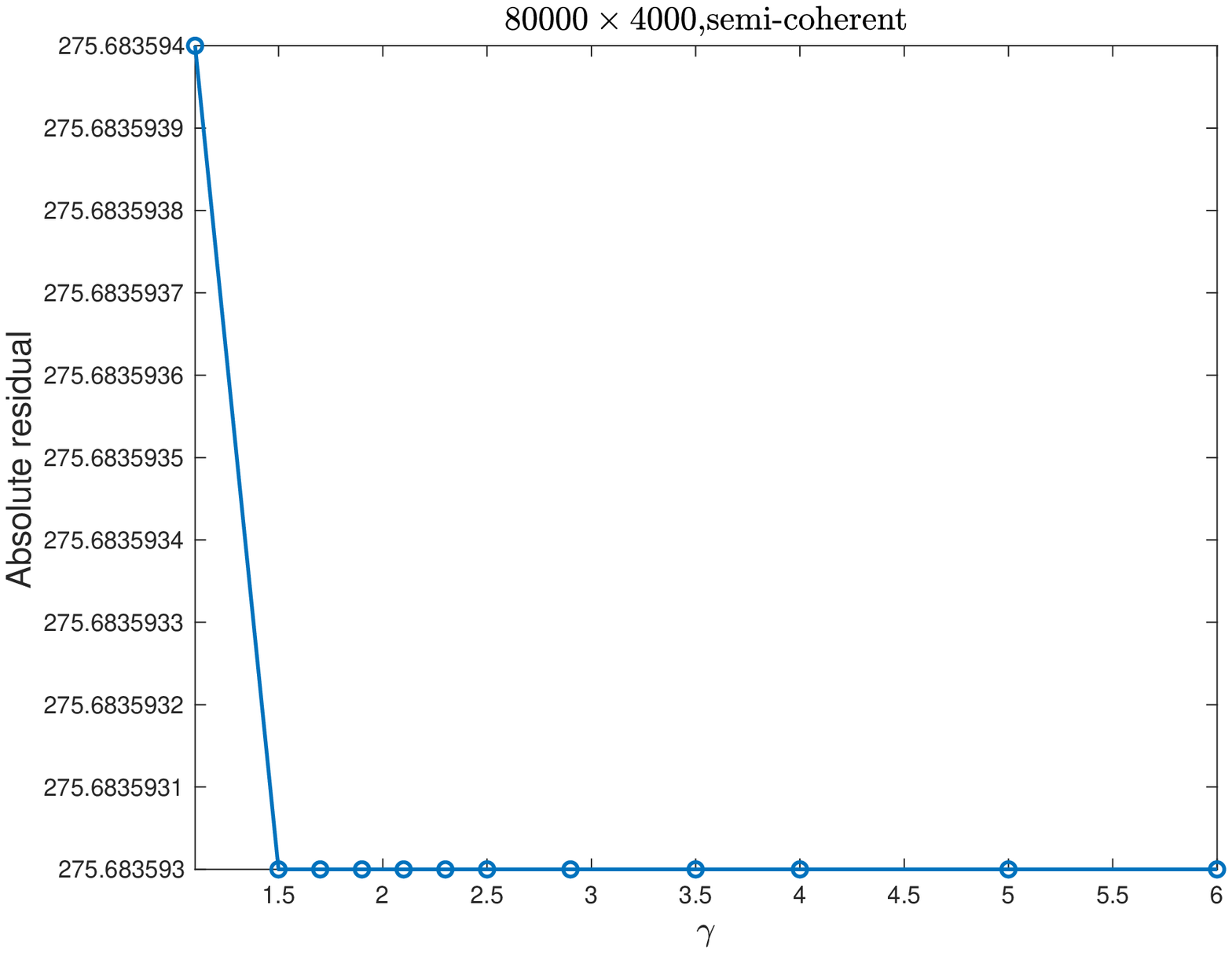}
{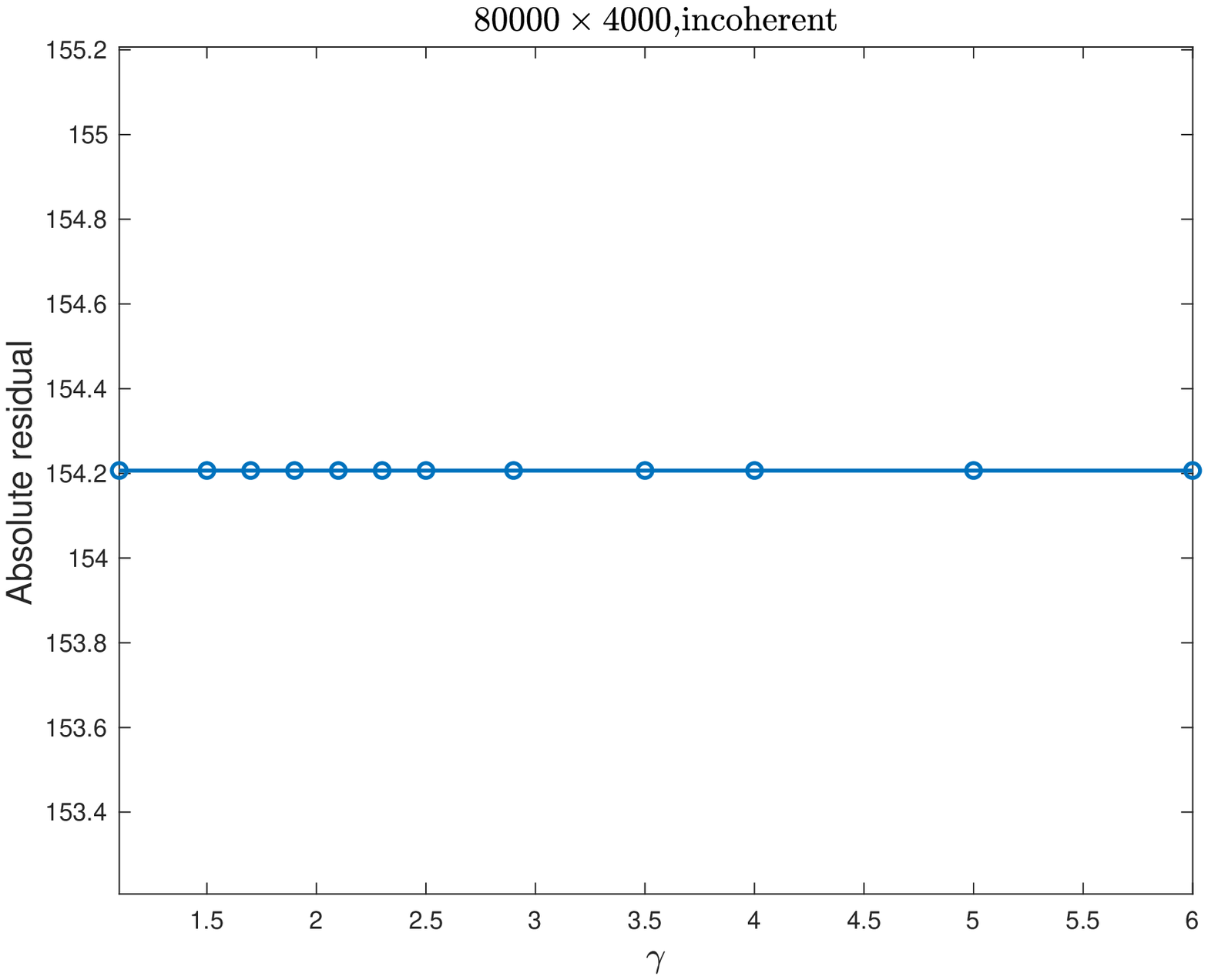}
{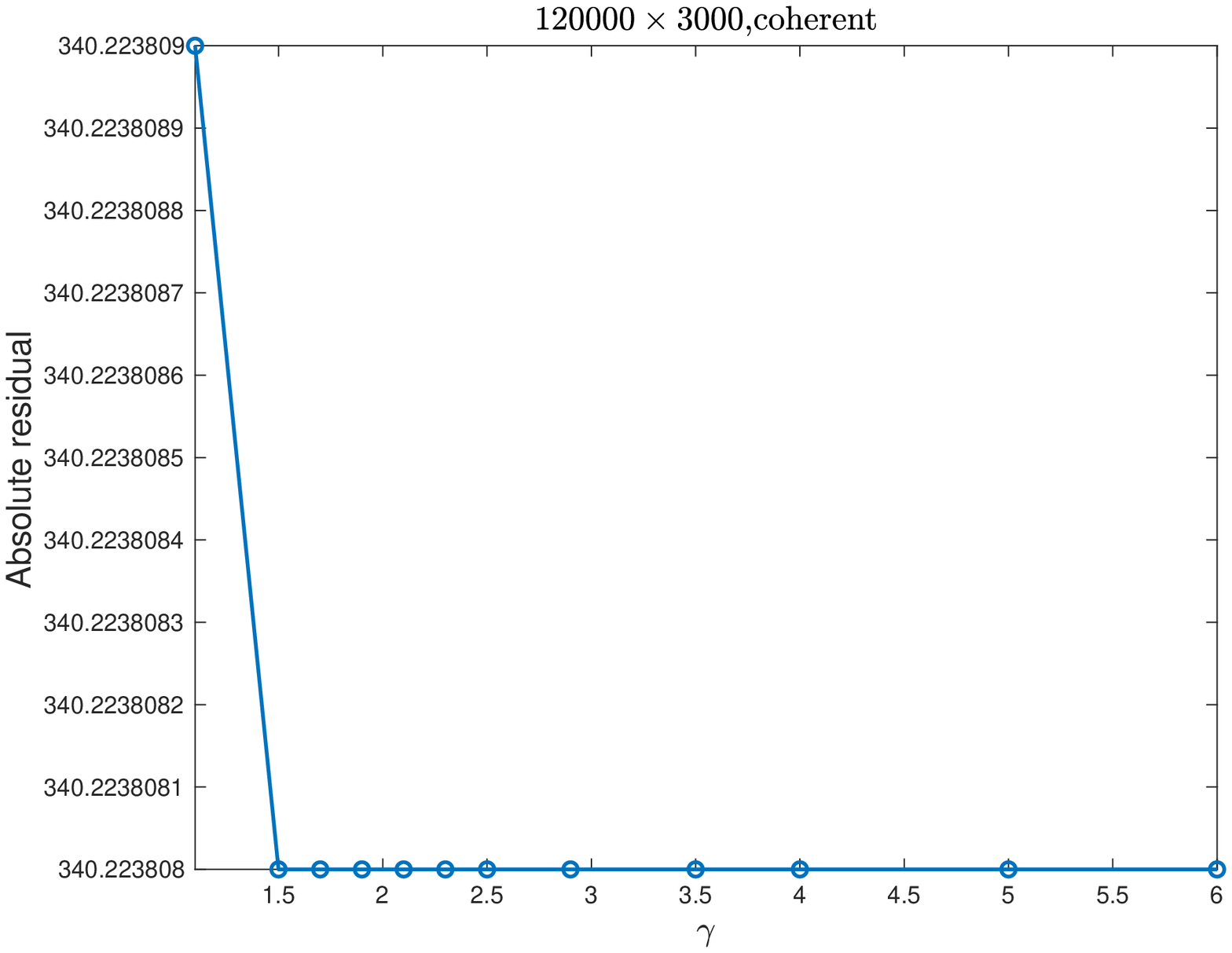}
{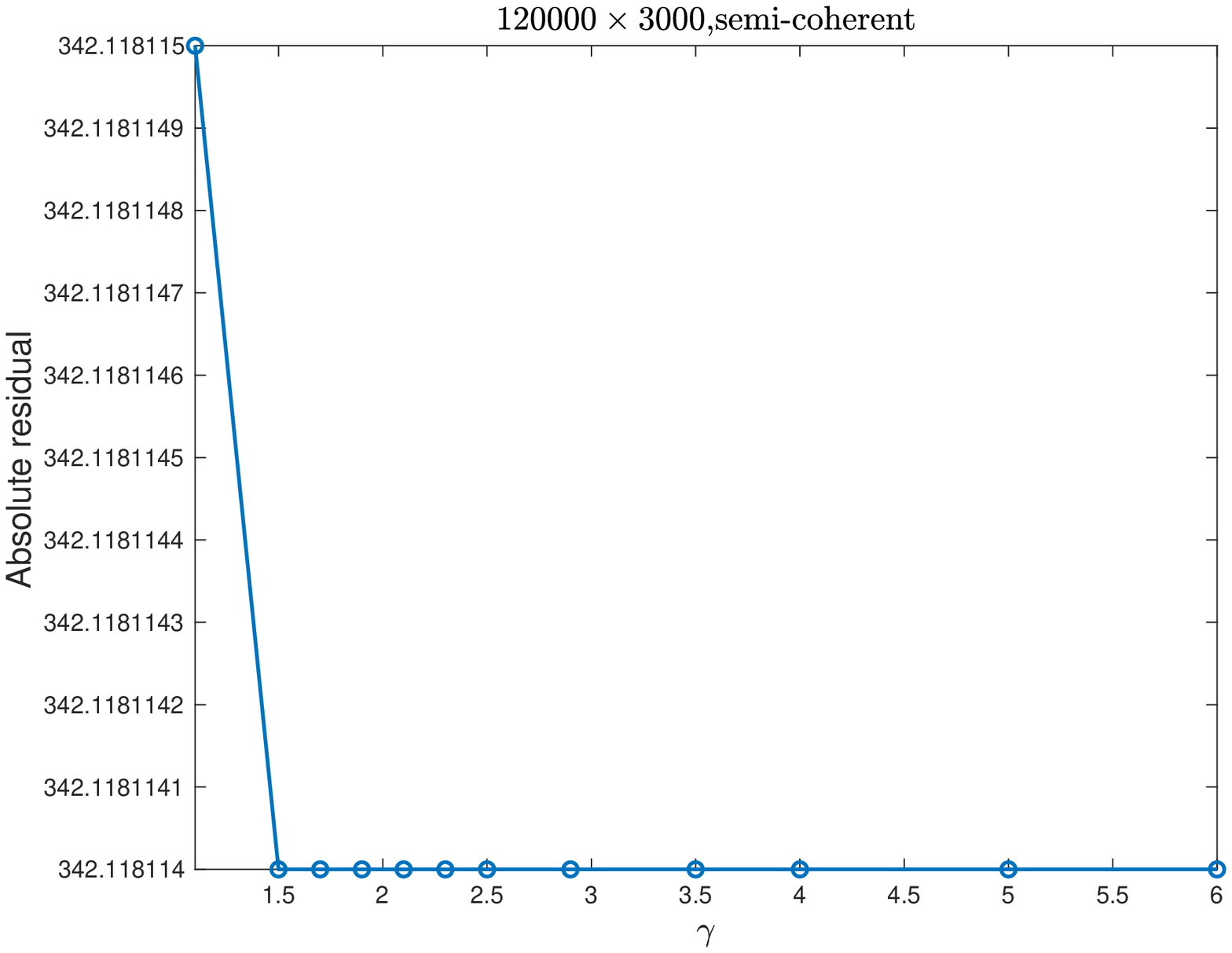}
{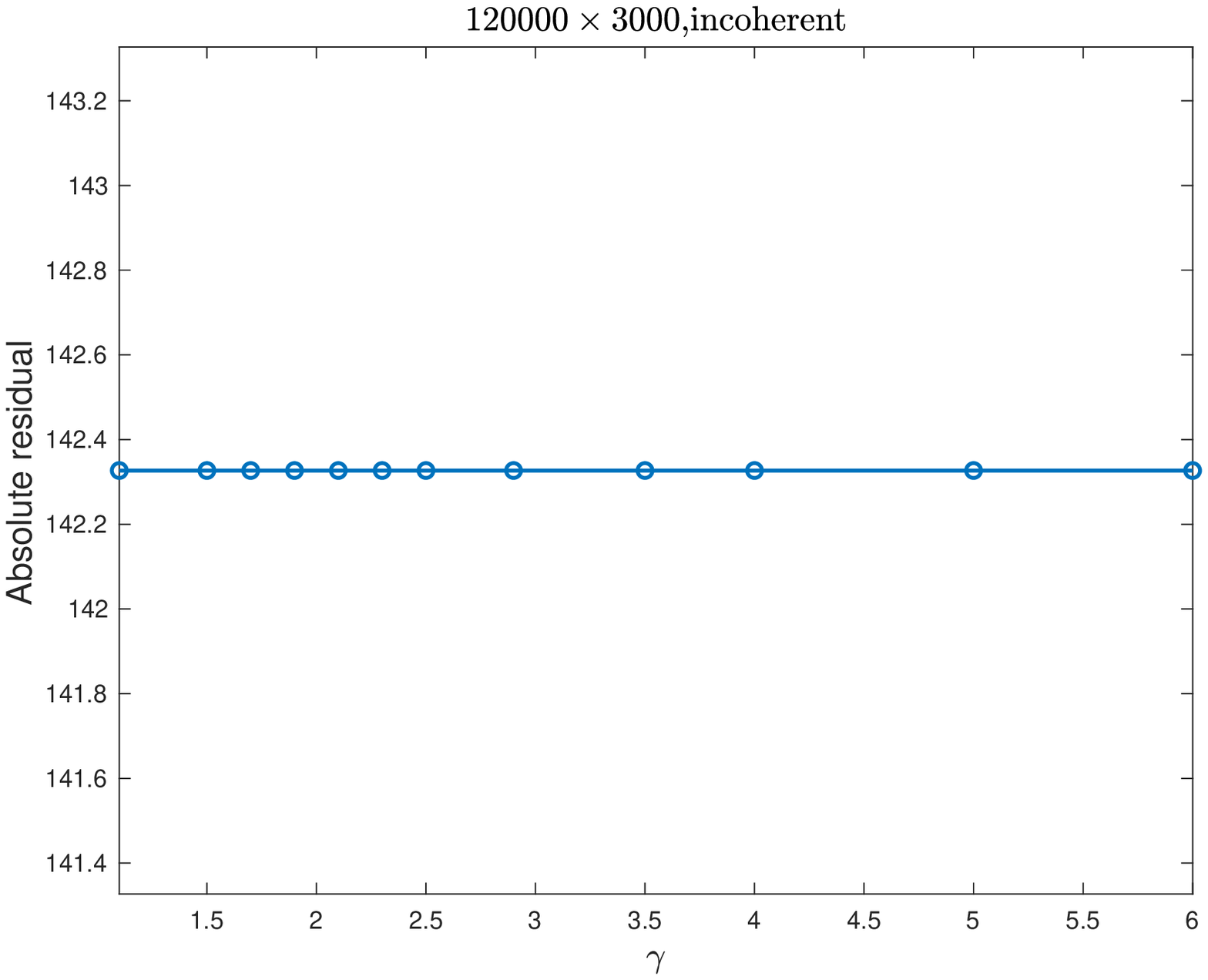}
{Residual values obtained by LSRN on
the same sparse problems as in Figure \ref{fig::Ls_lsrn_engineering_time}. Note that the reduced accuracy in one of the plot is likely due to the relative tolerance of the LSQR being set at $10^{-6}$.} 
{fig::Ls_lsrn_engineering_residual}

\section{Residual accuracy results for randomly generated problems}

See \autoref{tab::res_sparse_rand_inco}, \autoref{tab::res_sparse_rand_semi_co}, \autoref{tab::res_sparse_rand_co}.

\begin{table}
\begin{tabular}{|l|l|l|l|l|}
\hline
                     & \solverName{} & LS\_SPQR & LS\_HSL & LSRN    \\ \hline
$40000 \times 2000$  & 100.339       & 100.339  & 100.339 & 100.339 \\ \hline
$80000 \times 2000$  & 110.353       & 110.339  & 110.496 & 110.353 \\ \hline
$80000 \times 4000$  & 154.207       & 154.207  & 154.247 & 154.207 \\ \hline
$120000 \times 3000$ & 142.327       & 142.327  & 142.461 & 142.327 \\ \hline
$120000\times 5000$  & 182.171       & 182.171  & 182.208 & 182.171 \\ \hline
\end{tabular}
\caption{Residual values obtained by solvers on problems in Figure \ref{fig::sparse_rand_inco}.}
\label{tab::res_sparse_rand_inco}
\end{table}

\begin{table}
\begin{tabular}{|l|l|l|l|l|}
\hline
                     & \solverName{} & LS\_SPQR & LS\_HSL & LSRN    \\ \hline
$40000 \times 2000$  & 194.714       & 194.714  & 194.718 & 194.714 \\ \hline
$80000 \times 2000$  & 279.352       & 279.352  & 279.359 & 279.352 \\ \hline
$80000 \times 4000$  & 275.684       & 275.684  & 275.685 & 275.684 \\ \hline
$120000 \times 3000$ & 342.118       & 342.118  & 342.120 & 342.118 \\ \hline
$120000\times 5000$  & 339.129       & 339.129  & 339.141 & 339.129 \\ \hline
\end{tabular}
\caption{Residual values obtained by solvers on problems in Figure \ref{fig::sparse_rand_semi_co}.}
\label{tab::res_sparse_rand_semi_co}
\end{table}

\begin{table}
\begin{tabular}{|l|l|l|l|l|}
\hline
                     & \solverName{} & LS\_SPQR & LS\_HSL & LSRN    \\ \hline
$40000 \times 2000$  & 193.851       & 194.158  & 195.586 & 193.802 \\ \hline
$80000 \times 2000$  & 277.746       & 277.750  & 278.561 & 277.746 \\ \hline
$80000 \times 4000$  & 274.117       & 274.117  & 279.763 & 274.102 \\ \hline
$120000 \times 3000$ & 340.224       & 340.236  & 343.114 & 340.224 \\ \hline
$120000\times 5000$  & 337.100       & 337.144  & 342.291 & 337.085 \\ \hline
\end{tabular}
\caption{Residual values obtained by solvers on problems in Figure \ref{fig::sparse_rand_co}.}
\label{tab::res_sparse_rand_co}
\end{table}

\section{Additional Performance Profiles for the Florida Matrix Collection}

See Figures \ref{fig::all_solver_2} and 
\ref{fig::density0001}.

\renewcommand{\mysize}{0.42}
\twoFigures{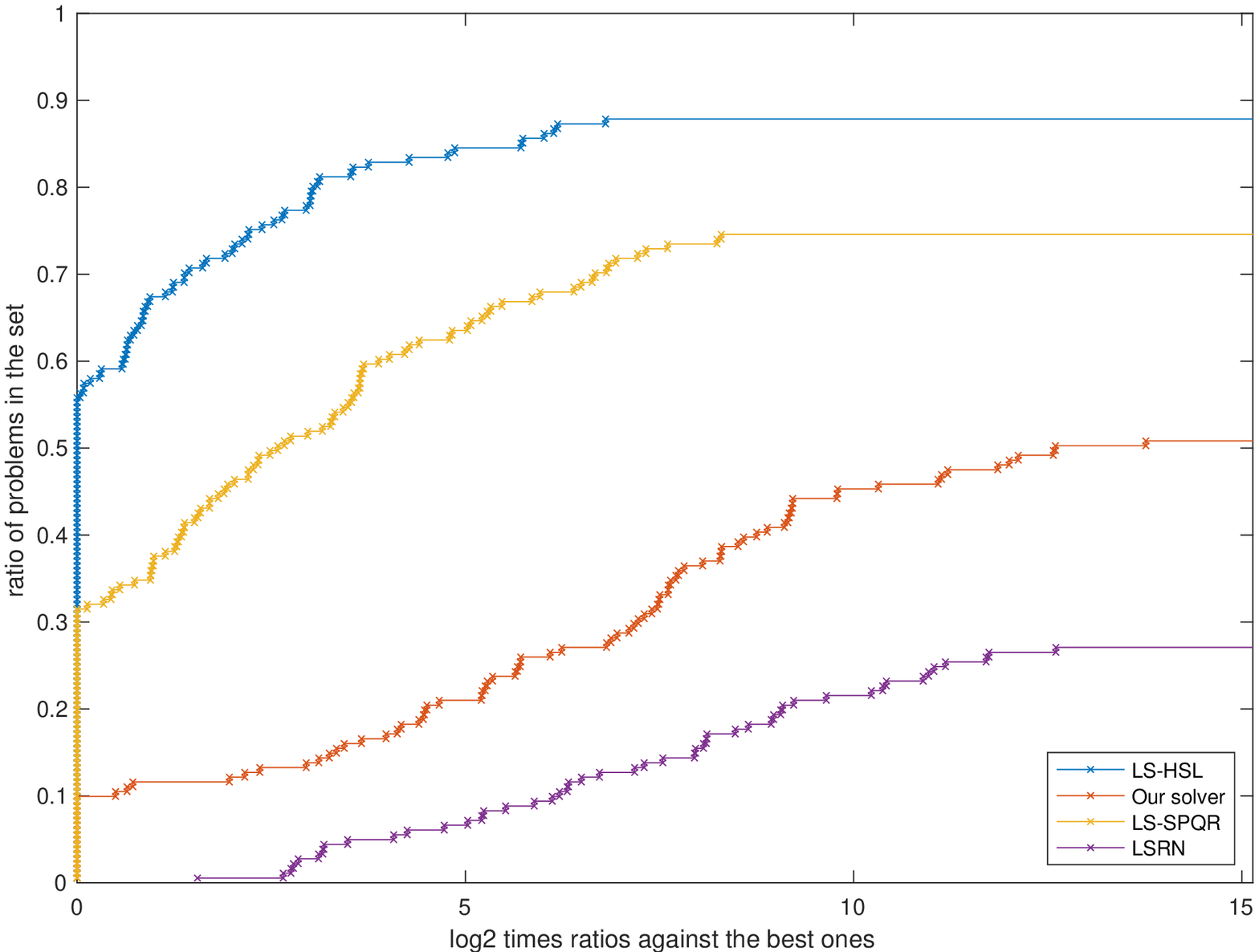}
{\performanceProfileCaption{$n\geq 2d$}.}
{fig::all_solver_2}
{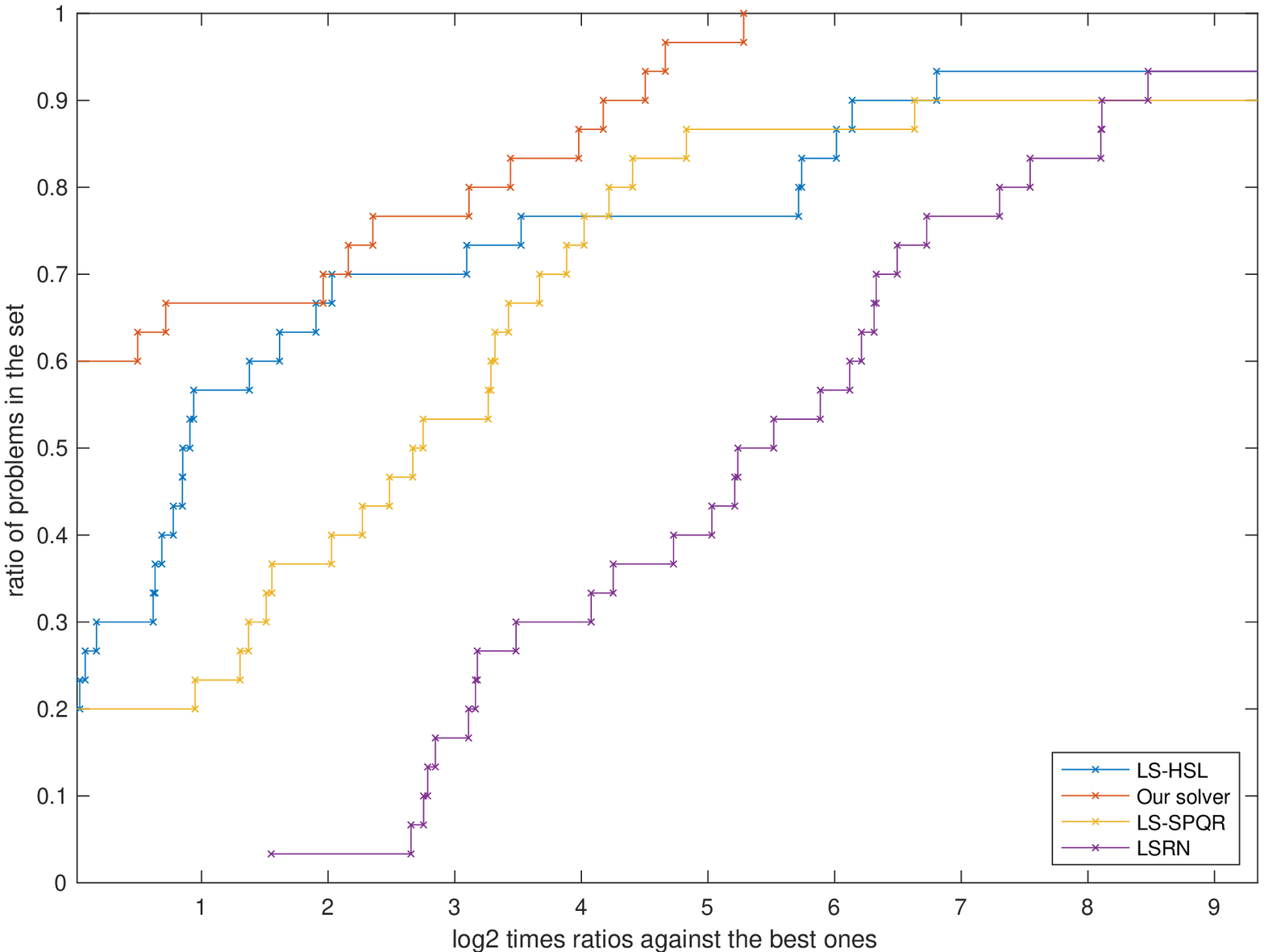}
{\performanceProfileCaption{$n \geq 10d$ and $\text{nnz}(A) \geq 0.001nd$}.}
{fig::density0001}

\section{Additional information on the Florida matrix collection}
\autoref{tab::rank_def_accuracy_dim} gives dimensions for the Florida collection problems tested in \autoref{tab::rank_def_accuracy}. 
\begin{table}
\centering
\scriptsize
\begin{tabular}{l|rrr}
\multicolumn{1}{r|}{} & nrow~  & ncol & rank            \\ 
\hline
lp\_ship12l           & 5533   & 1151 & 1042            \\
Franz1                & 2240   & 768  & 755             \\
GL7d26                & 2798   & 305  & 273             \\
cis-n4c6-b2           & 1330   & 210  & 190             \\
lp\_modszk1           & 1620   & 687  & 686             \\
rel5                  & 240    & 35   & 24              \\
ch5-5-b1              & 200    & 25   & 24              \\
n3c5-b2               & 120    & 45   & 36              \\
ch4-4-b1              & 72     & 16   & 15              \\
n3c5-b1               & 45     & 10   & 9               \\
n3c4-b1               & 15     & 6    & 5               \\
connectus             & 394792 & 512  & \textless{}458  \\
landmark              & 71952  & 2704 & 2671            \\
cis-n4c6-b3           & 5940   & 1330 & 1140           
\end{tabular}
\caption{Dimensions for the Florida collection problems tested in \autoref{tab::rank_def_accuracy}}
\label{tab::rank_def_accuracy_dim}
\end{table}

\end{document}